\documentclass{amsart}
\usepackage{amsthm,amssymb,mathtools}
\usepackage[british]{babel}
\usepackage{enumitem}
\usepackage[latin1]{inputenc}
\usepackage{stmaryrd}
\usepackage{longtable}
\usepackage{bussproofs}
\usepackage{url}
\usepackage{hyperref}
\usepackage{blindtext}

\newtheorem{theorem}{Theorem}
\numberwithin{theorem}{section}
\newtheorem{corollary}[theorem]{Corollary}
\newtheorem{lemma}[theorem]{Lemma}
\newtheorem{proposition}[theorem]{Proposition}
{\theoremstyle{definition}
\newtheorem{definition}[theorem]{Definition}
\newtheorem{remark}[theorem]{Remark}

\newtheorem{assumption}[theorem]{Standing Assumption}}

\newcommand{\atrs}{\mathbf{ATR_0^{\operatorname{set}}}}

\newcommand{\rca}{\mathbf{RCA_0}}

\newcommand{\supp}{\operatorname{supp}}
\newcommand{\suppl}{\operatorname{supp}^{\mathbf L}}
\newcommand{\supps}{\operatorname{supp}^{\operatorname{S}}}
\newcommand{\suppe}{\operatorname{supp}^{\operatorname{\varepsilon(S)}}}
\newcommand{\dom}{\operatorname{dom}}
\newcommand{\rng}{\operatorname{rng}}
\newcommand{\lef}{<^{\operatorname{fin}}}
\newcommand{\leqf}{\leq^{\operatorname{fin}}}
\newcommand{\ordi}{\operatorname{Ord}}
\newcommand{\en}{\operatorname{en}}
\newcommand{\code}{\operatorname{code}}
\newcommand{\hth}{\operatorname{ht}}
\newcommand{\ax}{\operatorname{Ax}}
\newcommand{\len}{\operatorname{len}}
\newcommand{\true}{\operatorname{True}}
\newcommand{\cut}{\operatorname{Cut}}
\newcommand{\refl}{\operatorname{Ref}}
\newcommand{\rep}{\operatorname{Rep}}
\newcommand{\rk}{\operatorname{rk}}

\title[$\Pi^1_1$-Comprehension as a Well-Ordering Principle]{$\Pi^1_1$-Comprehension as a Well-Ordering Principle\footnotemark[1]}
\author{Anton Freund}
\address{Anton Freund, Fachbereich Mathematik, Technische Universit\"at Darmstadt, Schloss\-gartenstr.~7, 64289 Darmstadt, Germany}

\begin{document}

\begin{abstract}
A dilator is a particularly uniform transformation $X\mapsto T_X$ of linear orders that preserves well-foundedness. We say that $X$ is a Bachmann-Howard fixed point of~$T$ if there is an almost order preserving collapsing function $\vartheta:T_X\rightarrow X$ (precise definition to follow). In the present paper we show that $\Pi^1_1$-comprehension is equivalent to the assertion that every dilator has a well-founded Bachmann-Howard fixed point. This proves a conjecture of M.~Rathjen and A.~Montalb{\'a}n.
\end{abstract}


\maketitle
{\let\thefootnote\relax\footnotetext{\copyright~2019. This manuscript version is made available under the CC-BY-NC-ND 4.0 license \url{http://creativecommons.org/licenses/by-nc-nd/4.0/}. The paper has been accepted for publication in Advances in Mathematics (doi:10.1016/j.aim.2019.106767).}}

\section{Introduction}

The present work rests on the idea that set existence axioms can be split into computationally simple transformations of linear orders and statements about the preservation of well-foundedness. The first example of this phenomenon was discovered by Girard~\cite[Theorem~5.4.1]{girard87}: Given an order~$X$, consider the set
\begin{equation*}
 2^X=\{\langle x_0,\dots,x_{n-1}\rangle\,|\,x_{n-1}<_X\dots <_X x_0\}
\end{equation*}
with the lexicographic order. Then the statement ``if $X$ is well-founded, then $2^X$ is well-founded as well'' is equivalent to arithmetical comprehension. The literature now contains many results of the same type, which characterize transfinite iterations of the Turing jump~\cite{marcone-montalban}, arithmetical transfinite recursion \cite{friedman-mw,rathjen-weiermann-atr,marcone-montalban}, $\omega$-models of arithmetical transfinite recursion~\cite{rathjen-atr}, $\omega$-models of bar induction~\cite{rathjen-model-bi}, and $\omega$-models of $\Pi^1_1$-comprehension with~\cite{thomson-thesis} and without \cite{thomson-rathjen-Pi-1-1} bar induction. For strong set existence axioms the corresponding transformations of orders are harder to grasp but nevertheless computable. There is no limit on the consistency strength of set existence axioms that can be characterized in this way, at least in principle. On the other hand, there is a limitation in terms of logical complexity: For any computable transformation of linear orders, preservation of well-foundedness is expressed by a $\Pi^1_2$-statement. Thus a genuine $\Pi^1_3$-statement, such as the principle of $\Pi^1_1$-comprehension, cannot be equivalent to an assertion of this form. The limitation arises because we have only considered well-ordering principles of type one, i.e.~transformations of well-orders into well-orders. Rathjen~\cite{rathjen-wops-chicago,rathjen-atr} and Montalb\'an~\cite{montalban-draft,montalban-open-problems} have conjectured that $\Pi^1_1$-comprehension is equivalent to a well-ordering principle of type two. Such a principle should transform a well-ordering principle of type one into a well-order (or into another well-ordering principle of type one, but the type of the codomain can be lowered by Currying). In the present paper we prove Rathjen and Montalb\'an's conjecture.

The type-two well-ordering principle that we will introduce can only take particularly uniform type-one well-ordering principles as input. In order to state the uniformity conditions, we consider the category of linear orders with order embeddings as morphisms. We will omit the forgetful functor from an order to its underlying set. Conversely, a subset of an ordered set will often be considered as a suborder. For a set $X$ we define
\begin{equation*}
 [X]^{<\omega}=\text{``the set of finite subsets of $X$''}.
\end{equation*}
To get a functor we map $f:X\rightarrow Y$ to the function $[f]^{<\omega}:[X]^{<\omega}\rightarrow[Y]^{<\omega}$ with
\begin{equation*}
 [f]^{<\omega}(a)=\{f(s)\,|\,s\in a\}.
\end{equation*}
The following class of type-one well-ordering principles has been singled out by Girard~\cite{girard-pi2} (cf.~also Jervell's~\cite{herbrand-symposium-stern} related notion of homogeneous tree):

\begin{definition}\label{def:prae-dilator}
 A prae-dilator consists of
 \begin{enumerate}[label=(\roman*)]
  \item an endofunctor $T$ of linear orders and
  \item a natural transformation $\supp^T:T\Rightarrow[\cdot]^{<\omega}$ that computes supports, in the sense that any $\sigma\in T_X$ lies in the range of $T_{\iota_\sigma}$, where $\iota_\sigma:\supp^T_X(\sigma)\hookrightarrow X$ is the inclusion.
 \end{enumerate}
 If $T_X$ is well-founded for every well-order $X$, then $(T,\supp^T)$ is called a dilator.
\end{definition}

Girard's notion of pre-dilator (note the different spelling) involves an additional monotonicity condition, which is automatic in the well-founded case, i.e.~for dilators. The natural transformation $\supp^T$ does not appear in Girard's original definition: Instead, Girard demands that (prae-)dilators preserve direct limits and pullbacks. This requirement is equivalent to the existence of (unique and thus natural) support functions, as verified in~\cite[Remark~2.2.2]{freund-thesis}. We will see that it is still very fruitful to make the supports explicit. Also note that Girard defines dilators as endofunctors on the category of ordinals, rather than arbitrary well-orders. This is convenient since isomorphic ordinals are equal. Nevertheless, we do not wish to adopt this restriction, since the Mostowski collapse of arbitrary well-orders is not available in weak set theories.

In order to state our well-ordering principle of type two we need some more terminology: If $(X,<_X)$ is a linear order (or just a preorder), then we define a preorder $\lef_X$ on $[X]^{<\omega}$ by stipulating
\begin{equation*}
 a\lef_X b\quad:\Leftrightarrow\quad\text{``for any $s\in a$ there is a $t\in b$ with $s<_X t$''.}
\end{equation*}
For singletons we write $s\lef_X b$ and $a\lef_X t$ rather than $\{s\}\lef_X b$ resp.~$a\lef_X \{t\}$. In the same manner we define a relation $\leq^{\operatorname{fin}}_X$. We can now introduce the central concept of our investigation (a similar definition can be found in the author's PhD thesis~\cite{freund-thesis} and a preliminary study~\cite{freund-bh-preprint} for the latter):

\begin{definition}\label{def:bachmann-howard-collapse}
 Consider a prae-dilator $(T,\supp^T)$ and an order $X$. A function
 \begin{equation*}
  \vartheta:T_X\rightarrow X
 \end{equation*}
 is called a Bachmann-Howard collapse if the following holds for all $\sigma,\tau\in T_X$:
 \begin{enumerate}[label=(\roman*)]
  \item If we have $\sigma<_{T_X}\tau$ and $\supp^T_X(\sigma)\lef_X\vartheta(\tau)$, then we have $\vartheta(\sigma)<_X\vartheta(t)$.
  \item We have $\supp^T_X(\sigma)\lef_X\vartheta(\sigma)$.
 \end{enumerate}
 If such a function exists, then $X$ is called a Bachmann-Howard fixed point of $T$.
\end{definition}

As an example, consider the transformation of an order $X$ into the set
\begin{equation*}
T_X=1+X+X=\{\bot\}\cup X\cup\{\Omega+x\,|\,x\in X\}
\end{equation*}
with the expected order relation (in particular $\bot<_{T_X}x<_{T_X}\Omega+y$ for all $x,y\in X$). It is straightforward to see that this gives rise to a dilator, where the support functions are given by $\supp^T_X(\bot)=\emptyset$ and $\supp^T_X(x)=\supp^T_X(\Omega+x)=\{x\}$. The order-type of $T_X$ is always bigger than the order-type of~$X$, so that we cannot hope for a well-order $X$ with a completely order preserving collapse $\vartheta:T_X\rightarrow X$. Instead, condition~(i) of the previous definition demands that the order is preserved under a side condition. This condition is inspired by the construction of the Bachmann-Howard ordinal, in particular by the notation system due to Rathjen (see~\cite[Section~1]{rathjen-weiermann-kruskal}). In the example of $T_X=1+X+X$ we can specify a Bachmann-Howard collapse $\vartheta:T_{\omega^\omega}\rightarrow\omega^\omega$ by setting $\vartheta(\bot)=0$, $\vartheta(\alpha)=\alpha+1$ and~$\vartheta(\Omega+\alpha)=\omega\cdot(\alpha+1)$. Conversely, if $\vartheta:T_X\rightarrow X$ is any Bachmann-Howard collapse, then we can define an order embedding $f:\omega^\omega\rightarrow X$ by stipulating $f(0)=\vartheta(\bot)$, $f(\alpha+1)=\vartheta(f(\alpha))$ and $f(\omega\cdot\alpha)=\vartheta(\Omega+f(\alpha))$ for $\alpha>0$. We will be interested in the following general principle:

\begin{definition}\label{def:abstract-bhp}
 The abstract Bachmann-Howard principle is the assertion that every dilator has a well-founded Bachmann-Howard fixed point.
\end{definition}

In order to consider the Bachmann-Howard principle from a meta-mathematical perspective we should discuss its formalization: Throughout this paper we will work in the theory $\atrs$, the set-theoretic version of arithmetical transfinite recursion due to Simpson~\cite{simpson82} (an equivalent but somewhat different axiomatization is presented in~\cite{simpson09}). This theory proves the totality of all primitive recursive set functions in the sense of Jensen and Karp~\cite{jensen-karp}. We may thus assume that a function symbol for each of these functions is present. Most constructions in the present paper will be primitive recursive (in the set-theoretic sense). Occasionally, we will need the additional axioms of $\atrs$: axiom beta, which asserts that every well-founded relation can be collapsed to the $\in$-relation; and the axiom of countability, which asserts that every set is countable. When we speak about class-sized objects (such as dilators) we have to observe two restrictions: We will only consider classes which are primitive recursive (with parameters). Furthermore, we cannot quantify over all primitive recursive classes. We can, however, quantify over each primitive recursive family $(F(u,\cdot))_{u\in\mathbb V}$ of class functions, by quantifying over the set-sized parameter~$u$. In the case of (prae-)dilators these restrictions are harmless: Due to their uniformity, dilators are essentially determined by their (set-sized) restrictions to the category of natural numbers, as shown by Girard~\cite{girard-pi2}. In~\cite{freund-computable} we build on this result to construct a single primitive recursive family that comprises isomorphic copies of all prae-dilators. This allows to express the abstract Bachmann-Howard principle by a single formula (see~\cite[Proposition~2.10]{freund-computable}). One can also represent dilators in second-order arithmetic, but this is not needed in the present paper. Having discussed the formalization of dilators, we can now state our main result (numbered according to its occurrence in the text). A similar result can be found in the author's PhD thesis~\cite{freund-thesis}, building on the earlier preprint~\cite{freund-bh-preprint}.

\newtheorem*{thm:main-abstract}{Theorem \ref{thm:main-abstract}}
\begin{thm:main-abstract}
 The following are equivalent over $\atrs$:
 \begin{enumerate}[label=(\roman*)]
  \item The principle of $\Pi^1_1$-comprehension.
  \item The statement that every set is an element of an admissible set.
  \item The abstract Bachmann-Howard principle.
 \end{enumerate}
\end{thm:main-abstract}

Recall that admissible sets are defined as transitive models of Kripke-Platek set theory. We will assume that any admissible set contains the ordinal $\omega$. Note, however, that statement (ii) would be just as strong without this assumption. The equivalence between (i) and~(ii) is known: J\"ager~\cite[Section~7]{jaeger-admissibles} has shown that representation trees for admissible sets can be constructed in $\mathbf{\Pi^1_1-CA_0}$. In \cite[Section~1.4]{freund-thesis} we have verified that $\atrs$ can transform these representation trees into actual admissible sets. The aim of the present paper is to prove the equivalence between (ii) and~(iii).

The abstract Bachmann-Howard principle appears interesting from various perspectives, including those of set theory, computability theory and proof theory: From a set-theoretic standpoint it can be read as a combinatorial version of $\Sigma$-reflection (which is the characteristic axiom of Kripke-Platek set theory). In computability theory one might ask whether the abstract Bachmann-Howard principle can be used to compute the hyperjump (just as the type-one well-ordering principle $X\mapsto 2^X$ can be used to compute the Turing jump, due to Hirst~\cite{hirst94}). For a proof theorist the theorem sheds light on the role of the Church-Kleene ordinal $\omega_1^{\operatorname{CK}}$ in the collapsing construction. More specifically, our result helps to understand an observation of Pohlers~\cite[Section~9.7]{pohlers-proof-theory}, who has shown that particular instances of collapsing can be interpreted in terms of ordinals below $\omega_1^{\operatorname{CK}}$.

The abstract Bachmann-Howard principle is attractive because of its simplicity, but this comes at a price: The transformation of input (a given dilator $T$) into output (a well-founded Bachmann-Howard fixed-point of $T$) is not underpinned by construction (the specification ``abstract'' refers to this fact). In particular the following questions are not separated: How strong is the assertion that any prae-dilator has a Bachmann-Howard fixed point? And how strong is the additional requirement that there are well-founded fixed points in the case of dilators? Thus the abstract Bachmann-Howard principle is not a well-ordering principle in the strict sense. This defect is fixed in two concurrent papers: In~\cite{freund-categorical} we show that each prae-dilator has a minimal Bachmann-Howard fixed point, which can be constructed by a primitive recursive set function. Due to its minimality, the well-foundedness of this particular fixed point is equivalent to the assertion that some well-founded fixed point exists. In~\cite{freund-computable} we show that the minimal Bachmann-Howard fixed point of a prae-dilator $T$ can be described by a notation system, which is computable relative to a representation of $T$ in second-order arithmetic. We also show that the Bachmann-Howard principle implies arithmetical transfinite recursion (based on a result of Rathjen and Valencia Vizca\'ino~\cite{rathjen-model-bi}). Thus we can finally split the principle of $\Pi^1_1$-comprehension into a computable construction and a statement about the preservation of well-foundedness, over the base theory $\rca$ of computable mathematics. Even though a completely satisfactory solution of Rathjen and Montalb\'an's conjecture requires these additional constructions, the main step is the proof of Theorem~\ref{thm:main-abstract} in the present paper.

We point out that a related characterization can be found in the unpublished second part of Girard's book on proof theory~\cite{girard-book-part2}: Girard states that $\Pi^1_1$-compre\-hension is equivalent to the assertion that his functor $\Lambda$ maps dilators to dilators. He describes a proof, which relies on functorial cut elimination for theories of inductive definitions, but notes that the proof is incomplete because intermediate results are missing. Our argument was devised independently of this approach. It would be very interesting to establish a direct connection between Girard's functor $\Lambda$ and our Bachmann-Howard principle, but we have not yet been able to find one. At the same time, the fundamental insights from Girard's published papers on $\Pi^1_2$-logic~\cite{girard-pi2,girard-intro} were a crucial ingredient for the present work.

Let us explain how to prove implication~(iii)$\Rightarrow$(ii) of Theorem~\ref{thm:main-abstract}: In order to construct admissible sets we use Sch\"utte's method of proof search via deduction chains (see~\cite{schuette56} as well the presentation in \cite[Section~II.3]{schuette77}). The idea is to build an attempted proof, starting with a formula $\varphi$ at the root. If the proof search terminates, then one has a well-founded proof of~$\varphi$. Otherwise the attempted proof has an infinite branch. From this branch one can construct a model in which all formulas on the branch fail. In particular one has a countermodel to $\varphi$. Sch\"utte refers to the nodes of the attempted proof as deduction chains; we will speak of a search tree in order to refer to the attempted proof as a whole. Our construction of admissible sets will work roughly as follows: For each ordinal~$\alpha$ we build a search tree $S_\alpha$, which assumes the Kripke-Platek axioms and derives a contradiction in $\mathbb L_\alpha$-logic. This means that $S_\alpha$ is an infinite proof tree which may use the rule
\begin{prooftree}
\Axiom$\cdots\qquad\varphi\fCenter(a)\qquad\cdots\qquad(a\in\mathbb L_\alpha)$
\UnaryInf$\forall_x\fCenter\varphi(x)$
\end{prooftree}
with a premise for each set in the $\alpha$-th stage of the constructible hierarchy. If one of the trees $S_\alpha$ has an infinite branch, then we can construct a set $M\subseteq\mathbb L_\alpha$ such that $(M,\in)$ satisfies the Kripke-Platek axioms. The transitive collapse of $M$ is the admissible set demanded by~(ii). If the trees $S_\alpha$ are all well-founded, then they form a dilator (with respect to the Kleene-Brouwer order). In that case the abstract Bachmann-Howard principle assumed in~(iii) yields a well-founded Bachmann-Howard fixed point $\beta$. The collapsing function $\vartheta:S_\beta\rightarrow\beta$ allows us to replicate J\"ager's~\cite{jaeger-kripke-platek} ordinal analysis of Kripke-Platek set theory. As a result we learn that $S_\beta$ cannot be a proof of contradiction after all. Thus one of the trees $S_\alpha$ must have an infinite branch, and we obtain~(ii) as explained above. We~remark that the method of deduction chains is well-established for $\omega$-proofs (see in particular~\cite{jaeger-strahm-bi-reflection} and~\cite{rathjen-afshari}). As far as the present author is aware, the only application to functorial families of proofs ($\beta$-proofs) can be found in a paper by Buchholz~\cite{buchholz-inductive-dilator} and its generalization by J\"ager~\cite{jaeger86}. The idea to employ deduction chains for $\beta$-proofs in order to obtain a characterization of $\Pi^1_1$-comprehension is due to Rathjen~\cite{rathjen-wops-chicago,rathjen-atr}. He suggested to use this approach in order to construct $\beta$-models of second-order arithmetic. The present paper seems to contain the first application of these methods in a set-theoretic context.

We now explain how the present paper is organized: Section~\ref{sect:admissible-tu-bh} starts with an easy proof that the abstract Bachmann-Howard principle is sound, which relies on the existence of an uncountable cardinal. We then show that the distinction between countable and uncountable sets can be replaced by the distinction between elements and subclasses of an admissible set. This will prove implication~(ii)$\Rightarrow$(iii) of Theorem~\ref{thm:main-abstract}. In Section~\ref{sect:functorial-L} we prepare the construction of the aforementioned search trees $S_\alpha$: To ensure that these trees form a prae-dilator we cannot, in fact, work with the actual constructible hierarchy~$\mathbb L$. Instead we will describe a functorial construction of term systems $\mathbf L_X$ for all linear orders~$X$. In case that $X$ is isomorphic to an ordinal $\alpha$, the terms in $\mathbf L_X$ can be interpreted by elements of $\mathbb L_\alpha$. The construction of the trees $S_\alpha$ itself can be found in Section~\ref{sect:deduct-admissible}. There we will also show how a branch in $S_\alpha$ can be transformed into an admissible set. Following the proof sketch above, it remains to show that the trees $S_\alpha$ cannot form a dilator. It turns out that a Bachmann-Howard fixed point of $S_\alpha$ is not quite sufficient for this purpose: We need a Bachmann-Howard collapse $\vartheta:\varepsilon(S)_\alpha\rightarrow\alpha$ of a strengthened dilator, which will be constructed in Section~\ref{section:epsilon-variant}. The ordinal analysis itself, which leads to the desired contradiction, can be found in Sections~\ref{sect:infinite-proofs} to~\ref{sect:collapsing}.

To conclude this introduction we briefly discuss how the methods of the present paper might be generalized: One would certainly expect that Kripke-Platek set theory can be replaced by any other set theory $\mathbf T$ for which we have an ordinal analysis. Studying the latter, one should be able to find a type-two well-ordering principle that is equivalent to the statement that every set is contained in a transitive model of~$\mathbf T$. In suitable cases this would also characterize the existence of corresponding $\beta$-models. More generally, one might hope that any natural $\Pi^1_3$-statement is equivalent to a meaningful well-ordering principle of type two. Rathjen and the present author~\cite{freund-rathjen_derivatives} have recently proved another equivalence of this form: The principle of $\Pi^1_1$-bar induction corresponds to the well-ordering principle that transforms a given normal function into its derivative. Apart from these concrete applications, the author hopes that the functorial version of the constructible hierarchy (see Section~\ref{sect:functorial-L} below) will prove fruitful in different contexts.

\subsection*{Acknowledgements} This paper is based on parts of my PhD thesis~\cite{freund-thesis}. I am deeply grateful to Michael Rathjen, my PhD supervisor, for everything he has taught me. Also, I would like to acknowledge support from the University of Leeds.

\section{From Admissible Sets to the Bachmann-Howard Principle}\label{sect:admissible-tu-bh}

The height of a transitive set $u$ is defined as its intersection
\begin{equation*}
 o(u)=u\cap\ordi
\end{equation*}
with the class of ordinals. In the present section we show that $o(\mathbb A)$ is a Bachmann-Howard fixed point of any dilator $T$ with parameters in the admissible set $\mathbb A$. This will establish the implication~(ii)$\Rightarrow$(iii) of Theorem~\ref{thm:main-abstract}.

To present the main idea we begin with a proof of the abstract Bachmann-Howard principle in a strong meta theory: Consider an arbitrary dilator $(T,\supp^T)$. As explained in the introduction, we may assume that $T$ is a primitive recursive set function, possibly with additional arguments as parameters. In a strong meta theory we can consider a regular cardinal $\kappa>\omega$ such that all these parameters are of hereditary cardinality below $\kappa$. It follows that the value $T_\alpha$ has cardinality below~$\kappa$ for any argument~$\alpha<\kappa$. To establish the abstract Bachmann-Howard principle we show that $\kappa$ is a Bachmann-Howard fixed point of~$T$. A Bachmann-Howard collapse $\vartheta:T_\kappa\rightarrow\kappa$ can be defined by recursion along the well-order $T_\kappa$: Assuming that $\vartheta(\sigma)$ is already defined for all~$\sigma<_{T_\kappa}\tau$, we consider the sets
\begin{equation*}
 C(\tau,\alpha)=\supp^T_\kappa(\tau)\cup\{\vartheta(\sigma)\,|\,\sigma<_{T_\kappa}\tau\text{ and }\supp^T_\kappa(\sigma)\subseteq\alpha\}\subseteq\kappa
\end{equation*}
for $\alpha<\kappa$. By the definition of (prae-)dilator any $\sigma\in T_\kappa$ lies in the range of~$T_{\iota_\sigma}$, where $\iota_\sigma:\supp^T_\kappa(\sigma)\hookrightarrow\kappa$ is the inclusion. The condition $\supp^T_\kappa(\sigma)\subseteq\alpha$ ensures that $\iota_\sigma$ factors through $\iota_\alpha:\alpha\hookrightarrow\kappa$. Thus $\sigma$ lies in the range of $T_{\iota_\alpha}:T_\alpha\rightarrow T_\kappa$ as~well. Since $T_\alpha$ has cardinality below $\kappa$, we learn that the same holds for $C(\tau,\alpha)$. We can thus construct a sequence $0=\alpha_0<\alpha_1<\dots<\kappa$ with $C(\tau,\alpha_n)\subseteq\alpha_{n+1}$ for all $n\in\omega$. It is easy to see that $\alpha=\sup_{n\in\omega}\alpha_n<\kappa$ satisfies $C(\tau,\alpha)\subseteq\alpha$. To complete the recursive definition of $\vartheta$ we can now set
\begin{equation*}
 \vartheta(\tau)=\min\{\alpha<\kappa\,|\,C(\tau,\alpha)\subseteq\alpha\}.
\end{equation*}
It is straightforward to verify that $\vartheta:T_\kappa\rightarrow\kappa$ is a Bachmann-Howard collapse (cf.~the proof of Proposition~\ref{prop:admissible-Bachmann-Howard-collapse} below). We point out that the given argument is inspired by the usual construction of the Bachmann-Howard ordinal. Rathjen~\cite[Section~4]{rathjen92} has observed that the cardinal~$\kappa$ can be replaced by the class of all ordinals, provided that the set-theoretic universe satisfies the Kripke-Platek axioms. Similarly, we will show that $\kappa$ can be replaced by an admissible ordinal $o(\mathbb A)$.

Given an admissible set $\mathbb A$ that contains the parameters of a dilator $(T,\supp^T)$, the idea is to define a Bachmann-Howard collapse $\vartheta:T_{o(\mathbb A)}\rightarrow o(\mathbb A)$ by stipulating
\begin{equation*}
 \vartheta(\sigma)=\alpha\quad\Leftrightarrow\quad\mathbb A\vDash\theta_T(\sigma,\alpha),
\end{equation*}
where the formula $\theta_T$ reflects the construction from the previous paragraph. Since $\theta_T$ will need to speak about $T$, we have to define this dilator within $\mathbb A$: According to \cite[Section~2.2]{jensen-karp} the primitive recursive definition of $X\mapsto T_X$ corresponds to a $\Sigma$-formula $\mathcal D_T(X,Y)$, with further free variables for the parameters of $T$, which defines $T$ in the set-theoretic universe. By induction over primitive recursive set functions one shows that $\forall_X\exists_Y\mathcal D_T(X,Y)$ is provable in Kripke-Platek set theory (the crucial case of a primitive recursion is covered by the $\Sigma$-recursion theorem, see e.g.~\cite[Theorem~I.6.4]{barwise-admissible}). Together with the upward absoluteness of $\Sigma$-formulas one can conclude
\begin{equation*}
 \forall_{X\in\mathbb A}\forall_Y(T_X=Y\leftrightarrow Y\in\mathbb A\land\mathbb A\vDash \mathcal D_T(X,Y))
\end{equation*}
for any admissible set $\mathbb A$ that contains the parameters of $T$. In particular we have $T_\gamma\in\mathbb A$ for any $\gamma<o(\mathbb A)$. As for any total function, we can infer that $\mathcal D_T(X,Y)$ is a $\Delta$-formula from the viewpoint of $\mathbb A$. In the following we write $\mathbb A\vDash T_X=Y$ rather than~$\mathbb A\vDash D_T(X,Y)$. Similarly we write $\mathbb A\vDash \sigma\in T_X$ to refer to a $\Delta$-definition of this relation in~$\mathbb A$. Even though we have defined $T$ within $\mathbb A$, the formula $\theta_T$ cannot refer to the value $T_{o(\mathbb A)}$ itself, since the argument $o(\mathbb A)$ is not contained in $\mathbb A$. In the following we work with dilators that approximate $T_{o(\mathbb A)}$ in a particularly convenient way. Afterwards we will transfer the result to arbitrary dilators.

\begin{definition}
 A dilator $(T,\supp^T)$ is called inclusive if any inclusion $\iota:X\hookrightarrow Y$ of linear orders is mapped to an inclusion $T_\iota:T_X\hookrightarrow T_Y$.
\end{definition}

If $T$ is an inclusive dilator with parameters in an admissible set $\mathbb A$, then we have
\begin{equation*}
 T_{o(\mathbb A)}=\bigcup_{\gamma<o(\mathbb A)}T_\gamma\subseteq\mathbb A.
\end{equation*}
The inclusion $\subseteq$ of the equality relies on the fact that any $\sigma\in T_{o(\mathbb A)}$ has finite support $\supp^T_{o(\mathbb A)}(\sigma)\subseteq\gamma$ for some ordinal $\gamma$ below the limit $o(\mathbb A)$. Using the second recursion theorem for admissible sets (see~\cite[Theorem~V.2.3]{barwise-admissible}) we can now construct the required formula:

\begin{definition}
 For each inclusive dilator $(T,\supp^T)$, let $\theta_T(\sigma,\alpha)$ be a $\Sigma$-formula such that $\mathbb A\vDash\theta_T(\sigma,\alpha)$ is equivalent to
 \begin{align*}
\mathbb A\vDash\exists_\gamma(&\sigma\in T_\gamma\,\land\,\forall_{\beta<\gamma}\,\sigma\notin T_\beta\,\land\,{}\\
& \begin{aligned}
\exists_f(&\text{``$f:\omega\rightarrow\ordi$ is a function"}\,\land\, f(0)=\gamma\,\land\,{}\\
& \begin{aligned}\forall_{n\in\omega}\exists_d(&\text{``$d:\{\tau\in T_{f(n)}\,|\,\tau<_{T_{f(n)}}\sigma\}\rightarrow\ordi$ is a function"}\,\land\,{}\\
		&\forall_{\tau\in\dom(d)}\,\theta_T(\tau,d(\tau))\,\land\,{}\\
		&f(n+1)=\sup\{d(\tau)+1\,|\,\tau\in\dom(d)\})\,\land\,{}\end{aligned}\\
& \alpha=\textstyle\sup_{n\in\omega} f(n))),
\end{aligned}
\end{align*}
 for any admissible set $\mathbb A\supseteq\{\sigma,\alpha\}$ that contains the parameters of $(T,\supp^T)$.
\end{definition}

In order to show that $\theta_T$ defines a function on $\mathbb A$, we first establish uniqueness:

\begin{lemma}
 We have
 \begin{equation*}
  \forall_{\sigma\in T_{o(\mathbb A)}}\forall_{\alpha_0,\alpha_1<o(\mathbb A)}(\mathbb A\vDash\theta_T(\sigma,\alpha_0)\land \mathbb A\vDash\theta_T(\sigma,\alpha_1)\rightarrow\alpha_0=\alpha_1),
 \end{equation*}
 for any inclusive dilator $T$ with parameters in the admissible set $\mathbb A$.
\end{lemma}
\begin{proof}
 We argue by induction over $T_{o(\mathbb A)}$, which is a well-order because $T$ is a dilator. Note that this induction can be formalized in our meta theory $\atrs$: Satisfaction in a model is a primitive recursive relation, and $\atrs$ proves separation for primitive recursive predicates (see~\cite[Section~1]{freund-thesis} for details). To establish the claim for $\sigma\in T_{o(\mathbb A)}$, let us assume $\mathbb A\vDash\theta_T(\sigma,\alpha_i)$ for $i=0,1$. By the defining equivalence of $\theta_T$ we obtain witnesses $\gamma_i,f_i\in\mathbb A$. To conclude $\alpha_0=\alpha_1$ it suffices to show $f_0(n)=f_1(n)$ by induction over $n$. In the base we observe $f_0(0)=\gamma_0=\gamma_1=f_1(0)$. In the step we write $f_0(n)=f(n)=f_1(n)$ and consider the functions
 \begin{equation*}
  d_i:\{\tau\in T_{f(n)}\,|\,\tau<_{T_{f(n)}}\sigma\}\rightarrow\ordi
 \end{equation*}
 from the defining equivalence of $\theta_T$. By functoriality $\tau<_{T_{f(n)}}\sigma$ implies $\tau<_{T_{o(\mathbb A)}}\sigma$. So by induction hypothesis the condition $\theta_T(\tau,d_i(\tau))$ determines~$d_i(\tau)$. This yields $d_0=d_1$ and then $f_0(n+1)=f_1(n+1)$, as desired.
\end{proof}

Using the fact that $\mathbb A$ is admissible, we can now establish existence:

\begin{proposition}
 We have
 \begin{equation*}
  \forall_{\sigma\in T_{o(\mathbb A)}}\exists_{\alpha<o(\mathbb A)}\,\mathbb A\vDash\theta_T(\sigma,\alpha),
 \end{equation*}
 for any inclusive dilator $T$ with parameters in the admissible set $\mathbb A$.
\end{proposition}
\begin{proof}
 As in the previous proof we argue by induction over $\sigma\in T_{o(\mathbb A)}$. We have already seen that $\sigma\in T_\gamma$ holds for some $\gamma<o(\mathbb A)\subseteq\mathbb A$. Pick the smallest such~$\gamma$, and observe that this is the witness required by the defining equivalence of $\theta_T$. To conclude we must show that $\mathbb A$ contains a suitable function $f:\omega\rightarrow\ordi$. By induction on $k$ we show that there are approximations $f_k\in\mathbb A$ with
 \begin{align*}
 \mathbb A\vDash&\text{``$f_k:k+1\rightarrow\ordi$ is a function"}\,\land\, f_k(0)=\gamma\,\land\\
& \begin{aligned}\forall_{n<k}\exists_d(&\text{``$d:\{\tau\in T_{f_k(n)}\,|\,\tau<_{T_{f_k(n)}}\sigma\}\rightarrow\ordi$ is a function"}\,\land\\
		&\forall_{\tau\in\dom(d)}\,\theta_T(\tau,d(\tau))\,\land\\
		&f_k(n+1)=\sup\{d(\tau)+1\,|\,\tau\in\dom(d)\}).\end{aligned}
\end{align*}
 For $k=0$ we simply set $f_0=\{\langle 0,\gamma\rangle\}$. In order to extend $f_k$ to $f_{k+1}$ it suffices to show that $\mathbb A$ contains a suitable function
 \begin{equation*}
  d:\{\tau\in T_{f_k(n)}\,|\,\tau<_{T_{f_k(n)}}\sigma\}\rightarrow\ordi.
 \end{equation*}
 The domain of $d$ exists by $\Delta$-separation in $\mathbb A$ (see \cite[Theorem~I.4.5]{barwise-admissible}). The induction hypothesis and the previous lemma imply
 \begin{equation*}
  \forall_{\tau\in T_{f_k(n)}}(\tau<_{T_{f_k(n)}}\sigma\rightarrow\exists!_{\delta<o(\mathbb A)}\,\mathbb A\vDash\theta_T(\tau,\delta)).
 \end{equation*}
 Now $\Sigma$-replacement in $\mathbb A$ (see \cite[Theorem~I.4.6]{barwise-admissible}) yields the desired function $d$. This completes the recursive construction of the functions $f_k$. As in the previous lemma, we see that these functions are unique. Thus another application of $\Sigma$-replacement shows that the function $k\mapsto f_k$ lies in $\mathbb A$. Finally, the desired function $f\in\mathbb A$ can be defined by $f(n)=f_n(n)$. It witnesses $\mathbb A\vDash\theta_T(\sigma,\alpha)$ for $\alpha=\sup_{n\in\omega}f(n)$.
\end{proof}

In our meta theory $\atrs$ we invoke primitive recursive separation to complete the construction of our collapsing function:

\begin{definition}
 Consider an inclusive dilator $(T,\supp^T)$ and an admissible set $\mathbb A$ that contains the parameters of $T$. In view of the previous results, the stipulation
 \begin{equation*}
  \vartheta_{\mathbb A}=\{\langle\sigma,\alpha\rangle\in T_{o(\mathbb A)}\times o(\mathbb A)\,|\,\mathbb A\vDash\theta_T(\sigma,\alpha)\}
 \end{equation*}
 defines a function $\vartheta_{\mathbb A}:T_{o(\mathbb A)}\rightarrow o(\mathbb A)$.
\end{definition}

Let us verify that we have indeed constructed a Bachmann-Howard collapse:

\begin{proposition}\label{prop:admissible-Bachmann-Howard-collapse}
 If $(T,\supp^T)$ is an inclusive dilator with parameters in the admissible set~$\mathbb A$, then $\vartheta_{\mathbb A}:T_{o(\mathbb A)}\rightarrow o(\mathbb A)$ is a Bachmann-Howard collapse.
\end{proposition}
\begin{proof}
 To verify condition~(i) of Definition~\ref{def:bachmann-howard-collapse} we consider elements $\sigma,\tau\in T_{o(\mathbb A)}$ with $\sigma<_{T_{o(\mathbb A)}}\tau$ and $\supp^T_{o(\mathbb A)}(\sigma)\lef\vartheta_{\mathbb A}(\tau)$. If $f$ witnesses $\theta_T(\tau,\vartheta_{\mathbb A}(\tau))$, then we get $\supp^T_{o(\mathbb A)}(\sigma)\lef f(n)$ for some $n$. Thus $\iota_\sigma:\supp^T_{o(\mathbb A)}(\sigma)\hookrightarrow o(\mathbb A)$ factors through the inclusion of $f(n)$ into $o(\mathbb A)$. Invoking the definition of (prae-)dilator we can infer~$\sigma\in T_{f(n)}$. We may also assume $\tau\in T_{f(n)}$, switching to $n=0$ if $f(n)<f(0)$. Then the definition of $\theta_T$ yields $\vartheta_{\mathbb A}(\sigma)<f(n+1)\leq\vartheta_{\mathbb A}(\tau)$, as required. To verify condition~(ii) we consider an arbitrary $\sigma\in T_{o(\mathbb A)}$. Let $\gamma$ and $f$ be witnesses for the defining equivalence of~$\theta_T(\sigma,\vartheta_{\mathbb A}(\sigma))$. In view of $\sigma\in T_\gamma$ we may write $\sigma=T_{\iota_\gamma}(\sigma)$, where $\iota_\gamma:\gamma\hookrightarrow o(\mathbb A)$ is the inclusion. Using the naturality of $\supp^T$ we obtain
 \begin{equation*}
  \supp^T_{o(\mathbb A)}(\sigma)=\supp^T_{o(\mathbb A)}(T_{\iota_\gamma}(\sigma))=[\iota_\gamma]^{<\omega}(\supp^T_\gamma(\sigma))\subseteq\gamma=f(0)\leq\vartheta_{\mathbb A}(\sigma),
 \end{equation*}
 which amounts to the desired condition $\supp^T_{o(\mathbb A)}(\sigma)\lef\vartheta_{\mathbb A}(\sigma)$.
\end{proof}

Finally, we can deduce the direction (ii)$\Rightarrow$(iii) of Theorem~\ref{thm:main-abstract}. Note that the result is established for arbitrary dilators, not just for inclusive ones:

\begin{theorem}\label{thm:admissible-to-bhp}
 If every set is contained in an admissible set, then every dilator has a well-founded Bachmann-Howard fixed point.
\end{theorem}
\begin{proof}
 Consider an arbitrary dilator $(T,\supp^T)$. The main task is to construct an equivalent dilator which is inclusive: For each order $X$ we consider the set
 \begin{equation*}
 D^T_X=\{\langle a,\sigma\rangle\,|\,a\in[X]^{<\omega}\text{ and }\sigma\in T_a\text{ with }\supp^T_a(\sigma)=a\}.
\end{equation*}
Define functions $\eta^T_X:D^T_X\rightarrow T_X$ by setting $\eta^T_X(\langle a,\sigma\rangle)=T_{\iota_a}(\sigma)$, where $\iota_a:a\hookrightarrow X$ is the inclusion. The condition $\supp^T_a(\sigma)=a$ and the naturality of $\supp^T$ ensure that $a$ can be recovered from $\eta^T_X(\langle a,\sigma\rangle)$, namely as
\begin{equation*}
 \supp^T_X(\eta^T_X(\langle a,\sigma\rangle))=\supp^T_X(T_{\iota_a}(\sigma))=[\iota_a]^{<\omega}(\supp^T_a(\sigma))=a.
\end{equation*}
Since $T_{\iota_a}$ is an embedding, it follows that the function $\eta^T_X:D^T_X\rightarrow T_X$ is injective. It is also surjective: Given $\sigma\in T_X$ we set $a=\supp^T_X(\sigma)$. Invoking the definition of (prae-)dilator we obtain $\sigma=T_{\iota_a}(\sigma_0)$ for some $\sigma_0\in T_a$. In view of
\begin{equation*}
 [\iota_a]^{<\omega}(\supp^T_a(\sigma_0))=\supp^T_X(T_{\iota_a}(\sigma_0))=\supp^T_X(\sigma)=a
\end{equation*}
we have $\supp^T_a(\sigma_0)=a$ and thus $\langle a,\sigma_0\rangle\in D^T_X$. So $\sigma=T_{\iota_a}(\sigma_0)=\eta^T_X(\langle a,\sigma_0\rangle)$ does indeed lie in the range of $\eta^T_X$. One can check that
\begin{equation*}
 D^T_f(\langle a,\sigma\rangle)=\langle[f]^{<\omega}(a),T_{f\restriction a}(\sigma)\rangle
\end{equation*}
turns $D^T$ into a functor from linear orders to sets, and that $\eta^T:D^T\Rightarrow T$ becomes a natural equivalence. Using this equivalence we can transfer the order from $T_X$ to~$D^T_X$. Together with the stipulation
\begin{equation*}
 \supp^{D^T}_X(\langle a,\sigma\rangle)=a
\end{equation*}
we obtain a dilator. One should also observe
\begin{equation*}
 \supp^T_X\circ\eta^T_X=\supp^{D^T}_X.
\end{equation*}
The point is that the dilator $D^T$ is inclusive. By assumption there is an admissible set $\mathbb A$ that contains its parameters. The previous proposition yields a Bachmann-Howard collapse $\vartheta_{\mathbb A}:D^T_{o(\mathbb A)}\rightarrow o(\mathbb A)$. It is straightforward to deduce that
\begin{equation*}
 \vartheta=\vartheta_{\mathbb A}\circ\left(\eta^T_{o(\mathbb A)}\right)^{-1}:T_{o(\mathbb A)}\rightarrow o(\mathbb A)
\end{equation*}
is a Bachmann-Howard collapse as well. This shows that $o(\mathbb A)$ is a well-founded Bachmann-Howard fixed point of $T$, as desired.
\end{proof}

\section{A Functorial Version of the Constructible Hierarchy}\label{sect:functorial-L}

In this section we construct a term system $\mathbf L_X^u$ for any linear order $X$ and any transitive set $u$. If $X$ is an ordinal, then the terms from $\mathbf L_X^u$ can be interpreted as sets in the constructible hierarchy over $u$. The point is that the construction of $\mathbf L^u_X$ is functorial in $X$. This property will be needed for our functorial approach to proof search, as described in the introduction to the present paper. We remark that term versions of the constructible hierarchy are well-known in proof theory (see the work of J\"ager~\cite{jaeger-KPN}, as well as Sch\"utte's~\cite{schuette64} earlier work on ramified analysis). It seems that the functoriality of these term systems has not been checked before.

We consider object formulas in the language with relation symbols $\in$ and $=$. All formulas are assumed to be in negation normal form. This means that they consist of negated and unnegated prime formulas, the connectives $\land$ and $\lor$, as well as existential and universal quantifiers. To compute the negation of a formula one applies de Morgan's rules and deletes double negations. Similarly, other connectives can be used as defined operations on negation normal forms. Bounded quantifiers are considered as separate logical symbols, so that the formulas $\exists_{x\in y}\varphi$ and \mbox{$\exists_x(x\in y\land\varphi)$} are equivalent but syntactically different. A formula without unbounded quantifiers is called $\Delta_0$-formula or bounded formula. Let us define the promised term systems:

\begin{definition}\label{def:term-version-L}
 Consider a transitive set $u$. For each linear order $X$ we define a set~$\mathbf L^u_X$ and a support function $\suppl_X:\mathbf L^u_X\rightarrow[X]^{<\omega}$ by the following induction:
 \begin{enumerate}[label=(\roman*)]
  \item Each element $v\in u$ is a term in $\mathbf L^u_X$, with support $\suppl_X(v)=\emptyset$.
  \item For each $s\in X$ we have a term $L^u_s$ in $\mathbf L^u_X$, with support $\suppl_X(L^u_s)=\{s\}$.
  \item Consider a $\Delta_0$-formula $\varphi(x,y_1,\dots,y_n)$ with all free variables displayed, as well as an $s\in X$ and terms $a_1,\dots,a_n\in\mathbf L^u_X$ with $\suppl_X(a_i)\lef_X s$. Then we have a term $\{x\in L^u_s\,|\,\varphi(x,a_1,\dots,a_n)\}$ in $\mathbf L^u_X$, with support
  \begin{equation*}
   \suppl_X(\{x\in L^u_s\,|\,\varphi(x,a_1,\dots,a_n)\})=\{s\}\cup\suppl_X(a_1)\cup\dots\cup\suppl_X(a_n).
  \end{equation*}
 \end{enumerate}
\end{definition}

Our first goal is to define an interpretation of $\mathbf L^u_X$ in the special case that $X$ is an ordinal. For $\alpha<\beta$ it is straightforward to observe
\begin{equation*}
 \mathbf L^u_\alpha=\{a\in\mathbf L^u_\beta\,|\,\suppl_\beta(a)\lef\alpha\}.
\end{equation*}
A more general property will be established in Proposition~\ref{prop:constructible-hierarchy-functorial}.

\begin{definition}\label{def:interpretation-constructible-terms}
Let $u$ be a transitive set. The interpretation function
\begin{equation*}
 \llbracket\cdot\rrbracket_\alpha:\mathbf L^u_\alpha\rightarrow\mathbb L^u_\alpha
\end{equation*}
and its image $\mathbb L^u_\alpha$ are defined by recursion over $\alpha$, with recursive clauses
 \begin{align*}
  \llbracket v\rrbracket_\alpha&=v\quad\text{for $v\in u$},\\
  \llbracket L^u_\gamma\rrbracket_\alpha&=\mathbb L^u_\gamma,\\
  \llbracket\{x\in L^u_\gamma\,|\,\varphi(x,b_1,\dots,b_n)\}\rrbracket_\alpha&=\{x\in\mathbb L^u_\gamma\,|\,\mathbb L^u_\gamma\vDash\varphi(x,\llbracket b_1\rrbracket_\gamma,\dots,\llbracket b_n\rrbracket_\gamma)\}.
  \end{align*}
\end{definition}

For $\alpha<\beta$ and $a\in\mathbf L^u_\alpha\subseteq\mathbf L^u_\beta$ we clearly have $\llbracket a\rrbracket_\alpha=\llbracket a\rrbracket_\beta$. Thus we will write~$\llbracket a\rrbracket$ or even $a$ at the place of $\llbracket a\rrbracket_\alpha$. Before we compare our sets $\mathbb L^u_\alpha$ with the usual constructible hierarchy, let us record some easy consequences of the definitions:

\begin{lemma}\label{lem:basic-properties-constructible}
 The following holds for any transitive set $u$:
 \begin{enumerate}[label=(\alph*)]
  \item The set $\mathbb L^u_\alpha$ is transitive for all ordinals $\alpha$.
  \item We have $\mathbb L^u_\alpha=\llbracket L^u_\alpha\rrbracket\in\mathbb L^u_\beta$ for all ordinals $\alpha<\beta$.
 \end{enumerate}
\end{lemma}

Based on these facts one can verify that we have
\begin{align*}
 \mathbb L^u_0&=u,\\
 \mathbb L^u_{\alpha+1}&=\text{``the $\Delta_0$-definable subsets of $\mathbb L^u_\alpha$''},\\
 \mathbb L^u_\lambda&=\bigcup_{\gamma<\lambda}\mathbb L^u_\gamma\quad\text{for limit ordinals $\lambda$.}
\end{align*}
The only difference to the usual constructible hierarchy is that we restrict the logical complexity of definable subsets in the successor step. This makes no essential difference, since any definable subset of $\mathbb L^u_\alpha$ is a $\Delta_0$-definable subset of $\mathbb L^u_{\alpha+1}$. The restriction to $\Delta_0$-formulas will be convenient for technical reasons. The reader may also have observed that the terms $L^u_\gamma$ and $\{x\in L^u_\gamma\,|\,x=x\}$ have the same interpretation. In the context of infinite proof trees it will nevertheless be important to have a separate term~$L^u_\gamma$. The following result about the height of the transitive set~$\mathbb L^u_\alpha$ is established as for the usual constructible hierarchy (see e.g.~\cite[Lemma~13.2]{jech03}):

\begin{lemma}\label{lem:height-L}
 For any transitive set $u$ and any ordinal $\alpha$ we have
 \begin{equation*}
  o(\mathbb L^u_\alpha)=o(u)+\alpha.
 \end{equation*}
\end{lemma}

We have reconstructed the usual constructible hierarchy $\mathbb L^u$ via a family of term systems $\mathbf L^u_X$. In the following we investigate these term systems in their own right. To exhibit their functorial properties we define corresponding maps on morphisms:

\begin{definition}\label{def:constructible-morphisms}
 For each embedding $f:X\rightarrow Y$ we define a map $\mathbf L^u_f:\mathbf L^u_X\rightarrow\mathbf L^u_Y$ by the recursion
 \begin{align*}
  \mathbf L^u_f(v)&=v,\\
  \mathbf L^u_f(L^u_s)&=L^u_{f(s)},\\
  \mathbf L^u_f(\{x\in L^u_s\,|\,\varphi(x,a_1,\dots,a_n)\})&=\{x\in L^u_{f(s)}\,|\,\varphi(x,\mathbf L^u_f(a_1),\dots,\mathbf L^u_f(a_n))\}.
 \end{align*}
\end{definition}

The following proof shows that the range of $\mathbf L^u_f$ is indeed contained in~$\mathbf L^u_Y$:

\begin{proposition}\label{prop:constructible-hierarchy-functorial}
 For each transitive set $u$, the maps $X\mapsto\mathbf L^u_X$ and $f\mapsto\mathbf L^u_f$ form a functor from linear orders to sets. The functions $\suppl_X:\mathbf L^u_X\rightarrow[X]^{<\omega}$ form a natural transformation. They compute supports, in the sense that we have
 \begin{equation*}
  \rng(\mathbf L^u_f)=\{b\in\mathbf L^u_Y\,|\,\suppl_Y(b)\subseteq\rng(f)\}
 \end{equation*}
 for any order embedding $f:X\rightarrow Y$.
\end{proposition}
\begin{proof}
 We should first verify $\mathbf L^u_f(a)\in\mathbf L^u_Y$ for $a\in\mathbf L^u_X$ and an embedding $f:X\rightarrow Y$. This can be done by induction over the term $a$, if one simultaneously checks the naturality condition
 \begin{equation*}
  \suppl_Y(\mathbf L^u_f(a))=[f]^{<\omega}(\suppl_X(a)).
 \end{equation*}
 The most interesting case is that of a term $a=\{x\in L^u_s\,|\,\varphi(x,a_1,\dots,a_n)\}$. Invoking naturality for $a_i$ we see that $\suppl_X(a_i)\lef_X s$ implies $\suppl_Y(\mathbf L^u_f(a_i))\lef_Y f(s)$. This ensures that $\mathbf L^u_f(a)=\{x\in L^u_{f(s)}\,|\,\varphi(x,\mathbf L^u_f(a_1),\dots,\mathbf L^u_f(a_n))\}$ is a term in $\mathbf L^u_Y$, as desired. The functoriality of $\mathbf L^u$ is established by a straightforward induction over terms. Since $\suppl$ is natural, any $b=\mathbf L^u_f(a)\in\rng(\mathbf L^u_f)$ must satisfy
 \begin{equation*}
  \suppl_Y(b)=[f]^{<\omega}(\suppl_X(a))\subseteq\rng(f).
 \end{equation*}
 Conversely, if an element $b\in\mathbf L^u_Y$ satisfies $\suppl_Y(b)\subseteq\rng(f)$, then the inclusion $\iota:\suppl_Y(b)\hookrightarrow Y$ factors as $\iota=f\circ g$ with $g:\suppl_Y(b)\rightarrow X$. By induction over terms one checks that $b$ lies in $\mathbf L^u_{\suppl_Y(b)}$ and that $\mathbf L^u_\iota:\mathbf L^u_{\suppl_Y(b)}\rightarrow\mathbf L^u_Y$ is an inclusion. Thus $b=\mathbf L^u_\iota(b)=\mathbf L^u_f(\mathbf L^u_g(b))$ lies in the range of $\mathbf L^u_f$, as claimed.
\end{proof}

For our functorial proof search it will be crucial to have compatible enumerations of the term systems $\mathbf L^u_X$. The following is inspired by Girard's~\cite{girard-pi2} work on dilators:

\begin{theorem}\label{thm:enumerations-L}
Consider a countable transitive set $u=\{u_i\,|\,i\in\omega\}$ with a fixed enumeration. One can construct functions
\begin{equation*}
\en_X:[X]^{<\omega}\times\omega\rightarrow\mathbf L_X^u,\qquad
\code_X:[X]^{<\omega}\times\mathbf L_X^u\rightarrow\omega
\end{equation*}
which satisfy
\begin{equation*}
\en_X(x,\code_X(x,a))=a
\end{equation*}
whenever we have $\suppl_X(a)\subseteq x$. The naturality conditions
\begin{align*}
\en_Y([f]^{<\omega}(x),n)&=\mathbf L_f^u(\en_X(x,n)),\\
\code_Y([f]^{<\omega}(x),\mathbf L_f^u(a))&=\code_X(x,a)
\end{align*}
hold for any order embedding $f:X\rightarrow Y$.
\end{theorem}
\begin{proof}
For any natural number $n$ we can define G\"odel numbers for the terms in~$\mathbf L^u_n$, using the given enumeration of $u$. This yields a family of functions
\begin{equation*}
\en_n^0:\omega\rightarrow\mathbf L_n^u,\qquad\code_n^0:\mathbf L_n^u\rightarrow\omega
\end{equation*}
which satisfy
\begin{equation*}
\en_n^0(\code_n^0(a))=a.
\end{equation*}
Given a finite subset $x$ of a linear order $X$, we write
\begin{equation*}
|\iota^X_x|:|x|\rightarrow X
\end{equation*}
for the increasing enumeration of $x\subseteq X$. By the previous proposition any $a\in\mathbf L_X^u$ with $\suppl_X(a)\subseteq x$ lies in the range of $\mathbf L_{|\iota^X_x|}^u$. The representation $a=\mathbf L_{|\iota^X_x|}^u(a_0)$ is unique, since $\mathbf L_{|\iota^X_x|}^u$ factors into the bijection $\mathbf L^u_{|x|}\cong\mathbf L^u_x$ and the inclusion $\mathbf L^u_x\hookrightarrow\mathbf L^u_X$. Note that $a_0$ can be computed by a primitive recursive set function, namely as
\begin{equation*}
 a_0=\bigcup\{b\in\mathbf L^u_{|x|}\,|\,a=\mathbf L_{|\iota^X_x|}^u(b)\}.
\end{equation*}
Now the desired functions can be defined by
\begin{align*}
\en_X(x,n)&=\mathbf L_{|\iota^X_x|}^u(\en_{|x|}^0(n)),\\
\code_X(x,a)&=\begin{cases}
		\code_{|x|}^0(a_0) & \text{if $\suppl_X(a)\subseteq x$ and $a=\mathbf L_{|\iota^X_x|}^u(a_0)$},\\
		0 & \text{if $\suppl_X(a)\nsubseteq x$}.
	       \end{cases}
\end{align*}
Consider an $a\in\mathbf L_X^u$ with $\suppl_X(a)\subseteq x$. Writing $a=\mathbf L_{|\iota^X_x|}^u(a_0)$ we obtain
\begin{equation*}
 \en_X(x,\code_X(x,a))=\mathbf L_{|\iota^X_x|}^u(\en_{|x|}^0(\code^0_{|x|}(a_0)))=\mathbf L_{|\iota^X_x|}^u(a_0)=a,
\end{equation*}
as desired. To establish naturality we note $|[f]^{<\omega}(x)|=|x|$ and $|\iota_{[f]^{<\omega}(x)}^Y|=f\circ|\iota^X_x|$ for any embedding $f:X\rightarrow Y$. Thus we can compute
\begingroup
\allowdisplaybreaks[0]
\begin{multline*}
\en_Y([f]^{<\omega}(x),n)=\mathbf L_{|\iota^Y_{[f]^{<\omega}(x)}|}^u(\en_{|[f]^{<\omega}(x)|}^0(n))=\\
=\mathbf L_f^u(\mathbf L_{|\iota^X_x|}^u(\en_{|x|}^0(n)))=\mathbf L_f^u(\en_X(x,n)).
\end{multline*}
\endgroup%
Since $\suppl$ is a natural transformation we see
\begin{equation*}
\suppl_X(a)\subseteq x\,\Leftrightarrow\, [f]^{<\omega}(\suppl_X(a))\subseteq[f]^{<\omega}(x)\,\Leftrightarrow\,\suppl_Y(\mathbf L_f^u(a))\subseteq[f]^{<\omega}(x).
\end{equation*}
So for $\suppl_X(a)\subseteq x$ with $a=\mathbf L_{|\iota^X_x|}^u(a_0)$ we obtain $\mathbf L_f^u(a)=\mathbf L_{|\iota_{[f]^{<\omega}(x)}^Y|}^u(a_0)$ and
\begin{equation*}
\code_Y([f]^{<\omega}(x),\mathbf L_f^u(a))=\code_{|[f]^{<\omega}(x)|}^0(a_0)=\code_{|x|}^0(a_0)=\code_X(x,a).
\end{equation*}
If we have $\suppl_X(a)\nsubseteq x$ and thus $\suppl_Y(\mathbf L_f^u(a))\nsubseteq[f]^{<\omega}(x)$, then we get
\begin{equation*}
\code_Y([f]^{<\omega}(x),\mathbf L_f^u(a))=0=\code_X(x,a),
\end{equation*}
which completes the proof of naturality.
\end{proof}

The following observation will be needed later:

\begin{corollary}\label{cor:support-enumeration}
 We have $\suppl_X(\en_X(x,n))\subseteq x$ for arbitrary $x\in[X]^{<\omega}$ and $n\in\omega$.
\end{corollary}
\begin{proof}
 Write $\iota_x:x\hookrightarrow X$ for the inclusion. The naturality of $\en$ and $\suppl$ yields
 \begin{multline*}
  \suppl_X(\en_X(x,n))=\suppl_X(\en_X([\iota_x]^{<\omega}(x),n))=\\
  =\suppl_X(\mathbf L^u_{\iota_x}(\en_x(x,n)))=[\iota_x]^{<\omega}(\suppl_x(\en_x(x,n)))\subseteq x,
 \end{multline*}
 as promised.
\end{proof}

Similarly to the usual order on the constructible hierarchy, we can now define compatible order relations on the term systems $\mathbf L^u_X$:

\begin{definition}
Let $u=\{u_i\,|\,i\in\omega\}$ be a countable transitive set with a fixed enumeration. If $(X,<_X)$ is a linear order, then we write $<_X^*$ for the colexicographic order on $[X]^{<\omega}$: For $x\neq x'$ we have
\begin{equation*}
x<_X^*x'\quad\Leftrightarrow\quad\textstyle\max_{<_X}(x\triangle x')\in x',
\end{equation*}
where $x\triangle x'$ denotes the symmetric difference. The relation $<_{\mathbf L^u_X}$ on $\mathbf L_X^u$ is given by
\begin{equation*}
a<_{\mathbf L^u_X} b\quad\Leftrightarrow\quad\begin{cases}
\text{either $\suppl_X(a)<_X^*\suppl_X(b)$},\\[1.5ex]
\parbox[t]{0.7\textwidth}{or $\suppl_X(a)=\suppl_X(b)$ and\newline\hspace*{2cm} $\code_X(\suppl_X(a),a)<\code_X(\suppl_X(b),b)$.}
\end{cases}
\end{equation*}
\end{definition}

We record the following fundamental fact:

\begin{lemma}\label{less-L-linear-well}
If $X$ is a linear order resp.~a well-order, then so is $(\mathbf L_X^u,<_{\mathbf L^u_X})$.
\end{lemma}
\begin{proof}
 To see that $<_{\mathbf L^u_X}$ is trichotomous one observes that $\suppl_X(a)=\suppl_X(b)$ and $\code_X(\suppl_X(a),a)=\code_X(\suppl_X(b),b)$ implies
\begin{multline*}
a=\en_X(\suppl_X(a),\code_X(\suppl_X(a),a))=\\
=\en_X(\suppl_X(b),\code_X(\suppl_X(b),b))=b.
\end{multline*}
 The remaining verifications are straightforward (cf.~\cite[Section~3.1]{freund-thesis}).
\end{proof}

Let us now summarize the functorial properties of our construction:

\begin{theorem}\label{thm:constructible-dilator}
 Let $u=\{u_i\,|\,i\in\omega\}$ be a countable transitive set with a fixed enumeration. Then $(\mathbf L^u,\suppl)$ is a dilator.
\end{theorem}
\begin{proof}
 In Proposition~\ref{prop:constructible-hierarchy-functorial} we have seen that $\mathbf L^u$ is a functor from linear orders to sets. To conclude that it is a functor into the category of linear orders we must show that $\mathbf L^u_f:\mathbf L^u_X\rightarrow\mathbf L^u_Y$ is order preserving for any embedding $f:X\rightarrow Y$. Let us consider the case where $a<_{\mathbf L^u_X}b$ holds because of $\suppl_X(a)=\suppl_X(b)$ and~$\code_X(\suppl_X(a),a)<\code_X(\suppl_X(b),b)$. The functoriality of $\suppl$ yields
 \begin{equation*}
  \suppl_Y(\mathbf L^u_f(a))=[f]^{<\omega}(\suppl_X(a))=[f]^{<\omega}(\suppl_X(b))=\suppl_Y(\mathbf L^u_f(b)).
 \end{equation*}
 Using Theorem~\ref{thm:enumerations-L} we obtain
\begin{multline*}
\code_Y(\suppl_Y(\mathbf L_f^u(a)),\mathbf L_f^u(a))=\code_Y([f]^{<\omega}(\suppl_X(a)),\mathbf L_f^u(a))=\\
=\code_X(\suppl_X(a),a)<\code_X(\suppl_X(b),b)=\\
=\code_Y([f]^{<\omega}(\suppl_X(b)),\mathbf L_f^u(b))=\code_Y(\suppl_Y(\mathbf L_f^u(b)),\mathbf L_f^u(b)).
\end{multline*}
Thus we see $\mathbf L_f^u(a)<_{\mathbf L^u_Y}\mathbf L_f^u(b)$, as required. The conditions for a prae-dilator have been established in Proposition~\ref{prop:constructible-hierarchy-functorial}. Together with Lemma~\ref{less-L-linear-well} we can conclude that $(\mathbf L^u,\suppl)$ is a dilator.
\end{proof}

In a sense, the theorem does not touch on the core of the constructible hierarchy: It considers the elements of $\mathbf L^u_X$ from a purely syntactical perspective and does not explain the significance of the formula $\varphi$ in a term of the form $\{x\in L^u_s\,|\,\varphi(x,\vec a)\}$. To give a functorial approach to semantics, we rely on an infinitary verification calculus introduced by J\"ager~\cite{jaeger-KPN,jaeger-kripke-platek}. The latter operates on $\mathbf L^u_X$-formulas, which are defined as set-theoretic formulas with elements of $\mathbf L^u_X$ as parameters (constant symbols). Recall that we only consider formulas in negation normal form. In particular negation is a defined operation and the formula $\neg\neg\varphi$ is syntactically equal to $\varphi$. In the following we adopt notation from Buchholz~\cite{buchholz-local-predicativity}, who attributes it to Tait. The assumption $\{0,1\}\subseteq u\subseteq\mathbf L^u_X$ provides indices for binary connectives.

\begin{definition}\label{def:disj-conj-formula}
Consider a transitive set $u\supseteq\{0,1\}$ and a linear order $X$. To each closed $\mathbf L^u_X$-formula~$\varphi$ we associate a disjunction $\varphi\simeq\bigvee_{a\in\iota_X(\varphi)}\varphi_a$ or a conjunction $\varphi\simeq\bigwedge_{a\in\iota_X(\varphi)}\varphi_a$, which can be infinite. More precisely, $\varphi$ is assigned a type (disjunctive or conjunctive), an index set $\iota_X(\varphi)\subseteq\mathbf L_X^u$, and a sequence of closed $\mathbf L^u_X$-formulas $\varphi_a$ for $a\in\iota_X(\varphi)$. The disjunctive types are
\begingroup
\allowdisplaybreaks
\begin{align*}
 b\in v&\simeq\begin{cases}
                \bigvee\emptyset & \text{if $b$ is a term $v'\in u$ with $v'\notin v$,}\\
                \textstyle\bigvee_{a\in\{v'\in u\,|\,v'\in v\}} a=b & \text{if $b$ is not of the form $v'\in u$,}
               \end{cases}\\
 b\notin v&\simeq\bigvee\emptyset\qquad\text{if $b$ is a term $v'\in u$ with $v'\in v$},\\
 b\in L_s^u&\simeq\textstyle\bigvee_{\suppl_X(a)\lef_X s}a=b,\\
 b\in\{x\in L_s^u\,|\,\varphi(x,\vec c)\}&\simeq\textstyle\bigvee_{\suppl_X(a)\lef_X s}\varphi(a,\vec c)\land a=b,\\
 b_0\neq b_1 & \simeq\textstyle\bigvee_{a\in\{0,1\}}\,\exists_{x\in b_a}x\notin b_{1-a},\\
 \psi_0\lor\psi_1&\simeq\textstyle\bigvee_{a\in\{0,1\}}\psi_a,\\
 \exists_{x\in b}\psi(x)&\simeq\begin{cases}
                                   \parbox[t]{0.65\textwidth}{\parbox[t]{3.5cm}{$\textstyle\bigvee_{a\in\{v'\in u\,|\,v'\in v\}}\psi(a)$}if $b$ is a term $v\in u$,}\\
                                   \parbox[t]{0.65\textwidth}{\parbox[t]{3.5cm}{$\textstyle\bigvee_{\suppl_X(a)\lef_X s}\psi(a)$}if $b$ is a term $L_s^u$,}\\
                                   \parbox[t]{0.65\textwidth}{$\textstyle\bigvee_{\suppl_X(a)\lef_X s}\varphi(a,\vec c)\land\psi(a)$\newline\hspace*{3.5cm}if $b$ is a term $\{y\in L_s^u\,|\,\varphi(y,\vec c)\}$,}
                               \end{cases}\\
 \exists_x\psi(x)&\simeq\textstyle\bigvee_{a\in\mathbf L_X^u}\psi(a).
\end{align*}
\endgroup
The given disjunctive clauses cover precisely one of the formulas $\psi=\neg\neg\psi$ and~$\neg\psi$. To cover the remaining formula we consider $\varphi=\neg\psi$ resp.~$\varphi=\psi$ and stipulate
\begin{equation*}
 \neg\varphi\simeq\textstyle\bigwedge_{a\in\iota_X(\varphi)}\neg\,\varphi_a\qquad\text{if}\qquad\varphi\simeq\textstyle\bigvee_{a\in\iota_X(\varphi)}\varphi_a.
\end{equation*}
More precisely, if the formula $\varphi$ is disjunctive, then the formula $\neg\varphi$ is conjunctive and we have $\iota_X(\neg\varphi)=\iota_X(\varphi)$ as well as $(\neg\varphi)_a=\neg(\varphi_a)$ for $a\in\iota_X(\varphi)$.
\end{definition}

Before we establish any functorial properties, let us show that we have defined a sound and complete calculus for satisfaction in the actual constructible hierarchy. Recall that Definition~\ref{def:interpretation-constructible-terms} constructs a function $\llbracket\cdot\rrbracket:\mathbf L^u_\alpha\rightarrow\mathbb L^u_\alpha$ for each ordinal $\alpha$. When we write $\mathbb L^u_\alpha\vDash\varphi$ for an $\mathbf L^u_\alpha$-formula $\varphi$ we assume that each constant symbol $a\in\mathbf L^u_\alpha$ that occurs in $\varphi$ is interpreted as the set $\llbracket a\rrbracket\in\mathbb L^u_\alpha$.

\begin{proposition}\label{prop:clauses-inf-verification-sound}
 For each transitive set $u\supseteq\{0,1\}$ and each ordinal $\alpha$ we have
 \begin{align*}
  \mathbb L_\alpha^u\vDash\varphi&\,\Leftrightarrow\,\exists_{a\in\iota_\alpha(\varphi)}\,\mathbb L^u_\alpha\vDash\varphi_a\quad\text{for each closed $\mathbf L_\alpha^u$-formula $\varphi\simeq\textstyle\bigvee_{a\in\iota_\alpha(\varphi)}\varphi_a$},\\
  \mathbb L_\alpha^u\vDash\varphi&\,\Leftrightarrow\,\forall_{a\in\iota_\alpha(\varphi)}\,\mathbb L^u_\alpha\vDash\varphi_a\quad\text{for each closed $\mathbf L_\alpha^u$-formula $\varphi\simeq\textstyle\bigwedge_{a\in\iota_\alpha(\varphi)}\varphi_a$}.
 \end{align*}
\end{proposition}
\begin{proof}
 As an example for the disjunctive case we consider a formula $\varphi=\exists_{x\in b}\psi(x)$ with $b=\{y\in L^u_\gamma\,|\,\theta(y)\}$ for some $\gamma<\alpha$. Recall that $\mathbb L^u_\gamma$ is the image of $\mathbf L^u_\gamma$ under the interpretation function $\llbracket\cdot\rrbracket$. For $a\in\mathbf L^u_\alpha$ the condition $\suppl_\alpha(a)\lef\gamma$ is equivalent to $a\in \mathbf L^u_\gamma$. Together with the absoluteness of $\Delta_0$-formulas we obtain
 \begin{equation*}
  \llbracket b\rrbracket=\{y\in\mathbb L^u_\gamma\,|\,\mathbb L^u_\gamma\vDash\theta(y)\}=\{\llbracket a\rrbracket\,|\,a\in\mathbf L^u_\alpha\text{ and }\suppl_\alpha(a)\lef\gamma\text{ and }\mathbb L^u_\alpha\vDash\theta(a)\}.
 \end{equation*}
 This implies the desired equivalence
 \begin{equation*}
  \exists_{x\in\llbracket b\rrbracket}\,\mathbb L^u_\alpha\vDash\psi(x)\,\Leftrightarrow\,\exists_{a\in\mathbf L^u_\alpha}(\suppl_\alpha(a)\lef\gamma\land\mathbb L^u_\alpha\vDash\theta(a)\land\psi(a)).
 \end{equation*}
 The other disjunctive cases are checked in a similar manner. Once they are established, the claim for conjunctive formulas follows by duality.
\end{proof}

Consider the infinitary proof system which allows to deduce an $\mathbf L^u_\alpha$-formula $\varphi$ with $\varphi\simeq\bigwedge_{a\in\iota_\alpha(\varphi)}\varphi_a$ resp.~$\varphi\simeq\bigvee_{a\in\iota_\alpha(\varphi)}\varphi_a$ once one has established the premise~$\varphi_a$ for all resp.~some $a\in\iota_\alpha(\varphi)$. The previous proposition implies that well-founded proofs in this system are sound. Completeness will be needed in the following form:

\begin{theorem}\label{thm:verifications-complete}
 Consider a transitive set $u\supseteq\{0,1\}$ and an ordinal $\alpha$. Assume that $F$ is a set of closed $\mathbf L^u_\alpha$-formulas with the following properties:
 \begin{itemize}
  \item If $F$ contains $\varphi\simeq\bigwedge_{a\in\iota_\alpha(\varphi)}\varphi_a$, then it contains $\varphi_a$ for some $a\in\iota_\alpha(\varphi)$.
  \item If $F$ contains $\varphi\simeq\bigvee_{a\in\iota_\alpha(\varphi)}\varphi_a$, then it contains $\varphi_a$ for all $a\in\iota_\alpha(\varphi)$.
 \end{itemize}
 Then we have $\mathbb L^u_\alpha\nvDash\varphi$ for any formula $\varphi\in F$.
\end{theorem}
\begin{proof}
 Following J\"ager~\cite{jaeger-KPN} we will assign an ordinal height $\hth(\varphi)$ to each closed $\mathbf L^u_\alpha$-formula $\varphi$. This assignment will satisfy the condition
 \begin{equation*}
  \hth(\varphi_a)<\hth(\varphi)\qquad\text{for all $a\in\iota_\alpha(\varphi)$}.
 \end{equation*}
 The claim that $\varphi\in F$ implies $\mathbb L^u_\alpha\nvDash\varphi$ can then be established by a straightforward induction on $\hth(\varphi)$, using the previous proposition for the induction step. The required height of formulas depends on the height of terms in $\mathbf L^u_\alpha$, which is given by
 \begin{equation*}
 \hth(v)=0\quad\text{for $v\in u$},\qquad\hth(L_\gamma^u)=\hth(\{x\in L_\gamma^u\,|\,\varphi(x,\vec a)\})=\omega\cdot(1+\gamma).
\end{equation*}
 By recursion over the length of closed $\mathbf L^u_\alpha$-formulas we can now define
 \begingroup
\allowdisplaybreaks
 \begin{align*}
 \hth(a\in b)=\hth(a\notin b)&=\max\{\hth(a)+6,\hth(b)+1\},\\
 \hth(a=b)=\hth(a\neq b)&=\max\{\hth(a),\hth(b),5\}+4,\\
 \hth(\varphi_0\lor\varphi_1)=\hth(\varphi_0\land\varphi_1)&=\max\{\hth(\varphi_0),\hth(\varphi_1)\}+1,\\
 \hth(\exists_{x\in b}\varphi(x))=\hth(\forall_{x\in b}\varphi(x))&=\max\{\hth(b),\hth(\varphi(0))+2\},\\
 \hth(\exists_x\varphi(x))=\hth(\forall_x\varphi(x))&=\max\{\omega\cdot(1+\alpha),\hth(\varphi(0))+1\}.
\end{align*}
\endgroup
 It is straightforward to show $\hth(\theta)<\omega\cdot(1+\gamma)$ for any $\gamma<\alpha$ and any $\Delta_0$-formula~$\theta$ with parameters in $\mathbf L^u_\gamma\subseteq\mathbf L^u_\alpha$. As in \cite[Lemma~3]{buchholz-local-predicativity} we also get
 \begin{equation*}
 \hth(\varphi(a))<\max\{\omega\cdot(1+\gamma),\hth(\varphi(0))+1\}\qquad\text{for all $a\in\mathbf L^u_\gamma$},
\end{equation*}
 for any $\gamma\leq\alpha$ and an arbitrary $\mathbf L^u_\alpha$-formula $\varphi=\varphi(x)$. Based on these facts one can verify the desired inequality $\hth(\varphi_a)<\hth(\varphi)$. As an example we consider a formula $\varphi=\exists_{x\in b}\psi(x)$ with $b=\{y\in L^u_\gamma\,|\,\theta(y)\}$ for some $\gamma<\alpha$. Any element $a\in\iota_\alpha(\varphi)$ satisfies~$a\in\mathbf L^u_\gamma$, which indeed yields
\begin{equation*}
 \hth(\varphi_a)=\max\{\hth(\theta(a))+1,\hth(\psi(a))+1\}<\max\{\omega\cdot(1+\gamma),\hth(\psi(0))+2\}=\hth(\varphi).
\end{equation*}
Once all disjunctive cases are checked, the conjunctive cases follow by duality. Then an induction on $\hth(\varphi)$ yields the claim of the theorem, as explained above.
\end{proof}

To conclude this section we show that our verification calculus is functorial. Given an $\mathbf L^u_X$-formula $\varphi$ and an order embedding $f:X\rightarrow Y$, we write $\varphi[f]$ for the $\mathbf L^u_Y$-formula that results from $\varphi$ by replacing any parameter $a\in\mathbf L^u_X$ in $\varphi$ by the parameter $\mathbf L^u_f(a)\in\mathbf L^u_Y$. Observe that negation is natural: For any $\mathbf L^u_X$-formula $\varphi$ and any order embedding $f:X\rightarrow Y$ we have
\begin{equation*}
 (\neg\varphi)[f]=\neg(\varphi[f]).
\end{equation*}
Substitution enjoys a similar property: Consider an $\mathbf L^u_X$-formula $\varphi(x)$ and write $\varphi[f](x)$ to indicate the free variable of this formula. Then we have
\begin{equation*}
 \varphi(a)[f]=\varphi[f](\mathbf L^u_f(a))
\end{equation*}
for any term $a\in\mathbf L^u_X$. The following completes our functorial approach to semantics:

\begin{theorem}\label{thm:conj-disj-natural}
Let $u\supseteq\{0,1\}$ be a transitive set. For any embedding $f:X\rightarrow Y$ the $\mathbf L_X^u$-formula $\varphi$ and the $\mathbf L_Y^u$-formula $\varphi[f]$ have the same type (disjunctive or conjunctive), we have
\begin{equation*}
a\in\iota_X(\varphi)\quad\Leftrightarrow\quad\mathbf L_f^u(a)\in\iota_Y(\varphi[f])
\end{equation*}
for all $a\in\mathbf L_X^u$, and $\varphi_a[f]=\varphi[f]_{\mathbf L_f^u(a)}$ holds for all $a\in\iota_X(\varphi)$.
\end{theorem}
\begin{proof}
 As a representative example we consider a disjunctive formula $\varphi=\exists_{x\in b}\psi(x)$ with $b=\{y\in L^u_s\,|\,\theta(y)\}$. The formula $\varphi[f]=\exists_{x\in\mathbf L^u_f(b)}\psi[f](x)$ is disjunctive as well. In view of $\mathbf L^u_f(b)=\{y\in L^u_{f(s)}\,|\,\theta[f](y)\}$ the naturality of $\suppl$ yields
 \begin{multline*}
  a\in\iota_X(\varphi)\,\Leftrightarrow\,\suppl_X(a)\lef_X s\,\Leftrightarrow\,[f]^{<\omega}(\suppl_X(a))\lef_Y f(s)\,\Leftrightarrow\\
  \Leftrightarrow\,\suppl_Y(\mathbf L_f^u(a))\lef_Y f(s)\,\Leftrightarrow\,\mathbf L_f^u(a)\in\iota_Y(\varphi[f]).
 \end{multline*}
 Using the naturality of substitution we also get
 \begin{equation*}
  \varphi_a[f]=(\theta(a)\land\psi(a))[f]=\theta[f](\mathbf L_f^u(a))\land\psi[f](\mathbf L_f^u(a))=\varphi[f]_{\mathbf L_f^u(a)}.
 \end{equation*}
 The other disjunctive cases are verified similarly. The claim for conjunctive formulas follows by duality and the naturality of negation (cf.~\cite[Section~3.1]{freund-thesis}).
\end{proof}

\section{From Search Trees to Admissible Sets}\label{sect:deduct-admissible}

In the introduction of this paper we have explained how admissible sets can be constructed via Sch\"utte's method of deduction chains (search trees). The details of the construction are worked out in the present section.

Recall that admissible sets are defined as transitive models of Kripke-Platek set theory (with infinity). The crucial axiom scheme of this theory is $\Delta_0$-collection, which has the form
\begin{equation*}
 \forall_{z_1,\dots,z_k}\forall_v(\forall_{x\in v}\exists_y\theta(x,y,z_1,\dots,z_k)\rightarrow\exists_w\forall_{x\in v}\exists_{y\in w}\theta(x,y,z_1,\dots,z_k))
\end{equation*}
with a $\Delta_0$-formula $\theta$. We call $k$ the number of parameters of the given instance of collection. For technical reasons it will be convenient to restrict this number:

\begin{lemma}\label{lem:parameters-collection-restricted}
 There is a number $C$ such that the following holds for any transitive set $u\ni\omega$ and any limit ordinal $\alpha$: If $\mathbb L^u_\alpha$ satisfies all instances of $\Delta_0$-collection with at most $C$ parameters, then it is an admissible set.
\end{lemma}
\begin{proof}
 Apart from collection and infinity, the Kripke-Platek axioms hold in any limit stage of the relativized constructible hierarchy (cf.~\cite[Exercise~II.5.16]{barwise-admissible}). The axiom of infinity is covered by the assumption $\omega\in u\subseteq\mathbb L^u_\alpha$. It is known that finitely many instances of $\Delta_0$-collection suffice to deduce all other instances of this axiom scheme, using a truth definition for $\Delta_0$-formulas. In particular we may place a bound on the number of parameters in the collection axioms.
\end{proof}

The method of deduction chains will allow us to search for an ordinal $\alpha$ such that~$\mathbb L^u_\alpha$ satisfies a given collection of formulas. In order to ensure the assumptions of the previous lemma we consider the following:

\begin{definition}\label{def:enum-axioms}
 Fix a number $C$ as in Lemma~\ref{lem:parameters-collection-restricted} and an enumeration $(\ax_n)_{n\geq 1}$ of all instances of $\Delta_0$-collection with at most $C$ parameters. Furthermore, let $\ax_0$ be the formula $\forall_x\exists_y\,y=x\cup\{x\}$.
\end{definition}

If the set $\mathbb L^u_\alpha$ satisfies $\ax_0$, then its height $o(\mathbb L^u_\alpha)=\mathbb L^u_\alpha\cap\ordi$ is a limit ordinal or zero. Since we have $o(\mathbb L^u_\alpha)=o(u)+\alpha$ by Lemma~\ref{lem:height-L}, the same must hold for~$\alpha$. The case $\alpha=0$ can be excluded if we assume that $o(u)$ is a successor ordinal. Thus we obtain the following consequence of Lemma~\ref{lem:parameters-collection-restricted}:

\begin{corollary}\label{cor:coll-yields-admissible}
 Consider an ordinal $\alpha$ and a transitive set $u\ni\omega$ such that $o(u)$ is a successor ordinal. If $\mathbb L^u_\alpha\vDash\ax_n$ holds for all $n\in\omega$, then $\mathbb L^u_\alpha$ is an admissible set.
\end{corollary}

Let us collect various properties of the parameter $u$ that have been used so far:

\begin{assumption}\label{ass:properties-u}
 Throughout the following we fix a transitive and countable set $u$ with given enumeration $u=\{u_i\,|\,i\in\omega\}$. We assume that its height is a successor ordinal $o(u)>\omega$. The assumptions on $u$ will be discharged in the proof of Theorem~\ref{thm:main-abstract}.
\end{assumption}

In particular $o(u)>\omega$ ensures $\{0,1\}\subseteq u\subseteq\mathbf L^u_X$, as required in Definition~\ref{def:disj-conj-formula}. For our construction of search trees we need some further terminology: Let $X^{<\omega}$ be the set of finite sequences with entries in a given set $X$. As usual we obtain a tree if we order the sequences in $X^{<\omega}$ by end extension. Given a sequence $\sigma=\langle\sigma_0,\dots,\sigma_{n-1}\rangle\in X^{<\omega}$ and an element $x\in X$ we write $\sigma^\frown x=\langle\sigma_0,\dots,\sigma_{n-1},x\rangle$; for $k\leq n=\len(\sigma)$ we put $\sigma\!\restriction\!k=\langle\sigma_0,\dots,\sigma_{k-1}\rangle$. In the previous section we have defined functorial term systems $\mathbf L^u_X$ that represent the constructible hierarchy along an arbitrary order $X$, together with support functions $\suppl_X:\mathbf L^u_X\rightarrow[X]^{<\omega}$. Let us now consider the tree $(\mathbf L^u_X)^{<\omega}$ and the functions
\begin{gather*}
 \supps_X:(\mathbf L^u_X)^{<\omega}\rightarrow[X]^{<\omega},\\
 \supps_X(\langle a_0,\dots,a_{n-1}\rangle)=\suppl_X(a_0)\cup\dots\cup\suppl_X(a_{n-1}).
\end{gather*}
For each order $X$ we will define a search tree $S^u_X\subseteq(\mathbf L^u_X)^{<\omega}$. The nodes of $S^u_X$ will be labelled by $\mathbf L^u_X$-sequents, which are defined as finite sequences of closed $\mathbf L^u_X$-formulas (cf.~the previous section). The intended interpretation of a sequent $\Gamma=\langle\varphi_0,\dots,\varphi_{n-1}\rangle$ is the disjunction $\varphi_0\lor\dots\lor\varphi_{n-1}$. Thus the empty sequent~$\langle\rangle$ is a canonical representation of falsity. In the context of sequents we write $\Gamma,\varphi$ rather than $\Gamma^\frown\varphi$. The order of formulas in a sequent will be crucial for some purposes but irrelevant for others: In the latter case we write $\varphi\in\Gamma$ resp.~$\Gamma\subseteq\Gamma'$ to express that $\varphi$ is an entry of $\Gamma$ and that $\varphi\in\Gamma'$ holds for all $\varphi\in\Gamma$. As a final ingredient for our search trees we fix a surjection
\begin{equation*}
 \omega\ni n\mapsto\langle\pi_0(n),\pi_1(n),\pi_2(n)\rangle\in\omega\times\omega\times\omega
\end{equation*}
with $\pi_i(n)\leq n$. Now we can describe the promised construction:

\begin{definition}\label{def:search-tree}
For each order $X$ we define a tree $S^u_X\subseteq(\mathbf L_X^u)^{<\omega}$ and a labelling $l_X:S^u_X\rightarrow\text{``$\mathbf L_X^u$-sequents''}$ by recursion over~$\sigma\in(\mathbf L_X^u)^{<\omega}$: In the base case we set
\begin{equation*}
 \langle\rangle\in S_X^u\quad\text{and}\quad l_X(\langle\rangle)=\langle\rangle.
\end{equation*}
In the step we distinguish odd and even stages: For $\sigma\in S^u_X$ of length $2n$ we set
\begin{equation*}
 \sigma^\frown a\in S_X^u\Leftrightarrow a=0\quad\text{and}\quad l_X(\sigma^\frown 0)=l_X(\sigma),\neg\ax_n.
\end{equation*}
If $\sigma\in S^u_X$ has length $2n+1$, then we write $\varphi$ for the $\pi_0(n)$-th formula in the sequent~$l_X(\sigma)$. In case $\varphi\simeq\bigwedge_{a\in\iota_X(\varphi)}\varphi_a$ is conjunctive (cf.~Definition~\ref{def:disj-conj-formula}) we set
\begin{equation*}
 \sigma^\frown a\in S_X^u\Leftrightarrow a\in\iota_X(\varphi)\quad\text{and}\quad l_X(\sigma^\frown a)=l_X(\sigma),\varphi_a.
\end{equation*}
In case $\varphi\simeq\bigvee_{a\in\iota_X(\varphi)}\varphi_a$ is disjunctive we compute
\begin{equation*}
 b=\en_X(\supps_X(\sigma\!\restriction\!\pi_1(n)),\pi_2(n)),
\end{equation*}
using the enumeration function from Theorem~\ref{thm:enumerations-L}. Then we set
\begin{equation*}
 \sigma^\frown a\in S_X^u\Leftrightarrow a=0\quad\text{and}\quad l_X(\sigma^\frown 0)=\begin{cases}
                                                                                       l_X(\sigma),\varphi_b & \text{if $b\in\iota_X(\varphi)$},\\
                                                                                       l_X(\sigma)           & \text{otherwise}.
                                                                                      \end{cases}
\end{equation*}
\end{definition}

For a function $f:\omega\rightarrow\mathbf L_X^u$ we write $f\!\restriction\!n=\langle f(0),\dots,f(n-1)\rangle\in(\mathbf L_X^u)^{<\omega}$ and
\begin{equation*}
 X_f=\bigcup_{n\in\omega}\supps_X(f\!\restriction\!n)=\bigcup_{n\in\omega}\suppl_X(f(n))\subseteq X.
\end{equation*}
If $f\!\restriction\!n\in S^u_X$ holds for all $n\in\omega$, then $f$ is called a branch of $S^u_X$. One of our main results shows that such a branch yields a transitive model of set theory:

\begin{theorem}
 Assume that $f$ is a branch of the search tree $S^u_X$, for some linear order $X$. If $X_f$ is well-founded with order-type $\alpha$, then $\mathbb L^u_\alpha$ is an admissible set.
\end{theorem}
\begin{proof}
 Let $h:\alpha\rightarrow X$ be the unique embedding with image $X_f$. We will verify that
 \begin{equation*}
  F=\{\varphi\text{ is an $\mathbf L^u_\alpha$-formula}\,|\,\varphi[h]\in l_X(f\!\restriction\!m)\text{ for some $m\in\omega$}\}
 \end{equation*}
 satisfies the assumptions of Theorem~\ref{thm:verifications-complete}. Assuming that this is the case, we can conclude as follows: The construction of $S^u_X$ ensures $\neg\ax_n\in l_X(f\!\restriction\!(2n+1))$ for any axiom from Definition~\ref{def:enum-axioms}. Since these axioms contain no parameters we have $\neg\ax_n[h]=\neg\ax_n$ and thus $\neg\ax_n\in F$. Using Theorem~\ref{thm:verifications-complete} we obtain $\mathbb L^u_\alpha\nvDash\neg\ax_n$ for all $n\in\omega$. The claim follows by Corollary~\ref{cor:coll-yields-admissible}. To establish the required properties of $F$, let us first consider a conjunctive $\mathbf L^u_\alpha$-formula $\varphi\simeq\bigwedge_{a\in\iota_\alpha(\varphi)}\varphi_a\in F$. Assume that $\varphi[h]$ is the $k$-th formula in $l_X(f\!\restriction\!m)$. Pick a number $n$ with $\pi_0(n)=k$ and $\pi_1(n)=m$, which ensures $m\leq n$. Considering the definition of $S^u_X$, we see that $\varphi[h]$ remains the $k$-th formula in $l_X(f\!\restriction\!(2n+1))$. Theorem~\ref{thm:conj-disj-natural} shows that $\varphi[h]$ is conjunctive. Since $f$ is a branch we have
\begin{equation*}
f\!\restriction\!(2n+2)=f\!\restriction\!(2n+1)^\frown f(2n+1)\in S^u_X.
\end{equation*}
The definition of $S^u_X$ yields $f(2n+1)\in\iota_X(\varphi[h])$ and $\varphi[h]_{f(2n+1)}\in l_X(f\!\restriction\!(2n+2))$. In view of $\suppl_X(f(2n+1))\subseteq X_f=\rng(h)$ Proposition~\ref{prop:constructible-hierarchy-functorial} provides an $a\in\mathbf L^u_\alpha$ with~$f(2n+1)=\mathbf L^u_h(a)$. By Theorem~\ref{thm:conj-disj-natural} we get $a\in\iota_\alpha(\varphi)$ and~$\varphi_a[h]=\varphi[h]_{f(2n+1)}$. This implies $\varphi_a\in F$, as required. It remains to consider a disjunctive $\mathbf L^u_\alpha$-formula $\varphi\simeq\bigvee_{a\in\iota_\alpha(\varphi)}\varphi_a$ that lies in $F$: Assume that $\varphi[h]$ is the $k_0$-th formula in $l_X(f\!\restriction\!m)$. For any $a\in\iota_\alpha(\varphi)\subseteq\mathbf L^u_\alpha$ we have $\suppl_X(\mathbf L^u_h(a))=[h]^{<\omega}(\suppl_\alpha(a))\subseteq X_f$. Thus we can pick a $k_1\geq m$ with $\suppl_X(\mathbf L^u_h(a))\subseteq\supps_X(f\!\restriction\!k_1)$. By Theorem~\ref{thm:enumerations-L} we have
\begin{equation*}
\mathbf L^u_h(a)=\en_X(\supps_X(f\!\restriction\!k_1),k_2)
\end{equation*}
for $k_2=\code_X(\supps_X(f\!\restriction\!k_1),\mathbf L^u_h(a))$. Now we pick a number $n$ with $\pi_i(n)=k_i$. In particular we have $m\leq n$, so that $\varphi[h]$ is the $\pi_0(n)$-th formula in $l_X(f\!\restriction\!(2n+1))$. Theorem~\ref{thm:conj-disj-natural} ensures that $\varphi[h]$ is disjunctive and that we have $\mathbf L^u_h(a)\in\iota_X(\varphi[h])$. Then the definition of $S^u_X$ yields $\varphi_a[h]=\varphi[h]_{\mathbf L^u_h(a)}\in l_X(f\!\restriction\!(2n+2))$. We have thus established $\varphi_a\in F$ for an arbitrary $a\in\iota_\alpha(\varphi)$.
\end{proof}

Recall that we work in the meta theory $\atrs$, which includes axiom beta. This axiom appears to be necessary if one wants to draw the following conclusion:

\begin{corollary}\label{cor:branch-to-admissible}
 If there is a well-order $X$ such that the search tree $S^u_X$ has a branch, then there is an admissible set $\mathbb A$ with $u\subseteq\mathbb A$.
\end{corollary}
\begin{proof}
 Consider a well-order $X$ and a branch $f$ of $S^u_X$. The suborder $X_f\subseteq X$ is well-founded as well. In the presence of axiom beta we can consider its order-type~$\alpha$. The previous theorem tells us that $\mathbb L^u_\alpha$ is an admissible set. In view of Definition~\ref{def:interpretation-constructible-terms} we have $u\subseteq\mathbb L^u_\alpha$.
\end{proof}

In the rest of this section we show that our construction of search trees is functorial. According to Definition~\ref{def:constructible-morphisms} any embedding $f:X\rightarrow Y$ yields a map~$\mathbf L^u_f:\mathbf L^u_X\rightarrow\mathbf L^u_Y$. On the level of finite sequences we can consider the function
\begin{gather*}
 S^u_f:(\mathbf L^u_X)^{<\omega}\rightarrow(\mathbf L^u_Y)^{<\omega},\\
 S^u_f(\langle a_0,\dots,a_{n-1}\rangle)=\langle \mathbf L^u_f(a_0),\dots,\mathbf L^u_f(a_{n-1})\rangle.
\end{gather*}
From Proposition~\ref{prop:constructible-hierarchy-functorial} we know that $\suppl:\mathbf L^u\Rightarrow[\cdot]^{<\omega}$ is a natural transformation. As a straightforward consequence, the maps $\supps_X:(\mathbf L^u_X)^{<\omega}\rightarrow[X]^{<\omega}$ defined above are natural as well. Let us establish the crucial functorial result:

\begin{proposition}\label{prop:search-trees-embeddings}
 For any embedding $f:X\rightarrow Y$ and any $\sigma\in(\mathbf L^u_X)^{<\omega}$ we have
 \begin{equation*}
  \sigma\in S^u_X\quad\Leftrightarrow\quad S^u_f(\sigma)\in S^u_Y.
 \end{equation*}
\end{proposition}
\begin{proof}
 In the previous section we have seen that each $\mathbf L^u_X$-formula $\varphi$ can be transformed into an $\mathbf L^u_Y$-formula $\varphi[f]$. Given an $\mathbf L^u_X$-sequent $\Gamma=\langle\varphi_0,\dots,\varphi_{n-1}\rangle$ we write~$\Gamma[f]=\langle\varphi_0[f],\dots,\varphi_{n-1}[f]\rangle$ for the corresponding $\mathbf L^u_Y$-sequent. The claim of the proposition can be established by induction on $\sigma$ if one simultaneously checks
 \begin{equation*}
  l_X(\sigma)[f]=l_Y(S^u_f(\sigma))
 \end{equation*}
 in case $\sigma\in S^u_X$. The induction step is most interesting for a sequence $\sigma\in S^u_X$ with~$\len(\sigma)=2n+1$. Let $\varphi$ be the $\pi_0(n)$-th formula in $l_X(\sigma)$. The simultaneous induction hypothesis ensures that $\varphi[f]$ is the $\pi_0(n)$-th formula in $l_Y(S^u_f(\sigma))$. By Theorem~\ref{thm:conj-disj-natural} the formulas $\varphi$ and $\varphi[f]$ have the same type. First assume that they are conjunctive. Invoking the construction of search trees and Theorem~\ref{thm:conj-disj-natural} we get
 \begin{equation*}
  \sigma^\frown a\in S^u_X\,\Leftrightarrow\, a\in\iota_X(\varphi)\,\Leftrightarrow\,\mathbf L^u_f(a)\in\iota_Y(\varphi[f])\,\Leftrightarrow\, S^u_f(\sigma^\frown a)=S^u_f(\sigma)^\frown\mathbf L^u_f(a)\in S^u_Y,
 \end{equation*}
 as required. If we have $\sigma^\frown a\in S^u_X$, then Theorem~\ref{thm:conj-disj-natural} also yields
\begin{equation*}
 l_X(\sigma^\frown a)[f]=l_X(\sigma)[f],\varphi_a[f]=l_Y(S^u_f(\sigma)),\varphi[f]_{\mathbf L^u_f(a)}=l_Y(S^u_f(\sigma)^\frown\mathbf L^u_f(a)).
\end{equation*}
Now assume that $\varphi$ and $\varphi[f]$ are disjunctive. Recall that we have $0\in o(u)\subseteq u$ by Assumption~\ref{ass:properties-u}. In view of Definition~\ref{def:constructible-morphisms} we can conclude
\begin{equation*}
 \sigma^\frown a\in S^u_X\,\Leftrightarrow\, a=0\,\Leftrightarrow\,\mathbf L^u_f(a)=0\,\Leftrightarrow\,S^u_f(\sigma^\frown a)=S^u_f(\sigma)^\frown\mathbf L^u_f(a)\in S^u_Y.
\end{equation*}
 Following the construction of search trees we consider the $\mathbf L^u_X$-term
 \begin{equation*}
  b=\en_X(\supps_X(\sigma\!\restriction\!\pi_1(n)),\pi_2(n)).
 \end{equation*}
 Crucially, Theorem~\ref{thm:enumerations-L} and the naturality of $\supps$ allow us to infer
 \begin{equation*}
  \mathbf L^u_f(b)=\en_Y([f]^{<\omega}\circ\supps_X(\sigma\!\restriction\!\pi_1(n)),\pi_2(n))=\en_Y(\supps_Y(S^u_f(\sigma)\!\restriction\!\pi_1(n)),\pi_2(n)).
 \end{equation*}
 Also note that $b\in\iota_X(\varphi)$ is equivalent to $\mathbf L^u_f(b)\in\iota_Y(\varphi[f])$, by Theorem~\ref{thm:conj-disj-natural}. If these equivalent statements are satisfied, then the construction of search trees yields
 \begin{equation*}
  l_X(\sigma^\frown 0)[f]=l_X(\sigma)[f],\varphi_b[f]=l_Y(S^u_f(\sigma)),\varphi[f]_{\mathbf L^u_f(b)}=l_Y(S^u_f(\sigma^\frown 0)).
 \end{equation*}
 The remaining cases are straightforward. Concerning the induction step for a sequence of even length, we point out that the axioms listed in Definition~\ref{def:enum-axioms} are formulas of pure set theory, so that we have $\neg\ax_n[f]=\neg\ax_n$.
\end{proof}

In view of the proposition we may write $S^u_f:S^u_X\rightarrow S^u_Y$ and $\supps_X:S^u_X\rightarrow[X]^{<\omega}$ for the restrictions of these functions to $S^u_X\subseteq(\mathbf L^u_X)^{<\omega}$. To consider $S^u_X$ as a linear order we recall that Lemma~\ref{less-L-linear-well} provides an order relation $<_{\mathbf L^u_X}$ on $\mathbf L^u_X$. As usual, the Kleene-Brouwer order on the tree $S^u_X\subseteq(\mathbf L^u_X)^{<\omega}$ is defined by
\begin{equation*}
\sigma_0<_{S^u_X}\sigma_1\quad\Leftrightarrow\quad\begin{cases}
                                                   \text{either $\sigma_0$ is a proper end extension of $\sigma_1$,}\\
                                                   \text{or we have $\sigma_i=\sigma^\frown a_i^\frown\sigma_i'$ with $a_0<_{\mathbf L^u_X} a_1$.}
                                                  \end{cases} 
\end{equation*}
To understand the second clause, note that $a_0$ and $a_1$ are the first entries on which the sequences $\sigma_0$ and~$\sigma_1$ disagree. If $X$ is well-founded, then Lemma~\ref{less-L-linear-well} ensures that $(\mathbf L^u_X,<_{\mathbf L^u_X})$ is well-founded as well. In this case $(S^u_X,<_{S^u_X})$ is well-founded if the tree $S^u_X$ has no branch (this well-known property of the Kleene-Brouwer order can be established in our meta theory~$\atrs$, as verified in~\cite[Lemma~3.3.4]{freund-thesis}). Let us now summarize the functorial properties of our search trees:

\begin{proposition}\label{prop:search-trees-prae-dilator}
 The maps $X\mapsto(S^u_X,<_{S^u_X})$, $f\mapsto S^u_f$ and $X\mapsto\supps_X$ form a prae-dilator.
\end{proposition}
\begin{proof}
From Theorem~\ref{thm:constructible-dilator} we know that $(\mathbf L^u,\suppl)$ is a (prae-)dilator. One easily infers that $S^u$ is an endofunctor of linear orders and that $\supps:S^u\Rightarrow[\cdot]^{<\omega}$ is a natural transformation. It remains to show that the functions $\supps_X:S^u_X\rightarrow[X]^{<\omega}$ compute supports in the sense of Definition~\ref{def:prae-dilator}: Consider an arbitrary $\sigma\in S^u_X$ and write $\iota_\sigma:\supps_X(\sigma)\hookrightarrow X$ for the inclusion. For each entry $a$ of the sequence~$\sigma$ we have $\suppl_X(a)\subseteq\supps_X(\sigma)=\rng(\iota_\sigma)$. By Proposition~\ref{prop:constructible-hierarchy-functorial} it follows that $a$ lies in the range of the function $\mathbf L^u_{\iota_\sigma}:\mathbf L^u_{\supps_X(\sigma)}\rightarrow\mathbf L^u_X$. If we apply this argument to all entries of $\sigma$, then we obtain a sequence $\sigma_0\in(\mathbf L^u_{\supps_X(\sigma)})^{<\omega}$ with $\sigma=S^u_{\iota_\sigma}(\sigma_0)$. Crucially, Proposition~\ref{prop:search-trees-embeddings} ensures $\sigma_0\in S^u_{\supps_X(\sigma)}$. Thus $\sigma$ lies in the range of the restricted function $S^u_{\iota_\sigma}:S^u_{\supps_X(\sigma)}\rightarrow S^u_X$, as required.
\end{proof}

The following byproduct of our investigation will be needed later:

\begin{corollary}\label{cor:operator-control-search-trees}
 Consider a node $\sigma\in S^u_X$ and a formula $\varphi\in l_X(\sigma)$. If the parameter $a\in\mathbf L^u_X$ occurs in $\varphi$, then we have $\suppl_X(a)\subseteq\supps_X(\sigma)$.
\end{corollary}
\begin{proof}
 As in the previous proof we write $\sigma=S^u_{\iota_\sigma}(\sigma_0)$ with $\sigma_0\in S^u_{\supps_X(\sigma)}$. The proof of Proposition~\ref{prop:search-trees-embeddings} yields an $\mathbf L^u_{\supps_X(\sigma)}$-formula $\varphi_0$ with $\varphi=\varphi_0[\iota_\sigma]$. Thus we have $a=\mathbf L^u_{\iota_\sigma}(a_0)$ for some $a_0\in\mathbf L^u_{\supps_X(\sigma)}$. Using the naturality of $\suppl$ we get
 \begin{equation*}
  \suppl_X(a)=[\iota_\sigma]^{<\omega}(\suppl_{\supps_X(\sigma)}(a_0))\subseteq\rng(\iota_\sigma)=\supps_X(\sigma),
 \end{equation*}
as promised.
\end{proof}

Combining functorial and non-functorial results we obtain the following:

\begin{theorem}\label{thm:admissible-set-dilator}
 One of the following statements must hold:
 \begin{enumerate}[label=(\roman*)]
  \item There is an admissible set $\mathbb A$ with $u\subseteq\mathbb A$.
  \item The construction of search trees results in a dilator $(S^u,\supps)$.
 \end{enumerate}
\end{theorem}
\begin{proof}
 Assume that statement~(ii) fails. Since $(S^u,\supps)$ is a prae-dilator, this can only happen if $(S^u_X,<_{S^u_X})$ is ill-founded for some well-order $X$. We can conclude that the tree $S^u_X$ has a branch, as explained above. Now Corollary~\ref{cor:branch-to-admissible} yields an admissible set $\mathbb A\supseteq u$, so that statement~(i) holds.
\end{proof}

In the remaining sections we will refute statement~(ii) under the assumption that the abstract Bachmann-Howard principle holds. Thus the latter implies the existence of admissible sets, as required for the crucial implication (iii)$\Rightarrow$(ii) of Theorem~\ref{thm:main-abstract}. Note that we can construct an admissible set $\mathbb A\ni x$ if we start with a suitable set~$u\ni x$.

\section{The \texorpdfstring{$\varepsilon$}{epsilon}-Variant of a Dilator}\label{section:epsilon-variant}

In this section we transform the (prae-)dilator $S^u$ from the previous section into a stronger (prae-)dilator $\varepsilon(S)^u$. The latter will be needed in the ordinal analysis that completes the proof of Theorem~\ref{thm:main-abstract}. The point of the construction becomes clear when we consider an ordinal $\alpha$ that is a Bachmann-Howard fixed point of~$\varepsilon(S)^u$: We will see that $\alpha$ is at least as big as the usual Bachmann-Howard ordinal.

Our first goal is to define an order $\varepsilon(S)^u_X$ for each given order $X$. Similar constructions (without reference to prae-dilators) have been studied by Afshari, Rathjen and Valencia Vizca\'ino~\cite{rathjen-afshari,rathjen-model-bi}. As a preparation we define auxiliary term systems $\varepsilon^0(S)^u_X\supseteq\varepsilon(S)^u_X$. In order to understand the following definition it may help to think of $\Omega$ as a notation for the cardinal~$\aleph_1$. The terms $\mathfrak e_x$ and~$\mathfrak E_\sigma$ represent countable resp.~uncountable $\varepsilon$-numbers.

\begin{definition}
 Given a linear order $(X,<_X)$, the set $\varepsilon^0(S)^u_X$ is inductively generated by the following clauses:
 \begin{enumerate}[label=(\roman*)]
  \item The set $\varepsilon^0(S)^u_X$ contains terms $0$ and $\Omega$. Furthermore it contains terms $\mathfrak e_x$ and $\mathfrak E_\sigma$ for all elements $x\in X$ resp.~$\sigma\in S^u_X$.
  \item If $t_0,\dots,t_n$ are elements of $\varepsilon^0(S)^u_X$, then so is the term $\omega^{t_0}+\cdots+\omega^{t_n}$.
 \end{enumerate}
 The length of terms is given by the function $L:\varepsilon^0(S)^u_X\rightarrow\omega$ with
 \begin{gather*}
  L(0)=L(\Omega)=L(\mathfrak e_x)=L(\mathfrak E_\sigma)=0,\\
  L(\omega^{t_0}+\cdots+\omega^{t_n})=L(t_0)+\dots+L(t_n)+1.
 \end{gather*}
\end{definition}

Intuitively, the cardinal $\aleph_1$ can be represented by $\omega^0+\omega^\Omega$ as well as $\Omega$. To obtain unique representations we single out the terms in Cantor normal form:

\begin{definition}\label{def:eps(S)-order}
 To define $\varepsilon(S)^u_X\subseteq\varepsilon^0(S)^u_X$ and ${<_{\varepsilon(S)^u_X}}\subseteq\varepsilon(S)^u_X\times\varepsilon(S)^u_X$ we decide $r\in\varepsilon(S)^u_X$ and $s<_{\varepsilon(S)^u_X}t$ by simultaneous recursion on $L(r)$ resp.~$L(s)+L(t)$:
 \begin{enumerate}[label=(\roman*)]
  \item We have $\{0,\Omega,\mathfrak e_x,\mathfrak E_\sigma\}\subseteq\varepsilon(S)^u_X$ for arbitrary $x\in X$ resp.~$\sigma\in S^u_X$.
  \item We have $\omega^{r_0}+\cdots+\omega^{r_n}\in\varepsilon(S)^u_X$ if we have $\{r_0,\dots,r_n\}\subseteq\varepsilon(S)^u_X$ and
  \begin{itemize}
   \item either $n=0$ and $r_0$ is not of the form $\Omega,\mathfrak e_x$ or $\mathfrak E_\sigma$,
   \item or $n>0$ and $r_n\leq_{\varepsilon(S)^u_X}\dots\leq_{\varepsilon(S)^u_X} r_0$ (where $s\leq_{\varepsilon(S)^u_X} t$ abbreviates $s<_{\varepsilon(S)^u_X}t\lor s=t$, the latter denoting equality as terms).
  \end{itemize}
 \end{enumerate}
 For $s,t\in\varepsilon(S)^u_X$ we have $s<_{\varepsilon(S)^u_X}t$ if and only if one of the following holds:
 \begin{enumerate}[label=(\roman*')]
  \item We have $s=0$ and $t\neq 0$.
  \item We have $s=\mathfrak e_x$ and
  \begin{itemize}
   \item either $t$ is of the form $\Omega,\mathfrak E_\tau$ or $\mathfrak e_y$ with $x<_X y$,
   \item or we have $t=\omega^{t_0}+\cdots+\omega^{t_n}$ and $s\leq t_0$.
  \end{itemize}
  \item We have $s=\Omega$ and
  \begin{itemize}
   \item either $t$ is of the form $\mathfrak E_\tau$,
   \item or we have $t=\omega^{t_0}+\cdots+\omega^{t_n}$ and $s\leq t_0$.
  \end{itemize}
  \item We have $s=\mathfrak E_\sigma$ and
  \begin{itemize}
   \item either $t$ is of the form $\mathfrak E_\tau$ with $\sigma<_{S^u_X}\tau$,
   \item or we have $t=\omega^{t_0}+\cdots+\omega^{t_n}$ and $s\leq t_0$.
  \end{itemize}
  \item We have $s=\omega^{s_0}+\cdots+\omega^{s_m}$ and
  \begin{itemize}
   \item either $t$ is of the form $\mathfrak e_y,\Omega$ or $\mathfrak E_\tau$ and we have $s_0<_{\varepsilon(S)^u_X}t$,
   \item or we have $t=\omega^{t_0}+\cdots+\omega^{t_n}$ and one of the following holds:
   \begin{itemize}
    \item Either we have $m<n$ and $s_i=t_i$ for all $i\leq m$,
    \item or there is $j\leq\min\{m,n\}$ with $s_j<_{\varepsilon(S)^u_X}t_j$ and $s_i=t_i$ for $i<j$.
   \end{itemize}
  \end{itemize}
 \end{enumerate}
\end{definition}

The following type of result is standard (cf.~\cite[Theorem~14.2]{schuette77}):

\begin{lemma}
 If $X$ is a linear order, then so is $(\varepsilon(S)^u_X,<_{\varepsilon(S)^u_X})$.
\end{lemma}
\begin{proof}
 Refute $s<_{\varepsilon(S)^u_X}s$ by induction on $L(s)$. Then show that $r<_{\varepsilon(S)^u_X}s<_{\varepsilon(S)^u_X}t$ implies $r<_{\varepsilon(S)^u_X}t$, arguing by induction on $L(r)+L(s)+L(t)$. Finally, use induction on $L(s)+L(t)$ to show that we have $s<_{\varepsilon(S)^u_X}t$, $s=t$ or $t<_{\varepsilon(S)^u_X}s$.
\end{proof}

Since $S^u$ is a functor, it transforms each embedding $f:X\rightarrow Y$ of linear orders into an embedding $S^u_f:S^u_X\rightarrow S^u_Y$. Our next goal is to deduce that the term systems $\varepsilon(S)^u_X$ are functorial as well.

\begin{definition}
 Consider an embedding $f:X\rightarrow Y$. By recursion over terms we define a function $\varepsilon(S)^u_f:\varepsilon(S)^u_X\rightarrow\varepsilon^0(S)^u_Y$ with
 \begin{gather*}
  \varepsilon(S)^u_f(0)=0,\quad\varepsilon(S)^u_f(\Omega)=\Omega,\quad \varepsilon(S)^u_f(\mathfrak e_x)=\mathfrak e_{f(x)},\quad\varepsilon(S)^u_f(\mathfrak E_\sigma)=\mathfrak E_{S^u_f(\sigma)},\\
  \varepsilon(S)^u_f(\omega^{t_0}+\cdots+\omega^{t_n})=\omega^{\varepsilon(S)^u_f(t_0)}+\dots+\omega^{\varepsilon(S)^u_f(t_n)}.
 \end{gather*}
\end{definition}

The following result allows us to view $\varepsilon(S)^u_f$ as a function with codomain $\varepsilon(S)^u_Y$ rather than $\varepsilon^0(S)^u_Y$.

\begin{lemma}
 Given any embedding $f:X\rightarrow Y$, the function $\varepsilon(S)^u_f$ is an embedding of $(\varepsilon(S)^u_X,<_{\varepsilon(S)^u_X})$ into $(\varepsilon(S)^u_Y,<_{\varepsilon(S)^u_Y})$.
\end{lemma}
\begin{proof}
By simultaneous induction on $L(r)$ resp.~$L(s)+L(t)$ one sees that $r\in\varepsilon(S)^u_X$ implies $\varepsilon(S)^u_f(r)\in\varepsilon(S)^u_Y$ and that $s<_{\varepsilon(S)^u_X}t$ implies $\varepsilon(S)^u_f(s)<_{\varepsilon(S)^u_Y}\varepsilon(S)^u_f(t)$.
\end{proof}

The prae-dilator $S^u$ comes with support functions $\supps_X:S^u_X\rightarrow[X]^{<\omega}$. In order to turn $\varepsilon(S)^u$ into a prae-dilator we extend them as follows:

\begin{definition}\label{def:suppe}
 For each order $X$ we define a map $\suppe_X:\varepsilon(S)^u_X\rightarrow[X]^{<\omega}$ by
 \begin{gather*}
  \suppe_X(0)=\suppe_X(\Omega)=\emptyset,\quad\suppe_X(\mathfrak e_x)=\{x\},\quad\suppe_X(\mathfrak E_\sigma)=\supps_X(\sigma),\\
  \suppe_X(\omega^{t_0}+\cdots+\omega^{t_n})=\suppe_X(t_0)\cup\dots\cup\suppe_X(t_n).
 \end{gather*}
\end{definition}

Let us summarize the functorial properties of the term systems $\varepsilon(S)^u_X$:

\begin{proposition}
 The maps $X\mapsto(\varepsilon(S)^u_X,<_{\varepsilon(S)^u_X})$, $f\mapsto\varepsilon(S)^u_f$ and $X\mapsto\suppe_X$ form a prae-dilator.
\end{proposition}
\begin{proof}
 We have already seen that $\varepsilon(S)^u$ maps orders to orders and embeddings to embeddings. From Proposition~\ref{prop:search-trees-prae-dilator} we know that $(S^u,\supps)$ is a prae-dilator. By induction over terms it is straightforward to deduce that $\varepsilon(S)^u$ is a functor and that $\suppe:\varepsilon(S)^u\Rightarrow[\cdot]^{<\omega}$ is a natural transformation. It remains to show that any term $s\in\varepsilon(S)^u_X$ lies in the range of $\varepsilon(S)^u_{\iota_s}$, where $\iota_s:\suppe_X(s)\hookrightarrow X$ is the inclusion. This can be established by induction over $s$. The first interesting case is~$s=\mathfrak E_\sigma$: In view of $\suppe_X(s)=\supps_X(\sigma)$ the corresponding property of $\supps_X$ ensures that $\sigma$ lies in the range of~$S^u_{\iota_s}$, say~$\sigma=S^u_{\iota_s}(\sigma_0)$. Then $s=\varepsilon(S)^u_{\iota_s}(\mathfrak E_{\sigma_0})$ lies in the range of $\varepsilon(S)^u_{\iota_s}$, as needed. Let us also consider the case $s=\omega^{s_0}+\dots+\omega^{s_n}$: For each $i\leq n$ we have $\suppe_X(s_i)\subseteq\suppe_X(s)$, which means that the inclusion $\suppe_X(s_i)\hookrightarrow X$ factors through~$\iota_s$. Together with the induction hypothesis we obtain terms $s_i'$ with $s_i=\varepsilon(S)^u_{\iota_s}(s_i')$. Since $\varepsilon(S)^u_{\iota_s}$ is an order embedding, the term $s'=\omega^{s_0'}+\dots+\omega^{s_n'}$ satisfies clause~(ii) of Definition~\ref{def:eps(S)-order}. Again we can conclude that $s=\varepsilon(S)^u_{\iota_s}(s')$ lies in the range of $\varepsilon(S)^u_{\iota_s}$.
\end{proof}

Let us also address the question of well-foundedness:

\begin{theorem}\label{thm:eps-dilator}
 If $S^u$ is a dilator, then so is $\varepsilon(S)^u$.
\end{theorem}
\begin{proof}
 In view of the previous proposition it suffices to show that $\varepsilon(S)^u$ preserves well-foundedness. Given a well-order $X$, we consider the disjoint union
 \begin{equation*}
  Y=X\cup\{\Omega\}\cup S^u_X
 \end{equation*}
 with the expected order (i.e.~we have $x<_Yx'<_Y\Omega<_Y\sigma<_Y\sigma'$ for any elements $x<_Xx'$ of $X$ and any elements $\sigma<_{S^u_X}\sigma'$ of $S^u_X$). The assumption that $S^u$ is a dilator tells us that $S^u_X$ is a well-order. Thus $Y$ is also a well-order. With the above definition of $Y$, our order $\varepsilon(S)^u_X$ coincides with the order $\varepsilon_Y$ considered by Afshari and Rathjen~\cite{rathjen-afshari}. These authors show that the well-foundedness of $\varepsilon_Y$ can be established in the theory $\mathbf{ACA_0^+}$, which is contained in our meta theory~$\atrs$. We remark that the result in~\cite{rathjen-afshari} relies on a syntactic well-ordering proof, which extends Gentzen's proof for the well-foundedness of $\varepsilon_0$. In the presence of axiom beta one can, alternatively, give a semantical argument: Let $\alpha$ be the order-type of the well-order $Y$. It is straightforward to construct an order isomorphism between the term system $\varepsilon_Y$ and the ordinal $\varepsilon_\alpha$ (recall that $\beta\mapsto\varepsilon_\beta$ is the increasing enumeration of the class $\{\gamma\,|\,\omega^\gamma=\gamma\}$).
\end{proof}

Our overall goal is to construct an admissible set $\mathbb A$ with $u\subseteq\mathbb A$. In view of Theorem~\ref{thm:admissible-set-dilator} we may assume that $S^u$ is a dilator. Then the previous theorem ensures that $\varepsilon(S)^u$ is a dilator as well. Invoking the abstract Bachmann-Howard principle from Definition~\ref{def:abstract-bhp} we can justify the following:

\begin{assumption}\label{ass:bachmann-howard-collapse}
 Throughout the following we fix an ordinal $\alpha$ and a Bachmann-Howard collapse $\vartheta:\varepsilon(S)^u_\alpha\rightarrow\alpha$. The assumption that such objects exist will be discharged in the proof of Theorem~\ref{thm:main-abstract}.
\end{assumption}

Similar to the usual construction of the Bachmann-Howard ordinal, the suborders
\begin{equation*}
 Z\cap\Omega=\{s\in Z\,|\,s<_{\varepsilon(S)^u_\alpha}\Omega\}
\end{equation*}
with $Z\subseteq\varepsilon(S)^u_\alpha$ will play a particularly important role.

\begin{lemma}\label{lem:countable-part-epsilon_alpha}
 The order $\varepsilon(S)^u_\alpha\cap\Omega$ is isomorphic to the ordinal~$\varepsilon_\alpha$.
\end{lemma}
\begin{proof}
 Interpret each term $\mathfrak e_\gamma\in\varepsilon(S)^u_\alpha$ with $\gamma\in\alpha$ by the ordinal $\varepsilon_\gamma<\varepsilon_\alpha$. Given interpretations of $s_0,\dots,s_n$, interpret the term $\omega^{s_0}+\dots+\omega^{s_n}$ as an ordinal in Cantor normal form.
\end{proof}

In contrast to the previous lemma, the next result relies on the assumption that~$\alpha$ is a Bachmann-Howard fixed point:

\begin{proposition}\label{prop:alpha-epsilon-number}
 The restriction of $\vartheta$ to $\varepsilon(S)^u_\alpha\cap\Omega\cong\varepsilon_\alpha$ is a fully order preserving map into $\alpha$. Thus we have $\varepsilon_\alpha=\alpha$.
\end{proposition}
\begin{proof}
 A straightforward induction on $L(s)+L(t)$ yields $\suppe_\alpha(s)\leq^{\operatorname{fin}}\suppe_\alpha(t)$ for $s<_{\varepsilon(S)^u_\alpha}t<_{\varepsilon(S)^u_\alpha}\Omega$. With Definition~\ref{def:bachmann-howard-collapse}(ii) we get $\suppe_\alpha(s)\lef\vartheta(t)$. By clause~(i) of the same definition we can infer $\vartheta(s)<\vartheta(t)$.
\end{proof}

We will often omit the isomorphism between $\varepsilon(S)^u_\alpha\cap\Omega$ and $\varepsilon_\alpha=\alpha$, as well as certain subscripts (e.g.~we write $\suppe_\alpha(s)\lef t$, where we have $\suppe_\alpha(s)\subseteq\alpha$ and $t\in\varepsilon(S)^u_\alpha$). The usual construction of the Bachmann-Howard ordinal (see e.g.~\cite[Section~1]{rathjen-weiermann-kruskal}) suggests that the support of $\mathfrak e_\gamma$ should be $\{\varepsilon_\gamma\}$ rather than~$\{\gamma\}$. The difference disappears in view of the following result:

\begin{proposition}\label{prop:collapsing-epsilon}
 We have $\varepsilon_{\vartheta(s)}=\vartheta(s)$ for any $s\in\varepsilon(S)^u_\alpha$ with $\Omega\leq s$.
\end{proposition}
\begin{proof}
We will show that $\vartheta(s)$ is a limit ordinal with $\varepsilon_\gamma<\vartheta(s)$ for any $\gamma<\vartheta(s)$. This is enough to establish the claim, since it yields
\begin{equation*}
 \varepsilon_{\vartheta(s)}=\sup\{\varepsilon_\gamma\,|\,\gamma<\vartheta(s)\}\leq\vartheta(s).
\end{equation*}
First observe $0<\vartheta(s)$: Indeed Definition~\ref{def:bachmann-howard-collapse} yields $\vartheta(0)<\vartheta(s)$, because we have $0<\Omega\leq s$ and $\suppe_\alpha(0)=\emptyset\lef\vartheta(s)$. To conclude it suffices to show that $\gamma<\vartheta(s)$ implies $\varepsilon_\gamma+1<\vartheta(s)$. The isomorphism $\varepsilon_\alpha\cong\varepsilon(S)^u_\alpha\cap\Omega$ identifies $\varepsilon_\gamma+1$ with $\omega^{\mathfrak e_\gamma}+\omega^0$. By the previous proposition we get $\varepsilon_\gamma+1\leq\vartheta(\omega^{\mathfrak e_\gamma}+\omega^0)$, as for any order preserving map between ordinals. Thus it remains to show $\vartheta(\omega^{\mathfrak e_\gamma}+\omega^0)<\vartheta(s)$ for an arbitrary $\gamma<\vartheta(s)$. We observe
 \begin{equation*}
  \suppe_\alpha(\omega^{\mathfrak e_\gamma}+\omega^0)=\{\gamma\}\lef\vartheta(s).
 \end{equation*}
 Together with $\omega^{\mathfrak e_\gamma}+\omega^0<\Omega\leq s$ we get $\vartheta(\omega^{\mathfrak e_\gamma}+\omega^0)<\vartheta(s)$ by Definition~\ref{def:bachmann-howard-collapse}.
\end{proof}

In order to avoid the side condition $\Omega\leq s$ we will work with $\vartheta(\Omega+s)$ at the place of $\vartheta(s)$. Rather than an ad hoc definition of $\Omega+s$, we give a general definition of addition and exponentiation, which will be needed later:

\begin{lemma}\label{lem:ordinal-addition}
 The usual operations of addition and exponentiation on the ordinal $\varepsilon_\alpha\cong\varepsilon(S)^u_\alpha\cap\Omega$ can be extended to functions
 \begin{equation*}
  +:\varepsilon(S)^u_\alpha\times\varepsilon(S)^u_\alpha\rightarrow\varepsilon(S)^u_\alpha\qquad\text{and}\qquad\omega^{(-)}:\varepsilon(S)^u_\alpha\rightarrow\varepsilon(S)^u_\alpha,
 \end{equation*}
 such that the following holds for all $r,s,t\in\varepsilon(S)^u_\alpha$:
 \begin{enumerate}[label=(\alph*)]
  \item If we have $s<t$, then we have $r+s<r+t$ and $\omega^s<\omega^t$.
  \item We have $s\leq s+t$ and $t\leq s+t$, as well as $s\leq\omega^s$.
  \item We have $(r+s)+t=r+(s+t)$.
  \item If we have $s<\omega^r$ and $t<\omega^r$, then we have $s+t<\omega^r$.
  \item We have $\suppe_\alpha(t)\subseteq\suppe_\alpha(s+t)\subseteq\suppe_\alpha(s)\cup\suppe_\alpha(t)$, as well as $\suppe_\alpha(\omega^s)=\suppe_\alpha(s)$.
 \end{enumerate}
\end{lemma}
\begin{proof}
 For $s_{n-1}\leq\dots\leq s_0$ we introduce the notation
 \begin{equation*}
  \omega\langle s_0,\dots,s_{n-1}\rangle=\begin{cases}
                                          0 & \text{if $n=0$},\\
                                          s_0 & \text{if $n=1$ and $s_0$ is of the form $\mathfrak e_x,\Omega$ or $\mathfrak E_\sigma$},\\
                                          \omega^{s_0}+\dots+\omega^{s_{n-1}} & \text{otherwise}.
                                         \end{cases}
 \end{equation*}
 Then we have $\omega\langle s_0,\dots,s_{n-1}\rangle\in\varepsilon(S)^u_\alpha$, and any term in $\varepsilon(S)^u_\alpha$ can be uniquely written in this form. Thus addition can be defined by
 \begin{equation*}
  \omega\langle s_0,\dots,s_{n-1}\rangle+\omega\langle t_0,\dots,t_{m-1}\rangle=\omega\langle s_0,\dots,s_{i-1},t_0,\dots,t_{m-1}\rangle,
 \end{equation*}
 where $i\leq n$ is maximal with $t_0\leq s_{i-1}$ (set $i=0$ if $s_0<t_0$, and $i=n$ if $m=0$). Exponentiation is given by $\omega^s=\omega\langle s\rangle$. The clauses from Definition~\ref{def:eps(S)-order} amount to
 \begin{equation*}
  \omega\langle s_0,\dots,s_{n-1}\rangle<\omega\langle t_0,\dots,t_{m-1}\rangle\,\Leftrightarrow\,\begin{cases}
                                                                                                     \text{either $n<m$ and $s_i=t_i$ for $i<n$},\\[.3ex]
                                                                                                     \parbox[t]{6cm}{or there is $j<\min\{n,m\}$ with $s_j<t_j$ and $s_i=t_i$ for $i<j$.}
                                                                                                  \end{cases}
 \end{equation*}
 One can also observe
 \begin{equation*}
  \suppe_\alpha(\omega\langle s_0,\dots,s_{n-1}\rangle)=\suppe_\alpha(s_0)\cup\dots\cup\suppe_\alpha(s_{n-1}).
 \end{equation*}
 Now it is standard to deduce the properties claimed in the lemma (cf.~\cite[Paragraph~14]{schuette77} as well as~\cite[Lemma~4.2.5]{freund-thesis}).
\end{proof}

The collapse $\vartheta$ provided by Assumption~\ref{ass:bachmann-howard-collapse} can now be modified as follows:

\begin{definition}
 The function $\bar\vartheta:\varepsilon(S)^u_\alpha\rightarrow\varepsilon(S)^u_\alpha\cap\Omega$ is defined by $\bar\vartheta(s)=\mathfrak e_{\vartheta(\Omega+s)}$.
\end{definition}

We recover properties of the usual Bachmann-Howard construction:

\begin{proposition}\label{prop:bar-collapse}
 The following holds for any $s,t\in\varepsilon(S)^u_\alpha$:
 \begin{enumerate}[label=(\alph*)]
  \item If we have $s<t$ and $\suppe_\alpha(s)\lef\bar\vartheta(t)$, then we have $\bar\vartheta(s)<\bar\vartheta(t)$.
  \item We have $\suppe_\alpha(s)\lef\bar\vartheta(s)$.
  \item If we have $\bar\vartheta(s)\leq^{\operatorname{fin}}\suppe_\alpha(t)$, then we have $\bar\vartheta(s)<\bar\vartheta(t)$.
 \end{enumerate}
\end{proposition}

\begin{proof}
 (a) Under the isomorphism $\varepsilon(S)^u_\alpha\cap\Omega\cong\varepsilon_\alpha=\alpha$, the term $\bar\vartheta(t)$ is identified with the ordinal $\varepsilon_{\vartheta(\Omega+t)}$. Using Lemma~\ref{lem:ordinal-addition} and Proposition~\ref{prop:collapsing-epsilon} we obtain
 \begin{equation*}
  \suppe_\alpha(\Omega+s)=\suppe_\alpha(s)\lef\varepsilon_{\vartheta(\Omega+t)}=\vartheta(\Omega+t)
 \end{equation*}
 and $\Omega+s<\Omega+t$. By Definition~\ref{def:bachmann-howard-collapse} we get $\vartheta(\Omega+s)<\vartheta(\Omega+t)$ and thus $\bar\vartheta(s)<\bar\vartheta(t)$.\\
 (b) Invoking Lemma~\ref{lem:ordinal-addition}, Definition~\ref{def:bachmann-howard-collapse} and Proposition~\ref{prop:collapsing-epsilon} we see
 \begin{equation*}
 \suppe_\alpha(s)=\suppe_\alpha(\Omega+s)\lef\vartheta(\Omega+s)=\varepsilon_{\vartheta(\Omega+s)}=\bar\vartheta(s).
 \end{equation*}
 (c) This is an immediate consequence of part (b).
\end{proof}

Let us collect more facts about supports and collapsing values:

\begin{lemma}\label{lem:collapse-support}
 The following holds for any $s,t\in\varepsilon(S)^u_\alpha$:
 \begin{enumerate}[label=(\alph*)]
  \item We have $\suppe_\alpha(s)\leqf s$.
  \item We have $\suppe_\alpha(\bar\vartheta(t))=\{\bar\vartheta(t)\}$.
  \item If we have $\suppe_\alpha(s)\lef\bar\vartheta(t)$ and $s<\Omega$, then we have $s<\bar\vartheta(t)$.
 \end{enumerate}
\end{lemma}
\begin{proof}
 (a) We argue by induction on the term $s$. For $s=\mathfrak e_\gamma$ we observe
 \begin{equation*}
  \suppe_\alpha(s)=\{\gamma\}\leqf\varepsilon_\gamma=s.
 \end{equation*}
 For $s=\omega^{s_0}+\dots+\omega^{s_n}$ we note that any $r\in\suppe_\alpha(s)$ lies in some set~$\suppe_\alpha(s_i)$. By induction hypothesis we get $r\leq s_i\leq s_0\leq s$. For $\Omega\leq s$ the claim is trivial, since we have $\suppe_\alpha(s)\subseteq\alpha\cong\varepsilon(S)^u_\alpha\cap\Omega$.\\
 (b) Using Proposition~\ref{prop:collapsing-epsilon} we get
 \begin{equation*}
  \suppe_\alpha(\bar\vartheta(t))=\suppe_\alpha(\mathfrak e_{\vartheta(\Omega+t)})=\{\vartheta(\Omega+t)\}=\{\varepsilon_{\vartheta(\Omega+t)}\}=\{\bar\vartheta(t)\}.
 \end{equation*}
 (c) We argue by induction on $s$. First assume $s=\mathfrak e_\gamma$. By Proposition~\ref{prop:collapsing-epsilon} we get
 \begin{equation*}
  \{\gamma\}=\suppe_\alpha(s)\lef\bar\vartheta(t)=\varepsilon_{\vartheta(\Omega+t)}=\vartheta(\Omega+t).
 \end{equation*}
 This implies $s=\mathfrak e_\gamma<\mathfrak e_{\vartheta(\Omega+t)}=\bar\vartheta(t)$. Now consider a term $s=\omega^{s_0}+\dots+\omega^{s_n}$. From the induction hypothesis we obtain $s_0<\bar\vartheta(t)=\mathfrak e_{\vartheta(\Omega+t)}$. Invoking clause~(v') of Definition~\ref{def:eps(S)-order} we can conclude $s<\bar\vartheta(t)$.
\end{proof}

The following observation will not be needed, but it is nevertheless instructive:

\begin{remark}
 Any ordinal $\alpha$ that satisfies Assumption~\ref{ass:bachmann-howard-collapse} is at least as big as the usual Bachmann-Howard ordinal: The latter can be characterized as the order-type of the notation system $\vartheta_\emptyset\cap\Omega$ presented in \cite[Section~2.1]{rathjen-model-bi} (due to Rathjen, cf.~the reference to his 1989 lecture notes in the cited paper). We recall that this notation system contains a function symbol $\vartheta$, i.e.~for each term $s\in\vartheta_\emptyset$ we have a term $\vartheta s\in\vartheta_\emptyset\cap\Omega$. The order of terms depends on a map $\vartheta_\emptyset\ni s\mapsto s^*\in\vartheta_\emptyset\cap\Omega$; in particular we have $(\vartheta s)^*=\vartheta s$. It is straightforward to define an embedding $f:\vartheta_\emptyset\rightarrow\varepsilon(S)^u_\alpha$ by induction on terms. The crucial clause is $f(\vartheta s)=\bar\vartheta(f(s))$. To see that $f$ is order preserving one must verify
 \begin{equation*}
  f(s^*)=\max(\suppe_\alpha(f(s))\cup\{0\})
 \end{equation*}
 for all $s\in\vartheta_\emptyset$. In the case of a term $\vartheta s$ the previous lemma yields
 \begin{equation*}
  \suppe_\alpha(f(\vartheta s))=\suppe_\alpha(\bar\vartheta(f(s)))=\{\bar\vartheta(f(s))\}.
 \end{equation*}
 Thus we indeed get
 \begin{equation*}
  \max(\suppe_\alpha(f(\vartheta s))\cup\{0\})=\bar\vartheta(f(s))=f(\vartheta s)=f((\vartheta s)^*).
 \end{equation*}
 Using Proposition~\ref{prop:bar-collapse} it is now straightforward to check that $f$ is an order embedding. It restricts to an embedding of $\vartheta_\emptyset\cap\Omega$ into $\varepsilon(S)^u_\alpha\cap\Omega\cong\alpha$. Thus $\alpha$ is at least as big as the order-type of $\vartheta_\emptyset\cap\Omega$, which is the Bachmann-Howard ordinal.
\end{remark}

\section{From Search Tree to Proof Tree}\label{sect:infinite-proofs}

In this section we construct an infinite proof tree~$P^u_\alpha$ that extends the search tree $S^u_\alpha$ from Section~\ref{sect:deduct-admissible}. Our first goal is to define the appropriate notion of infinite proof. Just as the search tree $S^u_\alpha$, our proofs will be labelled subtrees of $(\mathbf L^u_\alpha)^{<\omega}$. In particular each node will be labelled by a rule, which controls the local structure of the proof. To understand the following one should recall the verification calculus from Definition~\ref{def:disj-conj-formula}. We assume that all $\mathbf L^u_\alpha$-formulas are closed, unless indicated otherwise.

\begin{samepage}
\begin{definition}\label{def:rules}
 An $\mathbf L^u_\alpha$-rule is an expression from the first column of the following table, provided that the corresponding condition in the third column is satisfied. To each $\mathbf L^u_\alpha$-rule $r$ we assign an arity $\iota(r)\subseteq\mathbf L^u_\alpha$, given in the second column.
 {\def\arraystretch{1.75}\tabcolsep=8pt
\setlength{\LTpre}{\baselineskip}\setlength{\LTpost}{\baselineskip}
\begin{longtable}{l|l|p{0.55\textwidth}}\hline
$\mathbf L^u_\alpha$-rule & arity & condition\\ \hline
$(\true,\varphi)$ &  $\emptyset$ & $\varphi$ is a bounded $\mathbf L^u_\alpha$-formula with $\mathbb L^u_\alpha\vDash\varphi$\\
$(\bigwedge,\varphi)$ & $\iota_\alpha(\varphi)$ & $\varphi$ is a conjunctive $\mathbf L^u_\alpha$-formula\\
$(\bigvee,\varphi,a)$ & $\{0\}$ & $\varphi$ is a disjunctive $\mathbf L^u_\alpha$-formula and $a\in\iota_\alpha(\varphi)$\\
$(\cut,\varphi)$ & $\{0,1\}$ & $\varphi$ is an $\mathbf L^u_\alpha$-formula\\
$(\refl,\exists_w\forall_{x\in a}\exists_{y\in w}\theta)$ & $\{0\}$ & $\theta$ is a bounded $\mathbf L^u_\alpha$-formula in which $x$ and $y$ may be free, and we have $a\in\mathbf L^u_\alpha$\\
$(\rep,a)$ & $\{a\}$ & $a\in\mathbf L^u_\alpha$\\ \hline
\end{longtable}}
\end{definition}
\end{samepage}

We can now make our notion of infinite proof precise:

\begin{definition}\label{def:u-a-proof}
 A $(u,\alpha)$-proof consists of a tree $P\subseteq(\mathbf L^u_\alpha)^{<\omega}$ and labellings
 \begin{equation*}
  l:P\rightarrow\text{``$\mathbf L^u_\alpha$-sequents''},\qquad r:P\rightarrow\text{``$\mathbf L^u_\alpha$-rules''},\qquad o:P\rightarrow\varepsilon(S)^u_\alpha,
 \end{equation*}
 such that we have $\sigma^\frown a\in P$ and $o(\sigma^\frown a)<o(\sigma)$ for all $\sigma\in P$ and $a\in\iota(r(\sigma))$. Furthermore the following conditions must be satisfied for all $\sigma\in P$:
 {\def\arraystretch{1.75}\tabcolsep=8pt
\setlength{\LTpre}{\baselineskip}\setlength{\LTpost}{\baselineskip}
\begin{longtable}{l|p{0.65\textwidth}}\hline
If $r(\sigma)$ is \dots & \dots\ then we have \dots\\ \hline
$(\true,\varphi)$ & $\varphi\in l(\sigma)$.\\
$(\bigwedge,\varphi)$ & $\varphi\in l(\sigma)$ and $l(\sigma^\frown a)\subseteq l(\sigma),\varphi_a$ for all $a\in\iota_\alpha(\varphi)$.\\
$(\bigvee,\varphi,a)$ & $\varphi\in l(\sigma)$, $l(\sigma^\frown 0)\subseteq l(\sigma),\varphi_a$ and $\suppl_\alpha(a)\lef o(\sigma)$.\\ 
$(\cut,\varphi)$ & $l(\sigma^\frown 0)\subseteq l(\sigma),\neg\varphi$ and $l(\sigma^\frown 1)\subseteq l(\sigma),\varphi$.\\
$(\refl,\exists_w\forall_{x\in a}\exists_{y\in w}\theta)$ & $\exists_w\forall_{x\in a}\exists_{y\in w}\theta\in l(\sigma)$ and $l(\sigma^\frown 0)\subseteq l(\sigma),\forall_{x\in a}\exists_y\theta$,\newline as well as~$\Omega\leq o(\sigma)$.\\ 
$(\rep,a)$ & $l(\sigma^\frown a)\subseteq l(\sigma)$.\\ \hline
\end{longtable}}
\noindent To express that these conditions hold we say that $P$ is locally correct at $\sigma$.
\end{definition}

The condition $\suppl_\alpha(a)\lef o(\sigma)$ for a rule $(\bigvee,\cdot)$ controls the size of existential witnesses in a proof. Note that we can compare $\suppl_\alpha(a)\subseteq\alpha$ and $o(\sigma)\in\varepsilon(S)^u_\alpha$ via the isomorphism $\alpha=\varepsilon_\alpha\cong\varepsilon(S)^u_\alpha\cap\Omega$ from the previous section (in particular $\suppl_\alpha(a)\lef o(\sigma)$ is trivial in case $\Omega\leq o(\sigma)$). The repetition rule $(\rep,\cdot)$ is trivial from a semantical viewpoint, but it will be crucial for the formalization of proof transformations in a weak meta theory. We have the following soundness result:

\begin{proposition}\label{prop:soundness-below-Omega}
If $P=(P,l,r,o)$ is a $(u,\alpha)$-proof of height $o(\langle\rangle)<\Omega$, then we have $\mathbb L^u_\alpha\vDash\varphi$ for some $\mathbf L^u_\alpha$-formula $\varphi\in l(\langle\rangle)$ in the end-sequent of $P$.
\end{proposition}
\begin{proof}
By induction on $o(\sigma)\in\varepsilon(S)^u_\alpha\cap\Omega\cong\alpha$ we show that $\mathbb L^u_\alpha$ satisfies some formula in $l(\sigma)$. The induction step is established by case distinction on the rule~$r(\sigma)$. Let us consider the case $r(\sigma)=(\bigwedge,\varphi)$. By local correctness we have $\sigma^\frown a\in P$ and $o(\sigma^\frown a)<o(\sigma)<\Omega$ for any $a\in\iota_\alpha(\varphi)$. The induction hypothesis tells us that $\mathbb L^u_\alpha$ satisfies some formula in $l(\sigma^\frown a)\subseteq l(\sigma),\varphi_a$. If this formula lies in $l(\sigma)$, then we are done. Otherwise we get $\mathbb L^u_\alpha\vDash\varphi_a$ for all $a\in\iota_\alpha(\varphi)$. Then Proposition~\ref{prop:clauses-inf-verification-sound} allows us to conclude $\mathbb L^u_\alpha\vDash\varphi$. By local correctness we have $\varphi\in l(\sigma)$, which completes the induction step. The other cases are established similarly. Note that we cannot have $r(\sigma)=(\refl,\exists_w\forall_{x\in a}\exists_{y\in w}\theta)$, as local correctness would require $\Omega\leq o(\sigma)$.
\end{proof}

We cannot apply the semantic argument to proofs of height above $\Omega$, because we do not know whether $\mathbb L^u_\alpha$ satisfies reflection. Instead, we will extend the soundness result in an indirect way: J\"ager's~\cite{jaeger-kripke-platek} ordinal analysis of Kripke-Platek set theory shows that certain proofs can be collapsed to proofs of height below $\Omega$. In the following sections we will adapt this ordinal analysis to our setting. In the rest of the present section we show that the search tree $S^u_\alpha$ can be extended to a $(u,\alpha)$-proof of height above $\Omega$. Let us first construct proofs of the axioms listed in Definition~\ref{def:enum-axioms}:

\begin{lemma}\label{lem:P_0}
There is a $(u,\alpha)$-proof $P_0=(P_0,l_0,r_0,o_0)$ with height $o_0(\langle\rangle)=\Omega$ and end-sequent $l_0(\langle\rangle)=\langle\ax_0\rangle$, where $\ax_0$ is the formula $\forall_x\exists_y\,y=x\cup\{x\}$.
\end{lemma}
\begin{proof}
Proposition~\ref{prop:alpha-epsilon-number} ensures that $\alpha$ is a limit. For each $a\in\mathbf L^u_\alpha$ we compute
\begin{equation*}
\beta_a=\sup\{\beta+1\,|\,\beta\in\suppl_\alpha(a)\}\in\alpha.
\end{equation*}
From $\suppl_\alpha(a)\lef\beta_a$ we can infer $a\in\mathbf L^u_{\beta_a}$, as observed after Definition~\ref{def:term-version-L}. In view of Definition~\ref{def:interpretation-constructible-terms} we get $\llbracket a\rrbracket\in\mathbb L^u_{\beta_a}$. Now we form the $\mathbf L^u_\alpha$-term
\begin{equation*}
b_a=\{z\in L^u_{\beta_a}\,|\,z\in a\lor z=a\}.
\end{equation*}
Its interpretation is given by
\begin{equation*}
\llbracket b_a\rrbracket=\{z\in\mathbb L^u_{\beta_a}\,|\,\mathbb L^u_{\beta_a}\vDash z\in\llbracket a\rrbracket\lor z=\llbracket a\rrbracket\}=\llbracket a\rrbracket\cup\{\llbracket a\rrbracket\}.
\end{equation*}
The desired proof can now be visualized as
\begin{prooftree}
\def\extraVskip{3pt}
\AxiomC{$\cdots$}
\AxiomC{$\vdash^0 b_a=a\cup\{a\}$}
\RightLabel{$(\bigvee)$}
\UnaryInfC{$\vdash^{\beta_a+1}\exists_y\,y=a\cup\{a\}$}
\AxiomC{$\cdots$}
\RightLabel{($\bigwedge$).}
\TrinaryInfC{$\vdash^\Omega\forall_x\exists_y\,y=x\cup\{x\}$}
\end{prooftree}
This means that the leaves of $P_0$ have the form $\langle a,0\rangle$ for arbitrary $a\in\mathbf L^u_\alpha$. They receive the labels
\begin{equation*}
l_0(\langle a,0\rangle)=\langle b_a=a\cup\{a\}\rangle,\quad r_0(\langle a,0\rangle)=(\true,b_a=a\cup\{a\}),\quad o_0(\langle a,0\rangle)=0.
\end{equation*}
By the above we have $\mathbb L^u_\alpha\vDash b_a=a\cup\{a\}$, so that $r_0(\langle a,0\rangle)$ is indeed an $\mathbf L^u_\alpha$-rule. On the next level of $P_0$ we have the labels
\begin{equation*}
l_0(\langle a\rangle)=\langle\exists_y\,y=a\cup\{a\}\rangle,\quad r_0(\langle a\rangle)=(\bigvee,\exists_y\,y=a\cup\{a\},b_a),\quad o_0(\langle a\rangle)=\beta_a+1.
\end{equation*}
Local correctness follows from $\exists_y\,y=a\cup\{a\}\simeq\bigvee_{b\in\mathbf L^u_\alpha}b=a\cup\{a\}$ and
\begin{equation*}
\suppl_\alpha(b_a)=\{\beta_a\}\cup\suppl_\alpha(a)\lef\beta_a+1=o_0(\langle a\rangle).
\end{equation*}
The root of $P_0$ is labelled by
\begin{equation*}
l_0(\langle\rangle)=\langle\forall_x\exists_y\,y=x\cup\{x\}\rangle,\quad r_0(\langle\rangle)=(\bigwedge,\forall_x\exists_y\,y=x\cup\{x\}),\quad o_0(\langle\rangle)=\Omega.
\end{equation*}
Note that $\beta_a+1\in\alpha\cong\varepsilon(S)^u_\alpha\cap\Omega$ implies $o_0(\langle a\rangle)<o_0(\langle\rangle)$. The other local correctness conditions hold because of $\forall_x\exists_y\,y=x\cup\{x\}\simeq\bigwedge_{a\in\mathbf L^u_\alpha}\exists_y\,y=a\cup\{a\}$.
\end{proof}

The remaining axioms can be deduced from the reflection rule:

\begin{lemma}\label{lem:P_n+1}
There are $(u,\alpha)$-proofs $P_{n+1}=(P_{n+1},l_{n+1},r_{n+1},o_{n+1})$ with ordinal height $o_{n+1}(\langle\rangle)<\Omega+\omega=\omega^\Omega+\omega^{\omega^0}$ and end-sequents $l_{n+1}(\langle\rangle)=\langle\ax_{n+1}\rangle$, where
\begin{equation*}
\ax_{n+1}=\forall_{z_1,\dots,z_k}\forall_v(\forall_{x\in v}\exists_y\theta(x,y,z_1,\dots,z_k)\rightarrow\exists_w\forall_{x\in v}\exists_{y\in w}\theta(x,y,z_1,\dots,z_k))
\end{equation*}
is the $n$-th instance of $\Delta_0$-collection, following the enumeration from Definition~\ref{def:enum-axioms}.
\end{lemma}
\begin{proof}
Above the root of $P_{n+1}$ we have $k+1$ applications of the rule $(\bigwedge,\cdot)$, which introduce the universal quantifiers at the beginning of the formula $\ax_{n+1}$. Thus we have a node $\langle a_1,\dots,a_k,b\rangle\in P_{n+1}$ for all $a_1,\dots,a_k,b\in\mathbf L^u_\alpha$. The proof above this node can be visualized as
\begin{prooftree}
\def\extraVskip{3pt}
\AxiomC{$\vdash^{\Omega+5}\neg\forall_{x\in b}\exists_y\theta(x,y,\vec a),\forall_{x\in b}\exists_y\theta(x,y,\vec a)$}
\RightLabel{$(\refl)$}
\UnaryInfC{$\vdash^{\Omega+6}\neg\forall_{x\in b}\exists_y\theta(x,y,\vec a),\exists_w\forall_{x\in b}\exists_{y\in w}\theta(x,y,\vec a)$}
\RightLabel{$(\bigvee)$}
\UnaryInfC{$\vdash^{\Omega+7}\neg\forall_{x\in b}\exists_y\theta(x,y,\vec a),\neg\forall_{x\in b}\exists_y\theta(x,y,\vec a)\lor\exists_w\forall_{x\in b}\exists_{y\in w}\theta(x,y,\vec a)$}
\RightLabel{$(\bigvee).$}
\UnaryInfC{$\vdash^{\Omega+8}\neg\forall_{x\in b}\exists_y\theta(x,y,\vec a)\lor\exists_w\forall_{x\in b}\exists_{y\in w}\theta(x,y,\vec a)$}
\end{prooftree}
We leave it to the reader to determine the precise rule used at each node. At the crucial node $\langle a_1,\dots,a_k,b,0,0\rangle$ we have the labels
\begin{align*}
l_{n+1}(\langle\vec a,b,0,0\rangle)&=\langle\neg\forall_{x\in b}\exists_y\theta(x,y,\vec a),\exists_w\forall_{x\in b}\exists_{y\in w}\theta(x,y,\vec a)\rangle,\\
r_{n+1}(\langle\vec a,b,0,0\rangle)&=(\refl,\exists_w\forall_{x\in b}\exists_{y\in w}\theta(x,y,\vec a)),\\
o_{n+1}(\langle\vec a,b,0,0\rangle)&=\Omega+6=\omega^\Omega+\underbrace{\omega^0+\dots+\omega^0}_{\text{six times}}.
\end{align*}
It is straightforward to see that the local correctness conditions are satisfied. In particular we have $\Omega\leq o_{n+1}(\langle\vec a,b,0,0\rangle)$, as required in the case of a rule $(\refl,\cdot)$. It remains to derive the sequent above the reflection rule, which has the form $\langle\neg\varphi,\varphi\rangle$. Such sequents can be derived in general, but we prefer to focus on the present case: Assume that we are concerned with a term of the form $b=\{z\in L^u_\gamma\,|\,\varphi(z)\}$ (the cases $b=L^u_\gamma$ and $b=v$ with $v\in u$ are somewhat easier). Then we have
\begin{equation*}
\forall_{x\in b}\exists_y\theta(x,y,\vec a)\simeq\textstyle\bigwedge_{\suppl_\alpha(c)\lef\gamma}\neg\varphi(c)\lor\exists_y\theta(c,y,\vec a),
\end{equation*}
and the remaining part of the proof can be visualized as
\vspace{-\baselineskip}
\begin{prooftree}
\def\extraVskip{3pt}
\AxiomC{$\cdots$}
\AxiomC{$\vdash^0\varphi(c),\neg\varphi(c)$}
\AxiomC{$\cdots$}
\AxiomC{$\vdash^0\neg\theta(c,d,\vec a),\theta(c,d,\vec a)$}
\UnaryInfC{$\vdash^{\beta_d+1}\neg\theta(c,d,\vec a),\exists_y\theta(c,y,\vec a)$}
\AxiomC{$\cdots$}
\TrinaryInfC{$\vdash^\Omega\forall_y\neg\theta(c,y,\vec a),\exists_y\theta(c,y,\vec a)$}
\BinaryInfC{$\vdash^{\Omega+1}\varphi(c)\land\forall_y\neg\theta(c,y,\vec a),\neg\varphi(c),\exists_y\theta(c,y,\vec a)$}
\doubleLine
\UnaryInfC{$\vdash^{\Omega+3}\varphi(c)\land\forall_y\neg\theta(c,y,\vec a),\neg\varphi(c)\lor\exists_y\theta(c,y,\vec a)$}
\UnaryInfC{$\vdash^{\Omega+4}\neg\forall_{x\in b}\exists_y\theta(x,y,\vec a),\neg\varphi(c)\lor\exists_y\theta(c,y,\vec a)$}
\AxiomC{$\cdots$}
\RightLabel{.}
\TrinaryInfC{$\vdash^{\Omega+5}\neg\forall_{x\in b}\exists_y\theta(x,y,\vec a),\forall_{x\in b}\exists_y\theta(x,y,\vec a)$}
\end{prooftree}
The leaves can be labelled by rules $(\true,\cdot)$, since one of the bounded formulas $\varphi(c)$ and $\neg\varphi(c)$ (resp.~$\neg\theta(c,d,\vec a)$ and $\theta(c,d,\vec a)$) must hold in $\mathbb L^u_\alpha$. The required bound $\beta_d$ with $\suppl_\alpha(d)\lef\beta_d+1<\Omega$ is computed as in the previous proof. The double line indicates two applications of $(\bigvee,\cdot)$. 
\end{proof}

To get the following result we attach the constructed proofs to the open assumptions of the search tree $S^u_\alpha$:

\begin{proposition}\label{prop:proof-from-search-tree}
 There is a $(u,\alpha)$-proof $P^u_\alpha=(P^u_\alpha,l^u_\alpha,r^u_\alpha,o^u_\alpha)$ with empty end-sequent $l^u_\alpha(\langle\rangle)=\langle\rangle$ and ordinal height $o^u_\alpha(\langle\rangle)=\mathfrak E_{\langle\rangle}$.
\end{proposition}
\begin{proof}
 Invoking the proofs $P_n=(P_n,l_n,r_n,o_n)$ from the previous lemmas, the underlying tree of our proof can be given by
 \begin{equation*}
  P^u_\alpha=S^u_\alpha\cup\{\sigma^\frown 1^\frown\tau\,|\,\sigma\in S^u_\alpha\land\exists_{n\in\omega}(\len(\sigma)=2n\land\tau\in P_n)\}.
 \end{equation*}
 Note that the decomposition $\sigma^\frown 1^\frown\tau$ is unique: According to Definition~\ref{def:search-tree} we have $\sigma^\frown 1\notin S^u_\alpha$ when $\sigma$ has even length. Recall that the search tree comes with a function $l_\alpha:S^u_\alpha\rightarrow\text{``$\mathbf L^u_\alpha$-sequents''}$. Thus we can define $l^u_\alpha:P^u_\alpha\rightarrow\text{``$\mathbf L^u_\alpha$-sequents''}$~by
 \begin{align*}
  l^u_\alpha(\sigma)&=l_\alpha(\sigma)\quad&&\text{for $\sigma\in S^u_\alpha$},\\
  l^u_\alpha(\sigma^\frown 1^\frown\tau)&=l_n(\tau)\quad&&\text{for $\sigma\in S^u_\alpha$ with $\len(\sigma)=2n$ and $\tau\in P_n$}.
 \end{align*}
 In particular we have $l^u_\alpha(\langle\rangle)=l_\alpha(\langle\rangle)=\langle\rangle$, as claimed by the proposition. Similarly we set $r^u_\alpha(\sigma^\frown 1^\frown\tau)=r_n(\tau)$ and $o^u_\alpha(\sigma^\frown 1^\frown\tau)=o_n(\tau)$. Then local correctness at $\sigma^\frown 1^\frown\tau\in P^u_\alpha$ follows from local correctness at $\tau\in P_n$. It remains to define $r^u_\alpha(\sigma)$ and $o^u_\alpha(\sigma)$ for $\sigma\in S^u_\alpha$ and to verify local correctness at these nodes. In view of Definition~\ref{def:eps(S)-order} we can set
 \begin{equation*}
  o^u_\alpha(\sigma)=\mathfrak E_\sigma\qquad\text{for $\sigma\in S^u_\alpha$},
 \end{equation*}
 similar to a construction of Rathjen and Valencia Vizca\'ino~\cite{rathjen-model-bi}. To define $r^u_\alpha(\sigma)$ we follow Definition~\ref{def:search-tree}: For $\sigma\in S^u_\alpha$ with $\len(\sigma)=2n$ we put $r^u_\alpha(\sigma)=(\cut,\ax_n)$. The construction of search trees yields $\sigma^\frown 0\in S^u_\alpha\subseteq P^u_\alpha$ and $l^u_\alpha(\sigma^\frown 0)=l^u_\alpha(\sigma),\neg\ax_n$, as required by Definition~\ref{def:u-a-proof}. Since $<_{S^u_\alpha}$ is the Kleene-Brouwer order on $S^u_\alpha$ we have $\sigma^\frown 0<_{S^u_\alpha}\sigma$ and thus $o^u_\alpha(\sigma^\frown 0)<_{\varepsilon(S)^u_\alpha}o^u_\alpha(\sigma)$. From $\langle\rangle\in P_n$ we infer $\sigma^\frown 1\in P^u_\alpha$ and
 \begin{equation*}
  l^u_\alpha(\sigma^\frown 1)=l_n(\langle\rangle)=\langle\ax_n\rangle\subseteq l^u_\alpha(\sigma),\ax_n,
 \end{equation*}
 as well as
 \begin{equation*}
  o^u_\alpha(\sigma^\frown 1)=o_n(\langle\rangle)<\Omega+\omega=\omega^\Omega+\omega^{\omega^0}<\mathfrak E_\sigma=o^u_\alpha(\sigma).
 \end{equation*}
 This shows that $P^u_\alpha$ is locally correct at nodes $\sigma\in S^u_\alpha$ of even length. Now consider a node $\sigma\in S^u_\alpha$ with $\len(\sigma)=2n+1$. We write $\varphi$ for the $\pi_0(n)$-th formula of $l_\alpha(\sigma)$. If $\varphi$ is conjunctive, we put $r^u_\alpha(\sigma)=(\bigwedge,\varphi)$. If $\varphi$ is disjunctive, we compute $b\in\mathbf L^u_\alpha$ as in Definition~\ref{def:search-tree}. In case $b\in\iota_\alpha(\varphi)$ we set $r^u_\alpha(\sigma)=(\bigvee,\varphi,b)$, otherwise we set $r^u_\alpha(\sigma)=(\rep,0)$. It is straightforward to verify the local correctness conditions.
\end{proof}

\section{Transforming Infinite Proofs}

To define transformations of infinite proofs we would like to use recursion over their height. Unfortunately this is not possible in our meta theory~$\atrs$, which does not prove that the subtrees of $(\mathbf L^u_\alpha)^{<\omega}$ form a set. Buchholz~\cite{buchholz91,buchholz-notations-set-theory} has introduced an elegant method for the formalization of infinite proofs in weak theories: The idea is to represent each proof by a term that reflects its role in the ordinal analysis. In the first half of the present section we adapt this approach to our setting. In the second half we use it to implement cut elimination. Our first goal is to define a term system that reconstructs the proof $P^u_\alpha=(P^u_\alpha,l^u_\alpha,r^u_\alpha,o^u_\alpha)$ from Proposition~\ref{prop:proof-from-search-tree}:

\begin{definition}\label{def:proof-codes}
 A basic $(u,\alpha)$-code is an expression of the form $P^u_\alpha\sigma$ with $\sigma\in P^u_\alpha$. We consider the functions
  \begin{align*}
l_{\langle\rangle}&:\text{``basic $(u,\alpha)$-codes''}\rightarrow\text{``$\mathbf L^u_\alpha$-sequents''},& l_{\langle\rangle}(P^u_\alpha\sigma)&=l^u_\alpha(\sigma),\\
r_{\langle\rangle}&:\text{``basic $(u,\alpha)$-codes''}\rightarrow\text{``$\mathbf L^u_\alpha$-rules''}, & r_{\langle\rangle}(P^u_\alpha\sigma)&=r^u_\alpha(\sigma),\\
o_{\langle\rangle}&:\text{``basic $(u,\alpha)$-codes''}\rightarrow\varepsilon(S)^u_\alpha, & o_{\langle\rangle}(P^u_\alpha\sigma)&=o^u_\alpha(\sigma),\\
 \end{align*}
 as well as the function
 \begin{gather*}
  n:\text{``basic $(u,\alpha)$-codes''}\times\mathbf L^u_\alpha\rightarrow\text{``basic $(u,\alpha)$-codes''},\\
  n(P_S\sigma,a)=\begin{cases}
                   P^u_\alpha\,\sigma^\frown a & \text{if $\sigma^\frown a\in P^u_\alpha$},\\
                   P^u_\alpha\langle\rangle & \text{otherwise}.
                 \end{cases}
 \end{gather*}
\end{definition}

The basic $(u,\alpha)$-codes are the constant symbols of a system of $(u,\alpha)$-codes that will be dynamically extended over the following sections. The functions $l_{\langle\rangle}$, $r_{\langle\rangle}$, $o_{\langle\rangle}$ and $n$ will extend to all $(u,\alpha)$-codes by recursion over terms: Whenever we add a new function symbol to the system of $(u,\alpha)$-codes, we will add corresponding recursive clauses. Anticipating this development, the following is formulated for arbitrary $(u,\alpha)$-codes rather than just for basic ones.

\begin{definition}\label{def:code-reconstruct-proof}
 By recursion over finite sequences we define a function
 \begin{gather*}
  \bar n:\text{``$(u,\alpha)$-codes''}\times(\mathbf L^u_\alpha)^{<\omega}\rightarrow\text{``$(u,\alpha)$-codes''},\\
  \bar n(P,\langle\rangle)=P,\qquad\bar n(P,\sigma^\frown a)=n(\bar n(P,\sigma),a).
 \end{gather*}
 In the following we write $\iota(P)$ rather than $\iota(r_{\langle\rangle}(P))$ for the arity of the last rule of~$P$ (cf.~Definition~\ref{def:rules}). To each $(u,\alpha)$-code $P$ we associate a tree $[P]\subseteq(\mathbf L^u_\alpha)^{<\omega}$: We have $\langle\rangle\in[P]$, and if $\sigma\in[P]$ is given by recursion, then we stipulate
 \begin{equation*}
  \sigma^\frown a\in[P]\quad\Leftrightarrow\quad a\in\iota(\bar n(P,\sigma)).
 \end{equation*}
 We also consider the functions
\begin{align*}
l_P&:[P]\rightarrow\text{``$\mathbf L^u_\alpha$-sequents''},& l_P(\sigma)&=l_{\langle\rangle}(\bar n(P,\sigma)),\\
r_P&:[P]\rightarrow\text{``$\mathbf L^u_\alpha$-rules''},& r_P(\sigma)&=r_{\langle\rangle}(\bar n(P,\sigma)),\\
o_P&:[P]\rightarrow\varepsilon(S)^u_\alpha,& o_P(\sigma)&=o_{\langle\rangle}(\bar n(P,\sigma)).
\end{align*}
The tuple $[P]=([P],l_P,r_P,o_P)$ is called the interpretation of $P$.
\end{definition}

Let us reformulate Definition~\ref{def:u-a-proof} in terms of codes:

\begin{definition}\label{def:condition-L}
 We say that a $(u,\alpha)$-code $P$ satisfies condition (L) if we have
 \begin{equation*}
  o_{\langle\rangle}(n(P,a))<o_{\langle\rangle}(P)\qquad\text{for all $a\in\iota(P)$}
 \end{equation*}
 and if the relevant condition from the following table is satisfied:
 {\def\arraystretch{1.75}\tabcolsep=4pt
\setlength{\LTpre}{\baselineskip}\setlength{\LTpost}{\baselineskip}
\begin{longtable}{l|p{0.7\textwidth}}\hline
If $r_{\langle\rangle}(P)$ is \dots & \dots\ then we have \dots\\ \hline
$(\true,\varphi)$ & $\varphi\in l_{\langle\rangle}(P)$.\\
$(\bigwedge,\varphi)$ & $\varphi\in l_{\langle\rangle}(P)$ and $l_{\langle\rangle}(n(P,a))\subseteq l_{\langle\rangle}(P),\varphi_a$ for all $a\in\iota_\alpha(\varphi)$.\\
$(\bigvee,\varphi,a)$ & $\varphi\in l_{\langle\rangle}(P)$ and $l_{\langle\rangle}(n(P,0))\subseteq l_{\langle\rangle}(P),\varphi_a$,\newline as well as $\suppl_\alpha(a)\lef o_{\langle\rangle}(P)$.\\ 
$(\cut,\varphi)$ & $l_{\langle\rangle}(n(P,0))\subseteq l_{\langle\rangle}(P),\neg\varphi$ and $l_{\langle\rangle}(n(P,1))\subseteq l_{\langle\rangle}(P),\varphi$.\\
$(\refl,\exists_w\forall_{x\in a}\exists_{y\in w}\theta)$ & $\exists_w\forall_{x\in a}\exists_{y\in w}\theta\in l_{\langle\rangle}(P)$ and $l_{\langle\rangle}(n(P,0))\subseteq l_{\langle\rangle}(P),\forall_{x\in a}\exists_y\theta$,\newline as well as~$\Omega\leq o_{\langle\rangle}(P)$.\\ 
$(\rep,a)$ & $l_{\langle\rangle}(n(P,a))\subseteq l_{\langle\rangle}(P)$.\\ \hline
\end{longtable}}
\end{definition}

Condition (L) for the single $(u,\alpha)$-code $P$ only ensures local correctness of $[P]$ at the root. To see that $[P]$ is correct at every node we need to consider the entire system of $(u,\alpha)$-codes:

\begin{lemma}
 All basic $(u,\alpha)$-codes satisfy condition (L).
\end{lemma}
\begin{proof}
 Condition (L) for the $(u,\alpha)$-code $P^u_\alpha\sigma$ follows from local correctness of the $(u,\alpha)$-proof $P^u_\alpha$ at the node $\sigma$: Consider an arbitrary $a\in\iota(P^u_\alpha\sigma)=\iota(r^u_\alpha(\sigma))$. In view of Proposition~\ref{prop:proof-from-search-tree} and Definition~\ref{def:u-a-proof} we get $\sigma^\frown a\in P^u_\alpha$ and then
 \begin{equation*}
  o_{\langle\rangle}(n(P^u_\alpha\sigma,a))=o_{\langle\rangle}(P^u_\alpha\,\sigma^\frown a)=o^u_\alpha(\sigma^\frown a)<o^u_\alpha(\sigma)=o_{\langle\rangle}(P^u_\alpha\sigma).
 \end{equation*}
 The remaining conditions are verified similarly.
\end{proof}

Condition (L) for arbitrary $(u,\alpha)$-codes is established by induction over terms: Whenever we add a new function symbol to the system of $(u,\alpha)$-codes, we will verify the corresponding induction step. Thus the following remains valid for all $(u,\alpha)$-codes that we introduce:

\begin{proposition}\label{prop:code-into-proof}\label{prop:codes-to-proofs}
 If $P$ is a $(u,\alpha)$-code, then its interpretation $[P]$ is a $(u,\alpha)$-proof.
\end{proposition}
\begin{proof}
 The local correctness of $[P]$ at the node $\sigma$ follows from condition (L) for the $(u,\alpha)$-code $\bar n(P,\sigma)$: Consider an arbitrary $a\in\iota(r_P(\sigma))=\iota(\bar n(P,\sigma))$. In view of Definition~\ref{def:code-reconstruct-proof} we get $\sigma^\frown a\in[P]$ and
 \begin{equation*}
  o_P(\sigma^\frown a)=o_{\langle\rangle}(\bar n(P,\sigma^\frown a))=o_{\langle\rangle}(n(\bar n(P,\sigma),a))<o_{\langle\rangle}(\bar n(P,\sigma))=o_P(\sigma),
 \end{equation*}
 as required by Definition~\ref{def:u-a-proof}. The remaining conditions are verified similarly.
\end{proof}

We have already indicated that the system of $(u,\alpha)$-codes will be extended dynamically. Let us now explain how this works in detail:

\begin{remark}\label{rmk:codes-methodology}
 As we present different steps of our ordinal analysis we will introduce function symbols
 \begin{equation*}
  \mathcal I_{\varphi,a},\qquad\mathcal R_{\varphi},\qquad\mathcal E,\qquad\mathcal B_{\varphi,\gamma},\qquad\mathcal C_t.
 \end{equation*}
 At the end of Section~\ref{sect:collapsing} we will have completed the following inductive definition:
 \begin{itemize}
  \item Any basic $(u,\alpha)$-code is a $(u,\alpha)$-code.
  \item If $\mathcal F$ is a $k$-ary function symbol listed above and $P_1,\dots,P_k$ are $(u,\alpha)$-codes, then $\mathcal F P_1\dots P_k$ is a $(u,\alpha)$-code as well.
 \end{itemize}
 In order to capture certain properties of the $(u,\alpha)$-codes we will introduce functions
 \begin{align*}
  d&:\text{``$(u,\alpha)$-codes''}\rightarrow\omega,\\
  h_0&:\text{``$(u,\alpha)$-codes''}\rightarrow\varepsilon(S)^u_\alpha,\\
  h_1&:\text{``$(u,\alpha)$-codes''}\rightarrow[\alpha]^{<\omega}.
 \end{align*}
 We must also define extensions
 \begin{align*}
  l_{\langle\rangle}&:\text{``$(u,\alpha)$-codes''}\rightarrow\text{``$\mathbf L^u_\alpha$-sequents''},& r_{\langle\rangle}&:\text{``$(u,\alpha)$-codes''}\rightarrow\text{``$\mathbf L^u_\alpha$-rules''},\\
  o_{\langle\rangle}&:\text{``$(u,\alpha)$-codes''}\rightarrow\varepsilon(S)^u_\alpha,& n&:\text{``$(u,\alpha)$-codes''}\times\mathbf L^u_\alpha\rightarrow\text{``$(u,\alpha)$-codes''}
 \end{align*}
 of the functions considered above. This works as follows:
 \begin{enumerate}[label=(\roman*)]
  \item Define the values $l_{\langle\rangle}(P)$, $o_{\langle\rangle}(P)$, $d(P)$, $h_0(P)$ and $h_1(P)$ by simultaneous recursion over the $(u,\alpha)$-code $P$.
  \item Then define the value $r_{\langle\rangle}(P)$ and the function $a\mapsto n(P,a)$, again by simultaneous recursion over $P$.
 \end{enumerate}
 It is crucial to complete step~(i) before step~(ii): For example the rule $r_{\langle\rangle}(\mathcal C_tP)$ will depend on ordinals $o_{\langle\rangle}(n(P,a))$, similarly to~\cite[Definition~5.3]{buchholz-notations-set-theory}. Since $n(P,a)$ is not a subterm of~$P$ we cannot define $r_{\langle\rangle}$ and $o_{\langle\rangle}$ in the same recursion. At the end of Section~\ref{sect:collapsing} the functions $l_{\langle\rangle}$, $r_{\langle\rangle}$, $o_{\langle\rangle}$ and $n$ will be defined on all $(u,\alpha)$-codes. Thus Definition~\ref{def:code-reconstruct-proof} will provide an interpretation $[P]$ of any $(u,\alpha)$-code $P$. To control the behaviour of the functions $d$, $h_0$ and $h_1$ we will introduce local correctness conditions (C1), (C2), (H1), (H2) and~(H3). Condition (L) from Definition~\ref{def:condition-L} must also be established for all $(u,\alpha)$-codes. This can be achieved as follows:
 \begin{enumerate}[label=(\roman*')]
  \item Show that any $(u,\alpha)$-code $P$ satisfies condition (H1), by induction over $P$.
  \item Then use simultaneous induction over $P$ to show that any $(u,\alpha)$-code $P$ satisfies conditions (L), (C1), (C2), (H2) and~(H3). At the same time one should confirm that $r_{\langle\rangle}(P)$ is an $\mathbf L^u_\alpha$-rule and that $n(P,a)$ is a $(u,\alpha)$-code.
 \end{enumerate}
 Again it is crucial to complete step~(i') first: We will need condition~(H1) for $n(P,a)$ to show condition (L) for $\mathcal C_tP$ (cf.~\cite[Lemma~5.1]{buchholz-notations-set-theory}). At the end of Section~\ref{sect:collapsing} we will have established condition~(L) for all $(u,\alpha)$-codes. By the proof of Proposition~\ref{prop:code-into-proof} it will follow that the interpretation of any $(u,\alpha)$-code is a $(u,\alpha)$-proof in the sense of~Definition~\ref{def:u-a-proof}. In the present remark we have described the formal structure~of our argument. Over the course of the following sections it will be filled with content. We stress that the order of presentation will not coincide with the official order of the argument: To give a readable account of the ordinal analysis we will state the recursive clauses for (i,ii) and the induction steps for (i',ii') as they occur naturally. This is unproblematic, because it does not matter in which order the clauses are printed. What matters is that the clauses respect the dependencies discussed above (e.g.~$r_{\langle\rangle}(\mathcal C_tP)$ may depend on $o_{\langle\rangle}(n(P,a))$, but $o_{\langle\rangle}(\mathcal C_tP)$ may not depend on~$r_{\langle\rangle}(P)$). To conclude this remark we explain the following \emph{fa\c{c}on de parler}: Several results will claim that the system of $(u,\alpha)$-codes ``can be extended'' by a function symbol with certain properties. By such a claim we mean that we can provide suitable recursive clauses for (i,ii) and proofs of the induction steps for (i',ii').
\end{remark}

Following Buchholz~\cite{buchholz91}, we will now use $(u,\alpha)$-codes to define cut elimination. The latter is needed as an auxiliary construction for our ordinal analysis. The complexity of cut formulas is measured as follows:

\begin{definition}\label{def:formula-rank}
 The rank $\rk(\varphi)\in\omega$ of an $\mathbf L^u_\alpha$-formula $\varphi$ is inductively defined by the following clauses:
 \begin{enumerate}[label=(\roman*)]
  \item If $\varphi$ is bounded, then we set $\rk(\varphi)=0$.
  \item If $\varphi_0$ or $\varphi_1$ is unbounded, then we set
  \begin{equation*}
   \rk(\varphi_0\lor\varphi_1)=\rk(\varphi_0\land\varphi_1)=\max\{\rk(\varphi_0),\rk(\varphi_1)\}+1.
  \end{equation*}
  \item If $\varphi$ is unbounded and $a$ is an element of $\mathbf L^u_\alpha$ or a variable, then we set
  \begin{equation*}
   \rk(\exists_{x\in a}\varphi)=\rk(\forall_{x\in a}\varphi)=\rk(\varphi)+2.
  \end{equation*}
  \item We set $\rk(\exists_x\varphi)=\rk(\forall_x\varphi)=\rk(\varphi)+1$.
 \end{enumerate}
\end{definition}

It is straightforward to check the following connection with Definition~\ref{def:disj-conj-formula}:

\begin{lemma}\label{lem:rank-instances}
 Assume that $\varphi$ is a closed $\mathbf L^u_\alpha$-formula with $\rk(\varphi)>0$. Then we have $\rk(\varphi_a)<\rk(\varphi)$ for all $a\in\iota_\alpha(\varphi)$.
\end{lemma}

The function $d:\text{``$(u,\alpha)$-codes''}\rightarrow\omega$ mentioned in Remark~\ref{rmk:codes-methodology} will control the rank of cut formulas in a proof. We have explained that $d$ is to be defined by recursion over $(u,\alpha)$-codes. So far we have only introduced the basic $(u,\alpha)$-codes of Definition~\ref{def:proof-codes}. Thus we can currently only state the base case of the definition of~$d$. Further recursive clauses will be added as the system of $(u,\alpha)$-codes is extended.

\begin{definition}\label{def:cut-rank-basic}
 We set $d(P^u_\alpha\sigma)=C+6$ for any basic $(u,\alpha)$-code $P^u_\alpha\sigma$, where $C$ is the constant from Definition~\ref{def:enum-axioms}.
\end{definition}

As part of the following result we state the local correctness conditions (C1) and~(C2), which have also been mentioned in Remark~\ref{rmk:codes-methodology}. They are to be established by induction over $(u,\alpha)$-codes. At the moment we can only consider the basic $(u,\alpha)$-codes, which form the base case of the induction. Whenever we extend the system of $(u,\alpha)$-codes we will add a proof of the corresponding induction step.

\begin{lemma}
 The following holds for any basic $(u,\alpha)$-code $P$:
 \begin{enumerate}[label=(C\arabic*)]
  \item If $r_{\langle\rangle}(P)=(\cut,\varphi)$ is a cut rule, then we have $\rk(\varphi)<d(P)$.
  \item We have $d(n(P,a))\leq d(P)$ for any $a\in\iota(P)$.
 \end{enumerate}
\end{lemma}
\begin{proof}
 Condition (C2) is trivial for a basic $(u,\alpha)$-code $P=P^u_\alpha\sigma$: Since $n(P^u_\alpha\sigma,a)$ is a basic $(u,\alpha)$-code as well, we have $d(n(P^u_\alpha\sigma,a))=C+6=d(P^u_\alpha\sigma)$. Now assume $r_{\langle\rangle}(P^u_\alpha\sigma)=r^u_\alpha(\sigma)=(\cut,\varphi)$. Considering the proof of Proposition~\ref{prop:proof-from-search-tree}, the formula~$\varphi$ must be one of the axioms $\ax_n$ from Definition~\ref{def:enum-axioms}. We have
 \begin{equation*}
  \rk(\ax_0)=\rk(\forall_x\exists_y\,y=x\cup\{x\})=2<C+6=d(P^u_\alpha\sigma),
 \end{equation*}
 as required by condition (C1). For a $\Delta_0$-collection axiom
 \begin{equation*}
  \ax_{n+1}=\forall_{z_1,\dots,z_k}\forall_v(\forall_{x\in v}\exists_y\theta(x,y,z_1,\dots,z_k)\rightarrow\exists_w\forall_{x\in v}\exists_{y\in w}\theta(x,y,z_1,\dots,z_k))
 \end{equation*}
 with $k\leq C$ parameters we have $\rk(\ax_{n+1})=k+5<C+6=d(P^u_\alpha\sigma)$. 
\end{proof}

The following result constitutes the first extension of our system of $(u,\alpha)$-codes. It implements an operation known as inversion, which transforms a proof of a formula $\varphi=\forall_x\psi(x)\simeq\bigwedge_{a\in\mathbf L^u_\alpha}\psi(a)$ into a proof of an arbitrary instance $\varphi_a=\psi(a)$.

\begin{lemma}
 For any conjunctive $\mathbf L^u_\alpha$-formula $\varphi$ and any $a\in\iota_\alpha(\varphi)$ we can extend the system of $(u,\alpha)$-codes by a unary function symbol $\mathcal I_{\varphi,a}$, such that we have
 \begin{equation*}
  l_{\langle\rangle}(\mathcal I_{\varphi,a}P)=(l_{\langle\rangle}(P)\backslash\{\varphi\})\cup\{\varphi_a\},\qquad o_{\langle\rangle}(\mathcal I_{\varphi,a}P)=o_{\langle\rangle}(P),\qquad d(\mathcal I_{\varphi,a}P)=d(P)
 \end{equation*}
 for any $(u,\alpha)$-code $P$.
\end{lemma}
\begin{proof}
According to Remark~\ref{rmk:codes-methodology} we must state the recursive clauses and prove the induction steps for the new function symbols $\mathcal I_{\varphi,a}$. The clauses for $l_{\langle\rangle}$, $o_{\langle\rangle}$ and $d$ can be found in the statement of the lemma. The clauses for $r_{\langle\rangle}$ and $n$ are given by
 \begin{equation*}
 r_{\langle\rangle}(\mathcal I_{\varphi,a}P)=\begin{cases}
                      (\rep,a) & \text{if $r_{\langle\rangle}(P)=(\bigwedge,\varphi)$},\\
                      (\true,\varphi_a) & \text{if $r_{\langle\rangle}(P)=(\true,\varphi)$},\\
                      r_{\langle\rangle}(P) & \text{otherwise},
                     \end{cases}\qquad
 n(\mathcal I_{\varphi,a}P,b)=\mathcal I_{\varphi,a}n(P,b).
\end{equation*}
It remains to show that conditions (L), (C1) and (C2) for $P$ imply the same conditions for $\mathcal I_{\varphi,a}P$ (the functions $h_0$ and $h_1$ and the corresponding conditions (H1) to~(H3), which were also mentioned in Remark~\ref{rmk:codes-methodology}, will be introduced later). This can be verified by case distinction on the last rule of $P$. We consider the crucial case $r_{\langle\rangle}(P)=(\bigwedge,\varphi)$. In view of $a\in\iota_\alpha(\varphi)=\iota(P)$ condition~(L) for $P$ implies
\begin{equation*}
 o_{\langle\rangle}(n(\mathcal I_{\varphi,a}P,a))=o_{\langle\rangle}(\mathcal I_{\varphi,a}n(P,a))=o_{\langle\rangle}(n(P,a))<o_{\langle\rangle}(P)=o_{\langle\rangle}(\mathcal I_{\varphi,a}P),
\end{equation*}
as well as
\begin{multline*}
 l_{\langle\rangle}(n(\mathcal I_{\varphi,a}P,a))=l_{\langle\rangle}(\mathcal I_{\varphi,a}n(P,a))=(l_{\langle\rangle}(n(P,a))\backslash\{\varphi\})\cup\{\varphi_a\}\subseteq\\
 \subseteq((l_{\langle\rangle}(P)\cup\{\varphi_a\})\backslash\{\varphi\})\cup\{\varphi_a\}\subseteq(l_{\langle\rangle}(P)\backslash\{\varphi\})\cup\{\varphi_a\}=l_{\langle\rangle}(\mathcal I_{\varphi,a}P).
\end{multline*}
As we have $r_{\langle\rangle}(\mathcal I_{\varphi,a}P)=(\rep,a)$ this is just what condition (L) for $\mathcal I_{\varphi,a}P$ demands. Conditions (C1) and (C2) for $\mathcal I_{\varphi,a}P$ are easily deduced from the same conditions for $P$. Let us also consider the case $r_{\langle\rangle}(P)=(\true,\varphi)$. Inductively we may assume that $(\true,\varphi)$ is an $\mathbf L^u_\alpha$-rule, so that $\varphi$ is a bounded formula with $\mathbb L^u_\alpha\vDash\varphi$. Then $\varphi_a$ is a bounded formula as well. By Proposition~\ref{prop:clauses-inf-verification-sound} we get $\mathbb L^u_\alpha\vDash\varphi_a$, which confirms that $r_{\langle\rangle}(\mathcal I_{\varphi,a}P)=(\true,\varphi_a)$ is an $\mathbf L^u_\alpha$-rule. The remaining verifications are straightforward. Note that $\varphi$ is different from any formula $\psi$ introduced by a rule $(\bigvee,\psi,b)$ or $(\refl,\psi)$, since such a formula $\psi$ must be disjunctive.
\end{proof}

Even though recursion over well-founded trees is not available in $\atrs$, it can be a useful framework for an intuitive explanation of our constructions: Assume that the proof $P$ from the previous result deduces a sequent $\Gamma,\varphi$ by the rule $(\bigwedge,\varphi)$. Then the immediate subtrees $n(P,a)$ of $P$ deduce the sequents $\Gamma,\varphi,\varphi_a$. Recursively we may remove the formula $\varphi$ in these subproofs, to get proofs $\mathcal I_{\varphi,a}n(P,a)$ of the sequents $\Gamma,\varphi_a$. In a strong meta theory we could define $\mathcal I_{\varphi,a}P$ as the proof~$\mathcal I_{\varphi,a}n(P,a)$. The resulting equality $r_{\langle\rangle}(\mathcal I_{\varphi,a}P)=r_{\langle\rangle}(\mathcal I_{\varphi,a}n(P,a))$, however, cannot be used as a clause for our recursive definition. The repetition rule $r_{\langle\rangle}(\mathcal I_{\varphi,a}P)=(\rep,a)$ is crucial, because it allows us to call the proof $\mathcal I_{\varphi,a}n(P,a)$ without committing to its last rule. This use of the repetition rule is due to Mints~\cite{mints-continuous}. The improper $\omega$-rule considered by Schwichtenberg~\cite{schwichtenberg77} serves a similar purpose. We continue with an operation known as reduction: It combines a proof of $\Gamma,\varphi$ and a proof of $\Gamma,\neg\varphi$ into a proof of $\Gamma$, without applying a cut over $\varphi$. The assignment of ordinal heights relies on Lemma~\ref{lem:ordinal-addition}.

\begin{lemma}
 For any conjunctive $\mathbf L^u_\alpha$-formula $\varphi$ with $\rk(\varphi)\geq 2$ we can extend the system of $(u,\alpha)$-codes by a binary function symbol $\mathcal R_\varphi$, such that we have
\begin{align*}
 l_{\langle\rangle}(\mathcal R_\varphi P_0P_1)&=(l_{\langle\rangle}(P_0)\backslash\{\neg\varphi\})\cup(l_{\langle\rangle}(P_1)\backslash\{\varphi\}),\\
 o_{\langle\rangle}(\mathcal R_\varphi P_0P_1)&=o_{\langle\rangle}(P_1)+o_{\langle\rangle}(P_0),\\
 d(\mathcal R_\varphi P_0P_1)&=\max\{d(P_0),d(P_1),\rk(\varphi)\}
\end{align*}
for all $(u,\alpha)$-codes $P_0$ and $P_1$.
\end{lemma}
\begin{proof}
 The clauses from the statement of the lemma can be complemented by
 \begin{align*}
 r_{\langle\rangle}(\mathcal R_\varphi P_0P_1)&=\begin{cases}
                                                 (\cut,\varphi_b) & \text{if $r_{\langle\rangle}(P_0)=(\bigvee,\neg\varphi,b)$ for some $b\in\iota_\alpha(\neg\varphi)=\iota_\alpha(\varphi)$},\\
                                                 r_{\langle\rangle}(P_0) & \text{otherwise},
                                                \end{cases}\\
 n(\mathcal R_\varphi P_0P_1,a)&=\begin{cases}
                                  \mathcal I_{\varphi,b}P_1 & \text{if $r_{\langle\rangle}(P_0)=(\bigvee,\neg\varphi,b)$ and $a=1$},\\
                                  \mathcal R_\varphi n(P_0,a)P_1 & \text{otherwise}.
                                 \end{cases}
\end{align*}
It remains to verify the local correctness conditions (L), (C1) and (C2). Let us consider the crucial case of a rule $r_{\langle\rangle}(P_0)=(\bigvee,\neg\varphi,b)$. By condition~(L) for $P_0$ we have $o_{\langle\rangle}(n(P_0,0))<o_{\langle\rangle}(P_0)$. Using Lemma~\ref{lem:ordinal-addition} we can deduce
\begin{multline*}
 o_{\langle\rangle}(n(\mathcal R_\varphi P_0P_1,0))=o_{\langle\rangle}(\mathcal R_\varphi n(P_0,0)P_1)=o_{\langle\rangle}(P_1)+o_{\langle\rangle}(n(P_0,0))<\\
 <o_{\langle\rangle}(P_1)+o_{\langle\rangle}(P_0)=o_{\langle\rangle}(\mathcal R_\varphi P_0P_1),
\end{multline*}
as well as
\begin{equation*}
 o_{\langle\rangle}(n(\mathcal R_\varphi P_0P_1,1))=o_{\langle\rangle}(\mathcal I_{\varphi,b}P_1)=o_{\langle\rangle}(P_1)\leq o_{\langle\rangle}(P_1)+o_{\langle\rangle}(n(P_0,0))<o_{\langle\rangle}(\mathcal R_\varphi P_0P_1).
\end{equation*}
Condition (L) for $P_0$ also provides $l_{\langle\rangle}(n(P_0,0))\subseteq l_{\langle\rangle}(P_0)\cup\{\neg\varphi_b\}$, which implies
\begin{equation*}
 l_{\langle\rangle}(n(\mathcal R_\varphi P_0P_1,0))=l_{\langle\rangle}(\mathcal R_\varphi n(P_0,0)P_1)\subseteq l_{\langle\rangle}(\mathcal R_\varphi P_0P_1)\cup\{\neg\varphi_b\}.
\end{equation*}
Furthermore we have
\begin{equation*}
 l_{\langle\rangle}(n(\mathcal R_\varphi P_0P_1,1))=l_{\langle\rangle}(\mathcal I_{\varphi,b}P_1)=l_{\langle\rangle}(P_1)\backslash\{\varphi\}\cup\{\varphi_b\}\subseteq l_{\langle\rangle}(\mathcal R_\varphi P_0P_1)\cup\{\varphi_b\}.
\end{equation*}
In view of $r_{\langle\rangle}(\mathcal R_\varphi P_0P_1)=(\cut,\varphi_b)$ this is just what condition (L) for $\mathcal R_\varphi P_0P_1$ demands. Condition (C1) is satisfied since Lemma~\ref{lem:rank-instances} yields
\begin{equation*}
 \rk(\varphi_b)<\rk(\varphi)\leq d(\mathcal R_\varphi P_0P_1).
\end{equation*}
By condition (C2) for $P_0$ we have $d(n(P_0,0))\leq d(P_0)$, so that we get
\begin{equation*}
 d(n(\mathcal R_\varphi P_0P_1,0))=d(\mathcal R_\varphi n(P_0,0)P_1)\leq d(\mathcal R_\varphi P_0P_1).
\end{equation*}
To complete the verification of condition (C2) for $\mathcal R_\varphi P_0P_1$ we observe
\begin{equation*}
 d(n(\mathcal R_\varphi P_0P_1,1))=d(\mathcal I_{\varphi,b}P_1)=d(P_1)\leq d(\mathcal R_\varphi P_0P_1).
\end{equation*}
Let us also consider the case $r_{\langle\rangle}(P_0)=(\refl,\exists_w\forall_{x\in b}\exists_{y\in w}\theta)$. To verify condition~(L) one should observe that the formulas $\neg\varphi$ and $\exists_w\forall_{x\in b}\exists_{y\in w}\theta$ are different, since we have $\rk(\exists_w\forall_{x\in b}\exists_{y\in w}\theta)=1$ but $\rk(\neg\varphi)=\rk(\varphi)\geq 2$. Also note that $\Omega\leq o_{\langle\rangle}(P_0)$ implies $\Omega\leq o_{\langle\rangle}(P_1)+o_{\langle\rangle}(P_0)=o_{\langle\rangle}(\mathcal R_\varphi P_0P_1)$, using Lemma~\ref{lem:ordinal-addition}. The remaining verifications are straightforward.
\end{proof}

Now we have all ingredients for cut elimination. In the presence of the reflection rule $(\refl,\cdot)$ the cut rank can only be lowered as long as it is bigger than two:

\begin{proposition}\label{prop:cut-elimination}
 We can extend the system of $(u,\alpha)$-codes by a unary function symbol $\mathcal E$, such that we have
\begin{equation*}
 l_{\langle\rangle}(\mathcal EP)=l_{\langle\rangle}(P),\qquad o_{\langle\rangle}(\mathcal EP)=\omega^{o_{\langle\rangle}(P)},\qquad d(\mathcal EP)=\max\{2,d(P)-1\}
\end{equation*}
for any $(u,\alpha)$-code $P$.
\end{proposition}
\begin{proof}
The remaining recursive clauses can be given by
  \begin{align*}
  r_{\langle\rangle}(\mathcal EP)&=\begin{cases}
                                   (\rep,0) & \text{if $r_{\langle\rangle}(P)=(\cut,\varphi)$ with $\rk(\varphi)\geq 2$},\\
                                   r_{\langle\rangle}(P) & \text{otherwise},
                                  \end{cases}\\
  n(\mathcal EP,a)&=\begin{cases}
                    \mathcal R_\varphi(\mathcal E n(P,0))(\mathcal E n(P,1)) & \parbox[t]{.37\textwidth}{if $r_{\langle\rangle}(P)=(\cut,\varphi)$ where $\varphi$ is conjunctive and $\rk(\varphi)\geq 2$,}\\
                    \mathcal R_{\neg\varphi}(\mathcal E n(P,1))(\mathcal E n(P,0)) & \parbox[t]{.37\textwidth}{if $r_{\langle\rangle}(P)=(\cut,\varphi)$ where $\varphi$ is disjunctive and $\rk(\varphi)\geq 2$,}\\
                    \mathcal E n(P,a) & \text{otherwise}.
                   \end{cases}
 \end{align*}
 Let us verify the local correctness conditions in the crucial case $r_{\langle\rangle}(P)=(\cut,\varphi)$, where $\varphi$ is a conjunctive $\mathbf L^u_\alpha$-formula with $\rk(\varphi)\geq 2$. By condition (L) for $P$ we have $o_{\langle\rangle}(n(P,i))<o_{\langle\rangle}(P)$ for $i=0,1$. Using Lemma~\ref{lem:ordinal-addition} we can deduce
 \begin{multline*}
  o_{\langle\rangle}(n(\mathcal EP,0))=o_{\langle\rangle}(\mathcal R_\varphi(\mathcal E n(P,0))(\mathcal E n(P,1)))=o_{\langle\rangle}(\mathcal E n(P,1))+o_{\langle\rangle}(\mathcal E n(P,0))=\\
  =\omega^{o_{\langle\rangle}(n(P,1))}+\omega^{o_{\langle\rangle}(n(P,0))}<\omega^{o_{\langle\rangle}(P)}=o_{\langle\rangle}(\mathcal EP).
 \end{multline*}
 From condition (L) for $P$ we also get $l_{\langle\rangle}(\mathcal E n(P,0))=l_{\langle\rangle}(n(P,0))\subseteq l_{\langle\rangle}(P)\cup\{\neg\varphi\}$. Similarly we have $l_{\langle\rangle}(\mathcal E n(P,1))\subseteq l_{\langle\rangle}(P)\cup\{\varphi\}$, so that we obtain
 \begin{multline*}
  l_{\langle\rangle}(n(\mathcal EP,0))=l_{\langle\rangle}(\mathcal R_\varphi(\mathcal E n(P,0))(\mathcal E n(P,1)))=\\
  =(l_{\langle\rangle}(\mathcal E n(P,0))\backslash\{\neg\varphi\})\cup(l_{\langle\rangle}(\mathcal E n(P,1))\backslash\{\varphi\})\subseteq l_{\langle\rangle}(P)=l_{\langle\rangle}(\mathcal EP).
 \end{multline*}
 In view of $r_{\langle\rangle}(\mathcal EP)=(\rep,0)$ this is just what condition (L) for $\mathcal EP$ demands. Condition (C1) is void in the case of a repetition rule. Using condition (C2) for $P$ it is easy to show $d(\mathcal En(P,i))\leq d(\mathcal EP)$ for $i=0,1$. Also note that condition (C1) for $P$ ensures $\rk(\varphi)\leq d(P)-1\leq d(\mathcal EP)$. Together we get
 \begin{multline*}
  d(n(\mathcal EP,0))=d(\mathcal R_\varphi(\mathcal E n(P,0))(\mathcal E n(P,1)))=\\
  =\max\{d(\mathcal E n(P,0)),d(\mathcal E n(P,1)),\rk(\varphi)\}\leq d(\mathcal EP),
 \end{multline*}
 as required by condition (C2) for $\mathcal EP$. The other verifications are straightforward. To see that the side condition of a reflection rule is preserved one should observe that $\Omega\leq o_{\langle\rangle}(P)$ implies $\Omega\leq\omega^{o_{\langle\rangle}(P)}=o_{\langle\rangle}(\mathcal EP)$, by Lemma~\ref{lem:ordinal-addition}. 
\end{proof}

At the beginning of this section we have introduced the basic $(u,\alpha)$-codes $P^u_\alpha\sigma$ to represent the $(u,\alpha)$-proof $P^u_\alpha=(P^u_\alpha,l^u_\alpha,r^u_\alpha,o^u_\alpha)$ from Proposition~\ref{prop:proof-from-search-tree}. We can now consider the $(u,\alpha)$-code $\mathcal E P^u_\alpha\langle\rangle$. Proposition~\ref{prop:codes-to-proofs} tells us that it represents a $(u,\alpha)$-proof $[\mathcal E P^u_\alpha\langle\rangle]$. The latter has the same end-sequent as $P^u_\alpha$, namely
\begin{equation*}
 l_{\mathcal EP^u_\alpha\langle\rangle}(\langle\rangle)=l_{\langle\rangle}(\bar n(\mathcal EP^u_\alpha\langle\rangle,\langle\rangle))=l_{\langle\rangle}(\mathcal EP^u_\alpha\langle\rangle)=l_{\langle\rangle}(P^u_\alpha\langle\rangle)=l^u_\alpha(\langle\rangle).
\end{equation*}
At the same time the maximal complexity of cut formulas has been reduced, since we have $d(P^u_\alpha\langle\rangle)=C+6$ and $d(\mathcal EP^u_\alpha\langle\rangle)=C+5$. Thus we have managed to implement cut elimination for $(u,\alpha)$-proofs, taking a detour via $(u,\alpha)$-codes.

\section{Operator Control}

We would like to collapse certain proofs to height below $\Omega$, using the function
\begin{equation*}
 \bar\vartheta:\varepsilon(S)^u_\alpha\rightarrow\varepsilon(S)^u_\alpha\cap\Omega
\end{equation*}
from Section~\ref{section:epsilon-variant}. A complication arises from the fact that $\bar\vartheta$ is not fully monotone. Thus the inequality $o_{\langle\rangle}(n(P,a))<o_{\langle\rangle}(P)$ required by condition~(L) may not always be preserved. Buchholz~\cite{buchholz-local-predicativity} has introduced operator control as an elegant way to ensure that $\bar\vartheta$ is order preserving in all relevant cases. In the present section we adapt this approach to our setting. The first step is to define suitable operators. As before we identify the isomorphic orders $\alpha\cong\varepsilon(S)^u_\alpha\cap\Omega$.

\begin{definition}
 Given $t\in\varepsilon(S)^u_\alpha$ and $x\in[\alpha]^{<\omega}$, we construct $\mathcal H_t(x)\subseteq\varepsilon(S)^u_\alpha$ by
 \begin{align*}
  \mathcal H^0_t(x)&=\{s\in\varepsilon(S)^u_\alpha\,|\,\suppe_\alpha(s)\leq^{\operatorname{fin}}x\},\\
  \mathcal H^{n+1}_t(x)&\!\begin{aligned}[t]
                          {}={}&\{\bar\vartheta(s)\,|\,s\in\mathcal H^n_t(x)\text{ and }s\leq t\}\cup{}\\
                          &{}\cup\{s\in\varepsilon(S)^u_\alpha\,|\,\suppe_\alpha(s)\leq^{\operatorname{fin}}y\text{ for some finite set $y\subseteq\mathcal H^n_t(x)\cap\Omega$}\},
                         \end{aligned}\\
  \mathcal H_t(x)&=\textstyle\bigcup_{n\in\omega}\mathcal H^n_t(x).
 \end{align*}
\end{definition}

Let us observe some basic properties:

\begin{lemma}\label{lem:operators-basic}
 The following holds for any number $n$:
 \begin{enumerate}[label=(\alph*)]
  \item If we have $s\in\mathcal H^n_t(x)$, then we have $\suppe_\alpha(s)\subseteq\mathcal H^n_t(x)$.
  \item If we have $\suppe_\alpha(s)\subseteq\mathcal H^n_t(x)$, then we have $s\in\mathcal H^{n+1}_t(x)$.
  \item We have $\mathcal H^n_t(x)\subseteq\mathcal H^{n+1}_t(x)$.
 \end{enumerate}
It follows that $s\in\mathcal H_t(x)$ is equivalent to $\suppe_\alpha(s)\subseteq\mathcal H_t(x)$.
\end{lemma}
\begin{proof}
 (a) The case $s=\bar\vartheta(s')$ is trivial, as Lemma~\ref{lem:collapse-support} yields $\suppe_\alpha(s)=\{s\}$. Now assume that $s\in\mathcal H^n_t(x)$ holds because we have $\suppe_\alpha(s)\leq^{\operatorname{fin}}y$ with $y=x$ or~$y\subseteq\mathcal H^{n-1}_t(x)\cap\Omega$. For $r\in\suppe_\alpha(s)$ we get $\suppe_\alpha(r)\leqf y$ by Lemma~\ref{lem:collapse-support}. We can conclude $r\in\mathcal H^n_t(x)$ by construction.\\
 (b) It suffices to invoke the definition of $\mathcal H_t^{n+1}(x)$ with $y=\suppe_\alpha(s)$.\\
 (c) This follows from parts (a) and (b).
\end{proof}

As pointed out by Buchholz~\cite{buchholz-local-predicativity}, it is crucial that $\mathcal H_t$ is a closure operator:

\begin{lemma}\label{lem:operators-closure}
 The following holds for all sets $x,y\in[\alpha]^{<\omega}$:
 \begin{enumerate}[label=(\alph*)]
  \item We have $x\subseteq\mathcal H_t(x)$.
  \item If we have $x\subseteq\mathcal H_t(y)$, then we get $\mathcal H_t(x)\subseteq\mathcal H_t(y)$.
 \end{enumerate}
\end{lemma}
\begin{proof}
 (a) By Lemma~\ref{lem:collapse-support} we get $\suppe_\alpha(s)\leqf x$ for any $s\in x$. Then $s\in\mathcal H^0_t(x)$ holds by construction.\\
 (b) In view of the previous lemma we have $x\subseteq\mathcal H^m_t(y)$ for some $m$. A straightforward induction on $n$ shows $\mathcal H^n_t(x)\subseteq\mathcal H^{m+1+n}_t(y)$.
\end{proof}

Let us also show that $\mathcal H_t(x)$ is closed under basic ordinal arithmetic:

\begin{lemma}\label{lem:operators-arithmetic}
 The following holds for any $t\in\varepsilon(S)^u_\alpha$:
 \begin{enumerate}[label=(\alph*)]
  \item We have $0\in\mathcal H_t(\emptyset)$ and $\Omega\in\mathcal H_t(\emptyset)$.
  \item We have $\mathfrak E_\sigma\in\mathcal H_t(\supps_\alpha(\sigma))$ for all $\sigma\in S^u_\alpha$.
  \item If we have $s,s'\in\mathcal H_t(x)$, then we have $s+s'\in\mathcal H_t(x)$ and $\omega^s\in\mathcal H_t(x)$.
  \item If we have $s<s'$ for some $s'\in\mathcal H_t(x)\cap\Omega$, then we have $s\in\mathcal H_t(x)$.
 \end{enumerate}
\end{lemma}
\begin{proof}
 (a) In view of $\suppe_\alpha(0)=\suppe_\alpha(\Omega)=\emptyset$ this follows from Lemma~\ref{lem:operators-basic}.\\
 (b) It suffices to observe $\suppe_\alpha(\mathfrak E_\sigma)=\supps_\alpha(\sigma)\subseteq\mathcal H_t(\supps_\alpha(\sigma))$.\\
 (c) From Lemma~\ref{lem:ordinal-addition} and Lemma~\ref{lem:operators-basic} we get
 \begin{equation*}
  \suppe_\alpha(s+s')\subseteq\suppe_\alpha(s)\cup\suppe_\alpha(s')\subseteq\mathcal H_t(x),
 \end{equation*}
 which implies $s+s'\in\mathcal H_t(x)$. By $\suppe_\alpha(\omega^s)=\suppe_\alpha(s)$ we also get $\omega^s\in\mathcal H_t(x)$.\\
 (d) By Lemma~\ref{lem:collapse-support} we have $\suppe_\alpha(s)\lef s'$. Now it suffices to invoke the definition of $\mathcal H_t(x)$ with $y=\{s'\}$.
\end{proof}

The relation between operators and collapsing values is particularly important:

\begin{proposition}\label{prop:operators-collapse}
 The following holds:
 \begin{enumerate}[label=(\alph*)]
  \item For $t<t'$ we have $\mathcal H_t(x)\subseteq\mathcal H_{t'}(x)$.
  \item From $s\in\mathcal H_t(x)$ and $s\leq t$ we can infer $\bar\vartheta(s)\in\mathcal H_t(x)$.
  \item If we have $s\in\mathcal H_t(\emptyset)\cap\Omega$ and $t<t'$, then we have $s<\bar\vartheta(t')$.
  \item If we have $s,t\in\mathcal H_t(\emptyset)$ and $s<s'$, then we have $\bar\vartheta(t+\omega^s)<\bar\vartheta(t+\omega^{s'})$.
 \end{enumerate}
\end{proposition}
\begin{proof}
 (a) It is straightforward to establish $\mathcal H^n_t(x)\subseteq\mathcal H^n_{t'}(x)$ by induction on $n$.\\
 (b) This holds by the definition of $\mathcal H_t(x)$.\\
 (c) We prove the claim for $s\in\mathcal H^n_t(\emptyset)\cap\Omega$ by induction on $n$: First assume that $s\in\mathcal H^n_t(\emptyset)$ holds because we have $\suppe_\alpha(s)\leqf y$, with $y=\emptyset$ or $y\subseteq\mathcal H^{n-1}_t(\emptyset)\cap\Omega$. Then the induction hypothesis provides $\suppe_\alpha(s)\lef\bar\vartheta(t')$. By Lemma~\ref{lem:collapse-support} we get $s<\bar\vartheta(t')$. Now assume that we have $s=\bar\vartheta(s')$ with $s'\in\mathcal H^{n-1}_t(\emptyset)$ and $s'\leq t<t'$. Lemma~\ref{lem:operators-basic} ensures $\suppe_\alpha(s')\subseteq\mathcal H^{n-1}_t(\emptyset)$, so that the induction hypothesis yields $\suppe_\alpha(s')\lef\bar\vartheta(t')$. By Proposition~\ref{prop:bar-collapse} we get $s=\bar\vartheta(s')<\bar\vartheta(t')$, as desired.\\
 (d) Using the previous lemma and part (a) we get $t+\omega^s\in\mathcal H_t(\emptyset)\subseteq\mathcal H_{t+\omega^s}(\emptyset)$. By part (b) this yields $\bar\vartheta(t+\omega^s)\in\mathcal H_{t+\omega^s}(\emptyset)\cap\Omega$. Also note that $s<s'$ implies $t+\omega^s<t+\omega^{s'}$, by Lemma~\ref{lem:ordinal-addition}. Now the claim follows from part (c).
\end{proof}

To relate operators and infinite proofs we use the functions $h_0$ and $h_1$ mentioned in Remark~\ref{rmk:codes-methodology}. We have explained that these functions are to be defined by recursion over $(u,\alpha)$-codes. The following definition covers the codes that we have introduced so far. Whenever we extend the system of $(u,\alpha)$-codes by a new function symbol we will add corresponding recursive clauses. Let us point out that the clauses from the proof of Theorem~\ref{thm:collapsing} will lead to non-zero values of the function~$h_0$.

\begin{definition}\label{def:operators-for-proofs}
 The functions
 \begin{equation*}
  h_0:\text{``$(u,\alpha)$-codes''}\rightarrow\varepsilon(S)^u_\alpha\qquad\text{and}\qquad h_1:\text{``$(u,\alpha)$-codes''}\rightarrow[\alpha]^{<\omega}
 \end{equation*}
 are defined by the clauses
 \begin{align*}
h_0(P^u_\alpha\sigma)&=0,& h_1(P^u_\alpha\sigma)&=\supps_\alpha(\sigma),\\
h_0(\mathcal I_{\varphi,a}P)&=h_0(P),& h_1(\mathcal I_{\varphi,a}P)&=h_1(P)\cup\suppl_\alpha(a),\\
h_0(\mathcal R_\varphi P_0P_1)&=\max\{h_0(P_0),h_0(P_1)\},& h_1(\mathcal R_\varphi P_0P_1)&=h_1(P_0)\cup h_1(P_1),\\
h_0(\mathcal E P)&=h_0(P),& h_1(\mathcal E P)&=h_1(P).
\end{align*}
 We use the abbreviation $\mathcal H_P(x)=\mathcal H_{h_0(P)}(h_1(P)\cup x)$.
\end{definition}

The idea is that all relevant parameters at a node $\sigma\in[P]$ should be captured by the set $\mathcal H_P(\supps_\alpha(\sigma))$. In particular we will want to consider the support
\begin{equation*}
 \supp(\varphi)=\bigcup\{\suppl_\alpha(a)\,|\,\text{the parameter $a\in\mathbf L^u_\alpha$ occurs in $\varphi$}\}\in[\alpha]^{<\omega}
\end{equation*}
of an $\mathbf L^u_\alpha$-formula $\varphi$. Operator control is implemented via the local correctness conditions (H1) to (H3), which have also been mentioned in Remark~\ref{rmk:codes-methodology}:

\begin{proposition}\label{prop:local-correctness-operator}
The following holds for any $(u,\alpha)$-code $P$:
\begin{enumerate}[label=(H\arabic*)]
\item We have $o_{\langle\rangle}(P)\in\mathcal H_P(\emptyset)$.
\item If $r_{\langle\rangle}(P)$ is of the form $(\bigwedge,\varphi)$, then we have $\supp(\varphi)\subseteq\mathcal H_P(\emptyset)$. If $r_{\langle\rangle}(P)$ is of the form $(\bigvee,\varphi,a)$ or $(\rep,a)$, then we have $\suppl_\alpha(a)\subseteq\mathcal H_P(\emptyset)$.
\item We have $h_0(n(P,a))\leq h_0(P)$ and $h_1(n(P,a))\subseteq\mathcal H_P(\suppl_\alpha(a))$ for any element~$a\in\iota(P)$.
\end{enumerate}
\end{proposition}
\begin{proof}
 The conditions are established by induction over the $(u,\alpha)$-code $P$, as explained in Remark~\ref{rmk:codes-methodology}. Let us begin with the case of a basic $(u,\alpha)$-code $P=P^u_\alpha\sigma$. Condition (H3) is satisfied because $a\in\iota(P^u_\alpha\sigma)=\iota(r^u_\alpha(\sigma))$ implies $\sigma^\frown a\in P^u_\alpha$, so that we get $h_0(n(P^u_\alpha\sigma,a))=h_0(P^u_\alpha\,\sigma^\frown a)=0=h_0(P^u_\alpha\sigma)$ and
 \begin{multline*}
  h_1(n(P^u_\alpha\sigma,a))=h_1(P^u_\alpha\,\sigma^\frown a)=\supps_\alpha(\sigma^\frown a)=\supps_\alpha(\sigma)\cup\suppl_\alpha(a)\subseteq\\
  \subseteq\mathcal H_0(\supps_\alpha(\sigma)\cup\suppl_\alpha(a))=\mathcal H_{P^u_\alpha\sigma}(\suppl_\alpha(a)).
 \end{multline*}
 To verify the other conditions we distinguish two cases, following the proof of Proposition~\ref{prop:proof-from-search-tree}: First assume $\sigma\in S^u_\alpha\subseteq P^u_\alpha$. Then condition (H1) holds, since Lemma~\ref{lem:operators-arithmetic} yields
 \begin{equation*}
  o_{\langle\rangle}(P^u_\alpha\sigma)=o^u_\alpha(\sigma)=\mathfrak E_\sigma\in\mathcal H_0(\supps_\alpha(\sigma))=\mathcal H_{P^u_\alpha\sigma}(\emptyset).
 \end{equation*}
 If $\sigma$ has even length, then $r_{\langle\rangle}(P^u_\alpha\sigma)=r^u_\alpha(\sigma)$ is a cut rule and condition (H2) is void. Now assume that $\sigma$ has odd length $2n+1$. We write $\varphi$ for the $\pi_0(n)$-th formula of~$l_\alpha(\sigma)$. If $\varphi$ is conjunctive, then we have $r_{\langle\rangle}(P^u_\alpha\sigma)=(\bigwedge,\varphi)$. Corollary~\ref{cor:operator-control-search-trees} yields
 \begin{equation*}
  \supp(\varphi)\subseteq\supps_\alpha(\sigma)\subseteq\mathcal H_{P^u_\alpha\sigma}(\emptyset),
 \end{equation*}
 as required for condition (H2). If $\varphi$ is disjunctive, then we consider
 \begin{equation*}
  b=\en_\alpha(\supps_\alpha(\sigma\!\restriction\!\pi_1(n)),\pi_2(n)).
 \end{equation*}
 In case $b\in\iota_\alpha(\varphi)$ we have $r_{\langle\rangle}(P^u_\alpha\sigma)=(\bigvee,\varphi,b)$. Invoking Corollary~\ref{cor:support-enumeration} we get
 \begin{equation*}
  \suppl_\alpha(b)\subseteq\supps_\alpha(\sigma\!\restriction\!\pi_1(n))\subseteq\supps_\alpha(\sigma)\subseteq\mathcal H_{P^u_\alpha\sigma}(\emptyset),
 \end{equation*}
 as condition (H2) demands. In case $b\notin\iota_\alpha(\varphi)$ we have $r_{\langle\rangle}(P^u_\alpha\sigma)=(\rep,0)$. Here it suffices to observe $\suppl_\alpha(0)=\emptyset$ (recall that Assumption~\ref{ass:properties-u} ensures $0\in u$). Now assume that $\sigma\in P^u_\alpha$ is of the form $\sigma={\sigma_0}^\frown 1^\frown\tau$, where $\sigma_0\in S^u_\alpha$ has length~$2n$ and we have~$\tau\in P_n$ (cf.~the proof of Proposition~\ref{prop:proof-from-search-tree}). It is straightforward to check conditions~(H1) and~(H2) explicitly. As an example we consider $\sigma_0=\langle\rangle$ and~$\tau=\langle a\rangle\in P_0$. According to the proof of Lemma~\ref{lem:P_0} we have
 \begin{equation*}
  o_{\langle\rangle}(P^u_\alpha\sigma)=o^u_\alpha(\sigma)=o_0(\tau)=\beta_a+1\quad\text{with}\quad\beta_a=\sup\{\beta+1\,|\,\beta\in\suppl_\alpha(a)\}.
 \end{equation*}
 Lemma~\ref{lem:operators-arithmetic} yields $\beta_a+1\in\mathcal H_0(\suppl_\alpha(a))\subseteq\mathcal H_{P^u_\alpha\sigma}(\emptyset)$, as needed for condition~(H1). Furthermore we have
 \begin{equation*}
  r_{\langle\rangle}(P^u_\alpha\sigma)=r_0(\tau)=(\bigvee,\exists_y\,y=a\cup\{a\},b_a)\quad\text{with}\quad b_a=\{z\in L^u_{\beta_a}\,|\,z\in a\lor z=a\}.
 \end{equation*}
 Condition (H2) is satisfied in view of $\suppl_\alpha(b_a)=\{\beta_a\}\cup\suppl_\alpha(a)\subseteq\mathcal H_0(\suppl_\alpha(a))$. We have verified conditions (H1) to (H3) for all basic $(u,\alpha)$-codes $P=P^u_\alpha\sigma$. It remains to prove the induction steps for $(u,\alpha)$-codes of the form $\mathcal I_{\varphi,a}P_0$, $\mathcal R_\varphi P_0P_1$ and~$\mathcal EP_0$. As an example we consider a term $P=\mathcal R_\varphi P_0P_1$ with $r_{\langle\rangle}(P_0)=(\bigvee,\neg\varphi,b)$. Condition (H1) for $P_i$ yields $o_{\langle\rangle}(P_i)\in\mathcal H_{P_i}(\emptyset)$ for $i=0,1$. Using Lemma~\ref{lem:operators-closure} and Proposition~\ref{prop:operators-collapse}(a) we can infer $o_{\langle\rangle}(P_i)\in\mathcal H_P(\emptyset)$. Together with Lemma~\ref{lem:operators-arithmetic} we get
 \begin{equation*}
  o_{\langle\rangle}(P)=o_{\langle\rangle}(P_1)+o_{\langle\rangle}(P_0)\in\mathcal H_P(\emptyset),
 \end{equation*}
 as required by condition (H1) for $P$. Condition (H2) is void in the present case, since we have $r_{\langle\rangle}(P)=(\cut,\varphi_b)$. Using condition (H3) for $P_0$ we get
 \begin{multline*}
  h_0(n(P,0))=h_0(\mathcal R_\varphi n(P_0,0)P_1)=\max\{h_0(n(P_0,0)),h_0(P_1)\}\leq\\
  \leq\max\{h_0(P_0),h_0(P_1)\}=h_0(P).
 \end{multline*}
 Together with Lemma~\ref{lem:operators-closure} we can show
 \begin{multline*}
  h_1(n(P,0))=h_1(\mathcal R_\varphi n(P_0,0)P_1)=h_1(n(P_0,0))\cup h_1(P_1)\subseteq\\
  \subseteq\mathcal H_{P_0}(\suppl_\alpha(0))\cup\mathcal H_{P_1}(\emptyset)\subseteq\mathcal H_P(\suppl_\alpha(0)).
 \end{multline*}
 Even without the induction hypothesis we see
 \begin{equation*}
  h_0(n(P,1))=h_0(\mathcal I_{\varphi,b}P_1)=h_0(P_1)\leq h_0(P).
 \end{equation*}
 Crucially, condition (H2) for $P_0$ ensures $\suppl_\alpha(b)\subseteq\mathcal H_{P_0}(\emptyset)$. We can deduce
 \begin{equation*}
  h_1(n(P,1))=h_1(\mathcal I_{\varphi,b}P_1)=h_1(P_1)\cup\suppl_\alpha(b)\subseteq\mathcal H_P(\emptyset),
 \end{equation*}
 which completes the inductive verification of condition (H3) for $P=\mathcal R_\varphi P_0P_1$. The remaining cases are checked similarly.
\end{proof}

\section{Collapsing}\label{sect:collapsing}

In this section we show that suitable $(u,\alpha)$-proofs can be collapsed to proofs of height below $\Omega$. Using this result we complete the proof of our main theorem: the abstract Bachmann-Howard principle implies the existence of admissible sets. The required collapsing procedure for infinite proofs originates from J\"ager's~\cite{jaeger-kripke-platek} ordinal analysis of Kripke-Platek set theory. We also rely on Buchholz'~\cite{buchholz-local-predicativity} presentation of impredicative ordinal analysis in terms of operator controlled derivations.

When we collapse a proof we will need to relativize certain formulas that it contains: Consider an $\mathbf L^u_\alpha$-formula $\varphi$ and an ordinal $\gamma<\alpha$. We write $\varphi^\gamma$ for the $\mathbf L^u_\alpha$-formula that results from $\varphi$ when we replace all unbounded quantifiers $\exists_x$ and $\forall_x$ by the bounded quantifiers $\exists_{x\in L^u_\gamma}$ resp.~$\forall_{x\in L^u_\gamma}$. By a $\Sigma(\alpha)$-formula (resp.~$\Pi(\alpha)$-formula) we mean an $\mathbf L^u_\alpha$-formula that contains no unbounded universal (resp.~existential) quantifiers. Let us relate these notions to Definition~\ref{def:disj-conj-formula}:

\begin{lemma}\label{lem:relativization-properties}
 The following holds for any $\mathbf L^u_\alpha$-formula $\varphi$ and any ordinal $\gamma<\alpha$:
 \begin{enumerate}[label=(\alph*)]
  \item If $\varphi$ is a $\Sigma(\alpha)$-formula resp.~$\Pi(\alpha)$-formula, then so is $\varphi_a$ for any $a\in\iota_\alpha(\varphi)$.
  \item If $\varphi$ is disjunctive resp.~conjunctive, then so is $\varphi^\gamma$.
  \item We have $\iota_\alpha(\varphi^\gamma)\subseteq\iota_\alpha(\varphi)$, as well as $(\varphi^\gamma)_a=(\varphi_a)^\gamma$ for any $a\in\iota_\alpha(\varphi^\gamma)$.
  \item If we have $a\in\iota_\alpha(\varphi)$ and $\suppl_\alpha(a)\lef\gamma$, then we have $a\in\iota_\alpha(\varphi^\gamma)$.
  \item If $\varphi$ is a disjunctive $\Pi(\alpha)$-formula, then we have $\iota_\alpha(\varphi^\gamma)=\iota_\alpha(\varphi)$.
 \end{enumerate}
\end{lemma}
\begin{proof}
 The claims can be verified explicitly for all cases from Definition~\ref{def:disj-conj-formula}. As an example we consider a disjunctive formula
 \begin{equation*}
  \varphi=\exists_{x\in\{y\in L^u_\delta\,|\,\theta(y,\vec c)\}}\psi(x)\simeq\textstyle\bigvee_{\suppl_\alpha(a)\lef\delta}\theta(a,\vec c)\land\psi(a).
 \end{equation*}
 In view of Definition~\ref{def:term-version-L} the formula $\theta$ must be bounded. Thus $\varphi_a=\theta(a,\vec c)\land\psi(a)$ is a $\Sigma(\alpha)$-formula (resp.~$\Pi(\alpha)$-formula) whenever the same holds for $\varphi$. Now observe
 \begin{equation*}
  \varphi^\gamma=\exists_{x\in\{y\in L^u_\delta\,|\,\theta(y,\vec c)\}}\psi^\gamma(x)\simeq\textstyle\bigvee_{\suppl_\alpha(a)\lef\delta}\theta(a,\vec c)\land\psi(a)^\gamma.
 \end{equation*}
 So $\varphi^\gamma$ is disjunctive and we have $\iota_\alpha(\varphi^\gamma)=\{a\in\mathbf L^u_\alpha\,|\,\suppl_\alpha(a)\lef\delta\}=\iota_\alpha(\varphi)$. As $\theta$ is bounded we have $\theta(a,\vec c)^\gamma=\theta(a,\vec c)$ and thus $(\varphi^\gamma)_a=\theta(a,\vec c)\land\psi(a)^\gamma=(\varphi_a)^\gamma$. Once all disjunctive cases are verified, the conjunctive cases follow by duality (recall that we have $\iota_\alpha(\neg\varphi)=\iota_\alpha(\varphi)$ and $(\neg\varphi)_a=\neg(\varphi_a)$, and observe $(\neg\varphi)^\gamma=\neg(\varphi^\gamma)$).
\end{proof}

The following result covers two proof transformations that are often presented separately: Given a proof of the sequent $\Gamma,\varphi$ with height $\gamma\in\alpha\cong\varepsilon(S)^u_\alpha\cap\Omega$, we can construct a proof of $\Gamma,\varphi^\gamma$. If $\varphi$ is a $\Pi(\alpha)$-formula, then any proof of $\Gamma,\varphi$ (possibly with height above $\Omega$) can be transformed into a proof of $\Gamma,\varphi^\gamma$.

\begin{lemma}\label{lem:boundedness}
 For any $\mathbf L^u_\alpha$-formula $\varphi$ and any ordinal $\gamma<\alpha$ we can extend the system of $(u,\alpha)$-codes by a unary function symbol $\mathcal B_{\varphi,\gamma}$, such that we have
 \begin{gather*}
  l_{\langle\rangle}(\mathcal B_{\varphi,\gamma}P)=\begin{cases}
                                                           (l_{\langle\rangle}(P)\backslash\{\varphi\})\cup\{\varphi^\gamma\} & \text{if $o_{\langle\rangle}(P)\leq\gamma$ or if $\varphi$ is a $\Pi(\alpha)$-formula},\\
                                                           l_{\langle\rangle}(P) & \text{otherwise},
                                                          \end{cases}\\
  \begin{aligned}
   o_{\langle\rangle}(\mathcal B_{\varphi,\gamma}P)&=o_{\langle\rangle}(P),\qquad & d(\mathcal B_{\varphi,\gamma}P)&=d(P),\\
   h_0(\mathcal B_{\varphi,\gamma}P)&=h_0(P),\qquad & h_1(\mathcal B_{\varphi,\gamma}P)&=h_1(P)\cup\{\gamma\}
  \end{aligned}
 \end{gather*}
 for any $(u,\alpha)$-code $P$.
\end{lemma}
\begin{proof}
 As explained in Remark~\ref{rmk:codes-methodology}, we must complement the clauses from the lemma by recursive clauses for the functions $r_{\langle\rangle}$ and $n$. We must also prove the induction step for the local correctness conditions (L), (C1), (C2) and (H1) to~(H3). Let us first consider the ``unintended'' case, i.e.~we assume that we have $\gamma<o_{\langle\rangle}(P)$ and that $\varphi$ fails to be a $\Pi(\alpha)$-formula. In this situation we stipulate that $\mathcal B_{\varphi,\gamma}P$ behaves like $P$, i.e.~we set $r_{\langle\rangle}(\mathcal B_{\varphi,\gamma}P)=r_{\langle\rangle}(P)$ and $n(\mathcal B_{\varphi,\gamma}P,a)=n(P,a)$. Then the local correctness conditions for $\mathcal B_{\varphi,\gamma}P$ follow from the same conditions for $P$. In the rest of the proof we consider the ``intended'' case, i.e.~we assume that we have $o_{\langle\rangle}(P)\leq\gamma$ or that $\varphi$ is a $\Pi(\alpha)$-formula. In this situation we put
 \begin{align*}
 r_{\langle\rangle}(\mathcal B_{\varphi,\gamma}P)&=\begin{cases}
 (\bigwedge,\varphi^\gamma) & \text{if $r_{\langle\rangle}(P)=(\bigwedge,\varphi)$},\\
 (\bigvee,\varphi^\gamma,b) & \text{if $r_{\langle\rangle}(P)=(\bigvee,\varphi,b)$},\\
 r_{\langle\rangle}(P) & \text{otherwise},
 \end{cases}\\
 n(\mathcal B_{\varphi,\gamma}P,a)&=\begin{cases}
 \mathcal B_{\varphi_a,\gamma}\mathcal B_{\varphi,\gamma}n(P,a) & \text{if $r_{\langle\rangle}(P)=(\bigwedge,\varphi)$ and $a\in\iota_\alpha(\varphi)$},\\
 \mathcal B_{\varphi_b,\gamma}\mathcal B_{\varphi,\gamma}n(P,a) & \text{if $r_{\langle\rangle}(P)=(\bigvee,\varphi,b)$},\\
 \mathcal B_{\varphi,\gamma}n(P,a) & \text{otherwise}.
 \end{cases}
\end{align*}
In the case of a rule $r_{\langle\rangle}(P)=(\bigvee,\varphi,b)$ it is crucial to observe that condition (L) for~$P$ ensures $\suppl_\alpha(b)\lef o_{\langle\rangle}(P)$. By parts (d) and (e) of the previous lemma we get $b\in\iota_\alpha(\varphi^\gamma)$, so that $r_{\langle\rangle}(\mathcal B_{\varphi,\gamma}P)$ is indeed an $\mathbf L^u_\alpha$-rule. To show how local correctness is verified we consider the case $r_{\langle\rangle}(P)=(\bigwedge,\varphi)$. The crucial observation is that $\mathcal B_{\varphi_a,\gamma}\mathcal B_{\varphi,\gamma}n(P,a)$ and $\mathcal B_{\varphi,\gamma}n(P,a)$ are evaluated according to the ``intended'' case, for any $a\in\iota_\alpha(\varphi^\gamma)\subseteq\iota_\alpha(\varphi)$: If we have $o_{\langle\rangle}(P)\leq\gamma$, then condition (L) for $P$ yields
\begin{equation*}
 o_{\langle\rangle}(\mathcal B_{\varphi,\gamma}n(P,a))=o_{\langle\rangle}(n(P,a))<o_{\langle\rangle}(P)\leq\gamma.
\end{equation*}
If $\varphi$ is a $\Pi(\alpha)$-formula, then the previous lemma ensures that $\varphi_a$ is a $\Pi(\alpha)$-formula as well. It follows that we have
\begin{equation*}
 l_{\langle\rangle}(n(\mathcal B_{\varphi,\gamma}P,a))=l_{\langle\rangle}(\mathcal B_{\varphi_a,\gamma}\mathcal B_{\varphi,\gamma}n(P,a))=(l_{\langle\rangle}(n(P,a))\backslash\{\varphi,\varphi_a\})\cup\{\varphi^\gamma,\varphi_a^\gamma\}.
\end{equation*}
Note that the previous lemma allows us to write $\varphi^\gamma_a=(\varphi_a)^\gamma=(\varphi^\gamma)_a$. Condition~(L) for $P$ provides $l_{\langle\rangle}(n(P,a))\subseteq l_{\langle\rangle}(P)\cup\{\varphi_a\}$. Thus we get
\begin{equation*}
 l_{\langle\rangle}(n(\mathcal B_{\varphi,\gamma}P,a))\subseteq(l_{\langle\rangle}(P)\backslash\{\varphi\})\cup\{\varphi^\gamma,\varphi_a^\gamma\}=l_{\langle\rangle}(\mathcal B_{\varphi,\gamma}P)\cup\{\varphi_a^\gamma\},
\end{equation*}
as required by condition (L) for $\mathcal B_{\varphi,\gamma}P$. To establish condition (H2) we observe
\begin{equation*}
 \supp(\varphi^\gamma)\subseteq\supp(\varphi)\cup\suppl_\alpha(L^u_\gamma)=\supp(\varphi)\cup\{\gamma\}.
\end{equation*}
 By definition we have $\gamma\in h_1(\mathcal B_{\varphi,\gamma}P)\subseteq\mathcal H_{\mathcal B_{\varphi,\gamma}P}(\emptyset)$. Also note that condition~(H2) for $P$ yields $\supp(\varphi)\subseteq\mathcal H_P(\emptyset)$. Together we obtain $\supp(\varphi^\gamma)\subseteq\mathcal H_{\mathcal B_{\varphi,\gamma}P}(\emptyset)$, as condition~(H2) for $\mathcal B_{\varphi,\gamma}P$ demands. The remaining verifications are similar. Concerning the case of a reflection rule $r_{\langle\rangle}(P)=(\refl,\exists_w\forall_{x\in b}\exists_{y\in w}\theta)$ we remark that the formulas $\varphi$ and $\exists_w\forall_{x\in b}\exists_{y\in w}\theta$ must be different: Invoking condition (L) for $P$ we see $\gamma<\Omega\leq o_{\langle\rangle}(P)$, and $\exists_w\forall_{x\in b}\exists_{y\in w}\theta$ is not a $\Pi(\alpha)$-formula.
\end{proof}

Let us now single out the proofs that can be collapsed to height below $\Omega$:

\begin{definition}\label{def:t-collapsing}
Consider a term $t\in\varepsilon(S)^u_\alpha$ with $t\in\mathcal H_t(\emptyset)$. A $(u,\alpha)$-code $P$ is called $t$-controlled if the following conditions are satisfied:
\begin{enumerate}[label=(\roman*)]
\item The end-sequent $l_{\langle\rangle}(P)$ of $P$ consists of $\Sigma(\alpha)$-formulas.
\item We have $h_0(P)\leq t$ and $h_1(P)\subseteq\mathcal H_t(\emptyset)$.
\end{enumerate}
If $P$ is $t$-controlled and has cut rank $d(P)\leq 2$, then $P$ is called $t$-collapsing.
\end{definition}

The restriction to $\Sigma(\alpha)$-formulas and proofs of low cut rank can be explained in view of the following facts:

\begin{lemma}\label{lem:instances-conjunctive-Sigma}
The following holds for any $\mathbf L^u_\alpha$-formula $\varphi$:
\begin{enumerate}[label=(\alph*)]
\item If $\varphi$ is a conjunctive $\Sigma(\alpha)$-formula, then we have $\suppl_\alpha(a)\lef\supp(\varphi)$ for all $a\in\iota_\alpha(\varphi)$.
\item If we have $\rk(\varphi)\leq 1$ and $\varphi$ is disjunctive (resp.~conjunctive), then $\varphi$ is a $\Sigma(\alpha)$-formula (resp.~$\Pi(\alpha)$-formula).
\end{enumerate}
\end{lemma}
\begin{proof}
(a) Based on Definition~\ref{def:disj-conj-formula}, the claim can be checked explicitly for all possible forms of $\varphi$. The point is that $\varphi$ cannot begin with an unbounded quantifier: Note that $\forall_x\psi$ is no $\Sigma(\alpha)$-formula while $\exists_x\psi$ is not conjunctive. As a positive example we consider
\begin{equation*}
\varphi=\forall_{x\in\{y\in L^u_\gamma\,|\,\theta(y,\vec c)\}}\psi(x)\simeq\textstyle\bigwedge_{\suppl_\alpha(a)\lef\gamma}\neg\theta(a,\vec c)\lor\psi(a).
\end{equation*}
Since the $\mathbf L^u_\alpha$-term $\{y\in L^u_\gamma\,|\,\theta(y,\vec c)\}$ is a parameter of $\varphi$ we have
\begin{equation*}
\suppl_\alpha(a)\lef\gamma\in\suppl_\alpha(\{y\in L^u_\gamma\,|\,\theta(y,\vec c)\})\subseteq\supp(\varphi)
\end{equation*}
for any $a\in\iota_\alpha(\varphi)$.\\
(b) In view of Definition~\ref{def:formula-rank} any formula of rank zero must be bounded. Thus a formula of rank one must be of the form $\exists_x\theta$ or $\forall_x\theta$ with a bounded formula $\theta$.
\end{proof}

In Section~\ref{section:epsilon-variant} we have considered a collapsing function
\begin{equation*}
 \bar\vartheta:\varepsilon(S)^u_\alpha\rightarrow\varepsilon(S)^u_\alpha\cap\Omega\cong\alpha
\end{equation*}
on our ordinal notation system. Together with the previous proof transformations it allows us to collapse infinite proofs to height below $\Omega$ (cf.~the ``Kollabierungs\-lemma'' in J\"ager's~\cite{jaeger-kripke-platek} ordinal analysis of Kripke-Platek set theory):

\begin{theorem}\label{thm:collapsing}
For any $t\in\varepsilon(S)^u_\alpha$ we can extend the system of $(u,\alpha)$-codes by a unary function symbol $\mathcal C_t$, such that we have
\begin{align*}
l_{\langle\rangle}(\mathcal C_tP)&=l_{\langle\rangle}(P),\\
o_{\langle\rangle}(\mathcal C_tP)&=\begin{cases}
   \bar\vartheta(t+\omega^{o_{\langle\rangle}(P)}) & \text{if we have $t\in\mathcal H_t(\emptyset)$ and $P$ is $t$-collapsing},\\
   o_{\langle\rangle}(P)                             & \text{otherwise},
\end{cases}
\end{align*}
for any $(u,\alpha)$-code $P$.
\end{theorem}
\begin{proof}
Let us first observe that the given characterization of $o_{\langle\rangle}(\mathcal C_tP)$ is a valid recursive clause in the sense of Remark~\ref{rmk:codes-methodology}: Since we must decide whether $P$ is $t$-collapsing, the value $o_{\langle\rangle}(\mathcal C_tP)$ does not only depend on $o_{\langle\rangle}(P)$ but also on $l_{\langle\rangle}(P)$, $h_0(P)$, $h_1(P)$ and $d(P)$. The point is that all these values are defined in part (i) of the recursion mentioned in Remark~\ref{rmk:codes-methodology}. To prove the theorem we must provide the remaining recursive clauses and show the corresponding induction steps. As in the proof of Lemma~\ref{lem:boundedness} we begin with the ``unintended" case, i.e.~we assume that we have $t\notin\mathcal H_t(\emptyset)$ or that $P$ is not $t$-collapsing. In this situation we stipulate that $\mathcal C_tP$ behaves like $P$, i.e.~we set
\begin{gather*}
d(\mathcal C_tP)=d(P),\qquad h_0(\mathcal C_tP)=h_0(P),\qquad h_1(\mathcal C_tP)=h_1(P),\\
r_{\langle\rangle}(\mathcal C_tP)=r_{\langle\rangle}(P),\qquad n(\mathcal C_tP,a)=n(P,a).
\end{gather*}
Then the local correctness conditions for $\mathcal C_tP$ follow from the same conditions for~$P$. In the rest of the proof we consider the ``intended case", i.e.~we assume that we have $t\in\mathcal H_t(\emptyset)$ and that $P$ is $t$-collapsing. In this situation we put
\begin{equation*}
d(\mathcal C_tP)=1,\qquad h_0(\mathcal C_tP)=t+\omega^{o_{\langle\rangle}(P)},\qquad h_1(\mathcal C_tP)=\emptyset.
\end{equation*}
The value $r_{\langle\rangle}(\mathcal C_tP)$ and the function $a\mapsto n(\mathcal C_tP,a)$ are defined by case distinction over the rule $r_{\langle\rangle}(P)$. We verify the local correctness conditions as we go along:

\emph{Case $r_{\langle\rangle}(P)=(\true,\varphi)$:} We set
\begin{equation*}
r_{\langle\rangle}(\mathcal C_tP)=(\true,\varphi),\qquad n(\mathcal C_tP,a)=P.
\end{equation*}
The only interesting condition is~(H1): We have $o_{\langle\rangle}(P)\in\mathcal H_P(\emptyset)$ by the same condition for~$P$. Since $P$ is $t$-collapsing we can use Proposition~\ref{prop:operators-collapse}(a) and Lemma~\ref{lem:operators-closure} to infer
\begin{equation*}
o_{\langle\rangle}(P)\in\mathcal H_{h_0(P)}(h_1(P))\subseteq\mathcal H_t(h_1(P))\subseteq\mathcal H_t(\emptyset)\subseteq\mathcal H_{t+\omega^{o_{\langle\rangle}(P)}}(\emptyset).
\end{equation*}
By Lemma~\ref{lem:operators-arithmetic} we obtain $t+\omega^{o_{\langle\rangle}(P)}\in\mathcal H_{t+\omega^{o_{\langle\rangle}(P)}}(\emptyset)$. Proposition~\ref{prop:operators-collapse}(b) yields
\begin{equation*}
o_{\langle\rangle}(\mathcal C_tP)=\bar\vartheta(t+\omega^{o_{\langle\rangle}(P)})\in\mathcal H_{t+\omega^{o_{\langle\rangle}(P)}}(\emptyset)=\mathcal H_{\mathcal C_tP}(\emptyset),
\end{equation*}
as required by condition (H1) for $\mathcal C_tP$.

\emph{Case $r_{\langle\rangle}(P)=(\bigwedge,\varphi)$:} In this case we set
\begin{equation*}
r_{\langle\rangle}(\mathcal C_tP)=(\bigwedge,\varphi),\qquad n(\mathcal C_tP,a)=\mathcal C_t n(P,a).
\end{equation*}
To see that $\mathcal C_t n(P,a)$ is evaluated according to the intended case we must show that $n(P,a)$ is $t$-collapsing, for any $a\in\iota(\mathcal C_tP)=\iota_\alpha(\varphi)$: Observe that $\varphi$ must be a $\Sigma(\alpha)$-formula, since condition (L) for $P$ ensures $\varphi\in l_{\langle\rangle}(P)$. In view of Lemma~\ref{lem:relativization-properties} we can infer that $l_{\langle\rangle}(n(P,a))\subseteq l_{\langle\rangle}(P)\cup\{\varphi_a\}$ consists of $\Sigma(\alpha)$-formulas, as required by condition~(i) of Definition~\ref{def:t-collapsing}. Condition (H3) for $P$ yields $h_0(n(P,a))\leq h_0(P)\leq t$. Crucially, Lemma~\ref{lem:instances-conjunctive-Sigma} and condition (H2) for $P$ ensure
\begin{equation*}
\suppl_\alpha(a)\lef\supp(\varphi)\subseteq\mathcal H_P(\emptyset).
\end{equation*}
In view of $\supp(\varphi)\subseteq\alpha\cong\varepsilon(S)^u_\alpha\cap\Omega$ we get $\suppl_\alpha(a)\subseteq\mathcal H_P(\emptyset)$ by Lemma~\ref{lem:operators-arithmetic}(d). From condition (H3) for $P$ and Lemma~\ref{lem:operators-closure} we can now deduce
\begin{equation*}
h_1(n(P,a))\subseteq\mathcal H_P(\suppl_\alpha(a))\subseteq\mathcal H_P(\emptyset)\subseteq\mathcal H_t(\emptyset),
\end{equation*}
as required by condition~(ii) of Definition~\ref{def:t-collapsing}. Invoking condition (C2) for $P$ we also get $d(n(P,a))\leq d(P)\leq 2$, completing the verification that $n(P,a)$ is $t$-collapsing. Let us now establish condition (L) for $\mathcal C_tP$: Even though $n(P,a)$ may not be a subterm of $P$ we can use condition (H1) for $n(P,a)$, since the latter is established in part~(i') of the induction mentioned in Remark~\ref{rmk:codes-methodology}. With the above we get
\begin{equation*}
o_{\langle\rangle}(n(P,a))\in\mathcal H_{n(P,a)}(\emptyset)\subseteq\mathcal H_t(\emptyset).
\end{equation*}
Also note that condition (L) for $P$ provides $o_{\langle\rangle}(n(P,a))<o_{\langle\rangle}(P)$. Using the fact that $n(P,a)$ is $t$-collapsing and Proposition~\ref{prop:operators-collapse}(d) we can infer
\begin{equation*}
o_{\langle\rangle}(n(\mathcal C_tP,a))=o_{\langle\rangle}(\mathcal C_t n(P,a))=\bar\vartheta(t+\omega^{o_{\langle\rangle}(n(P,a))})<\bar\vartheta(t+\omega^{o_{\langle\rangle}(P)})=o_{\langle\rangle}(\mathcal C_tP),
\end{equation*}
as required by condition (L) for $\mathcal C_tP$. The remaining verifications are similar.

\emph{Case $r_{\langle\rangle}(P)=(\bigvee,\varphi,b)$:} We set
\begin{equation*}
r_{\langle\rangle}(\mathcal C_tP)=(\bigvee,\varphi,b),\qquad n(\mathcal C_tP,a)=\mathcal C_t n(P,a).
\end{equation*}
The crucial observation is that the side condition $\suppl_\alpha(b)\lef o_{\langle\rangle}(\mathcal C_tP)$ of condition (L) is preserved: Using condition (H2) for $P$ and the fact that $P$ is $t$-collapsing we get
\begin{equation*}
\suppl_\alpha(b)\subseteq\mathcal H_P(\emptyset)\subseteq\mathcal H_t(\emptyset).
\end{equation*}
In view of $\suppl_\alpha(b)\subseteq\alpha\cong\varepsilon(S)^u_\alpha\cap\Omega$ and $t<t+\omega^{o_{\langle\rangle}(P)}$ Proposition~\ref{prop:operators-collapse}(c) yields
\begin{equation*}
\suppl_\alpha(b)\lef\bar\vartheta(t+\omega^{o_{\langle\rangle}(P)})=o_{\langle\rangle}(\mathcal C_tP),
\end{equation*}
as required. The other conditions are shown as in the previous case. In particular one should observe that $n(P,0)$ is $t$-collapsing, so that $n(\mathcal C_tP,0)=\mathcal C_t n(P,0)$ is evaluated according to the intended case.

\emph{Case $r_{\langle\rangle}(P)=(\cut,\varphi)$ with a disjunctive formula $\varphi$:} Invoking condition (C1) for $P$ and the fact that $P$ is $t$-collapsing we see
\begin{equation*}
\rk(\varphi)<d(P)\leq 2.
\end{equation*}
By the previous lemma it follows that $\varphi$ is a $\Sigma(\alpha)$-formula. As condition (L) for $P$ provides $l_{\langle\rangle}(n(P,1))\subseteq l_{\langle\rangle}(P),\varphi$ we can conclude that $n(P,1)$ is $t$-collapsing, as in the previous cases. The problem is that $l_{\langle\rangle}(n(P,0))\subseteq l_{\langle\rangle}(P),\neg\varphi$ may not consist of $\Sigma(\alpha)$-formulas. Before we can collapse $n(P,0)$ we must use Lemma~\ref{lem:boundedness} to restrict the $\Pi(\alpha)$-formula $\neg\varphi$ to a bounded formula $\neg\varphi^{\bar\vartheta(s)}$. To reapply a cut we must also restrict $\varphi$ to $\varphi^{\bar\vartheta(s)}$, but in this case after collapsing. This leads to the clauses
\begin{equation*}
r_{\langle\rangle}(\mathcal C_tP)=(\cut,\varphi^{\bar\vartheta(s)}),\qquad n(\mathcal C_tP,a)=\begin{cases}
\mathcal B_{\varphi,\bar\vartheta(s)}\mathcal C_t n(P,1) & \text{if $a=1$},\\
\mathcal C_s\mathcal B_{\neg\varphi,\bar\vartheta(s)} n(P,a) & \text{otherwise},
\end{cases}
\end{equation*}
where we set
\begin{equation*}
s=t+\omega^{o_{\langle\rangle}(n(P,1))}.
\end{equation*}
Note that $r_{\langle\rangle}(\mathcal C_tP)$ and $n(\mathcal C_tP,a)$ depend on $o_{\langle\rangle}(n(P,1))$. This is permitted even though $n(P,1)$ may not be a subterm of $P$, since the functions $r_{\langle\rangle}$ and $n$ are defined in part~(ii) of the recursion mentioned in Remark~\ref{rmk:codes-methodology}. Let us verify condition (L) for~$\mathcal C_tP$: Above we have observed that $n(P,1)$ is $t$-collapsing. Similarly to the previous cases we can infer
\begin{multline*}
o_{\langle\rangle}(n(\mathcal C_tP,1))=o_{\langle\rangle}(\mathcal B_{\varphi,\bar\vartheta(s)}\mathcal C_t n(P,1))=o_{\langle\rangle}(\mathcal C_t n(P,1))=\\
=\bar\vartheta(t+\omega^{o_{\langle\rangle}(n(P,1))})<\bar\vartheta(t+\omega^{o_{\langle\rangle}(P)})=o_{\langle\rangle}(\mathcal C_tP).
\end{multline*}
In particular we have $o_{\langle\rangle}(\mathcal C_t n(P,1))\leq\bar\vartheta(s)$, so that $\mathcal B_{\varphi,\bar\vartheta(s)}\mathcal C_t n(P,1)$ is evaluated according to the intended case. Using condition (L) for $P$ we get
\begin{multline*}
l_{\langle\rangle}(n(\mathcal C_tP,1))=l_{\langle\rangle}(\mathcal B_{\varphi,\bar\vartheta(s)}\mathcal C_t n(P,1))=(l_{\langle\rangle}(n(P,1))\backslash\{\varphi\})\cup\{\varphi^{\bar\vartheta(s)}\}\subseteq\\
\subseteq l_{\langle\rangle}(P)\cup\{\varphi^{\bar\vartheta(s)}\}=l_{\langle\rangle}(\mathcal C_tP)\cup\{\varphi^{\bar\vartheta(s)}\},
\end{multline*}
as condition (L) for $\mathcal C_tP$ demands. To establish the rest of condition (L) we must show that we have $s\in\mathcal H_s(\emptyset)$ and that $\mathcal B_{\neg\varphi,\bar\vartheta(s)}n(P,0)$ is $s$-collapsing: Since $\neg\varphi$ is a $\Pi(\alpha)$-formula we know that $\mathcal B_{\neg\varphi,\bar\vartheta(s)}n(P,0)$ is evaluated according to the intended case, independently of the ordinal height $o_{\langle\rangle}(n(P,0))$. We get
\begin{equation*}
l_{\langle\rangle}(\mathcal B_{\neg\varphi,\bar\vartheta(s)}n(P,0))=(l_{\langle\rangle}(n(P,0))\backslash\{\neg\varphi\})\cup\{\neg\varphi^{\bar\vartheta(s)}\}\subseteq l_{\langle\rangle}(P)\cup\{\neg\varphi^{\bar\vartheta(s)}\}.
\end{equation*}
Since the formula $\neg\varphi^{\bar\vartheta(s)}$ is bounded we learn that $l_{\langle\rangle}(\mathcal B_{\neg\varphi,\bar\vartheta(s)}n(P,0))$ consists of $\Sigma(\alpha)$-formulas. Condition (H3) for $P$ yields
\begin{equation*}
h_0(\mathcal B_{\neg\varphi,\bar\vartheta(s)}n(P,0))=h_0(n(P,0))\leq h_0(P)\leq t\leq s.
\end{equation*}
Using condition (H1) for $n(P,1)$ we get $o_{\langle\rangle}(n(P,1))\in\mathcal H_{n(P,1)}(\emptyset)\subseteq\mathcal H_t(\emptyset)$. Together with $t\in\mathcal H_t(\emptyset)$ this implies
\begin{equation*}
s\in\mathcal H_t(\emptyset)\subseteq\mathcal H_s(\emptyset)
\end{equation*}
and then
\begin{equation*}
h_1(\mathcal B_{\neg\varphi,\bar\vartheta(s)}n(P,0))=h_1(n(P,0))\cup\{\bar\vartheta(s)\}\subseteq\mathcal H_s(\emptyset).
\end{equation*}
Condition (C2) for $P$ yields
\begin{equation*}
d(\mathcal B_{\neg\varphi,\bar\vartheta(s)}n(P,0))=d(n(P,0))\leq d(P)\leq 2,
\end{equation*}
completing the proof that $\mathcal B_{\neg\varphi,\bar\vartheta(s)}n(P,0)$ is $s$-collapsing. Due to this fact we have
\begin{equation*}
o_{\langle\rangle}(n(\mathcal C_tP,0))=o_{\langle\rangle}(\mathcal C_s\mathcal B_{\neg\varphi,\bar\vartheta(s)} n(P,0))=\bar\vartheta(s+\omega^{o_{\langle\rangle}(n(P,0))}).
\end{equation*}
By condition (H1) for $n(P,0)$ we get $o_{\langle\rangle}(n(P,0))\in\mathcal H_{n(P,0)}(\emptyset)\subseteq\mathcal H_t(\emptyset)$ and then
\begin{equation*}
\bar\vartheta(s+\omega^{o_{\langle\rangle}(n(P,0))})\in\mathcal H_{s+\omega^{o_{\langle\rangle}(n(P,0))}}(\emptyset)\cap\Omega.
\end{equation*}
Condition (L) for $P$ provides $o_{\langle\rangle}(n(P,i))<o_{\langle\rangle}(P)$ for $i=0,1$. By Lemma~\ref{lem:ordinal-addition} we can conclude
\begin{equation*}
s+\omega^{o_{\langle\rangle}(n(P,0))}=t+(\omega^{o_{\langle\rangle}(n(P,1))}+\omega^{o_{\langle\rangle}(n(P,0))})<t+\omega^{o_{\langle\rangle}(P)}.
\end{equation*}
Now Proposition~\ref{prop:operators-collapse}(c) yields
\begin{equation*}
o_{\langle\rangle}(n(\mathcal C_tP,0))=\bar\vartheta(s+\omega^{o_{\langle\rangle}(n(P,0))})<\bar\vartheta(t+\omega^{o_{\langle\rangle}(P)})=o_{\langle\rangle}(\mathcal C_tP).
\end{equation*}
To complete the verification of condition (L) we observe
\begin{multline*}
l_{\langle\rangle}(n(\mathcal C_tP,0))=l_{\langle\rangle}(\mathcal C_s\mathcal B_{\neg\varphi,\bar\vartheta(s)} n(P,0))=l_{\langle\rangle}(\mathcal B_{\neg\varphi,\bar\vartheta(s)} n(P,0))\subseteq\\
\subseteq l_{\langle\rangle}(P)\cup\{\neg\varphi^{\bar\vartheta(s)}\}=l_{\langle\rangle}(\mathcal C_tP)\cup\{\neg\varphi^{\bar\vartheta(s)}\}.
\end{multline*}
Since the formula $\varphi^{\bar\vartheta(s)}$ is bounded we have
\begin{equation*}
\rk(\varphi^{\bar\vartheta(s)})=0<1=d(\mathcal C_tP),
\end{equation*}
as condition (C1) for $\mathcal C_tP$ demands. The remaining verifications are straightforward.

\emph{Case $r_{\langle\rangle}(P)=(\cut,\varphi)$ with a conjunctive formula $\varphi$:} Analogous to the previous case we set $s=t+\omega^{o_{\langle\rangle}(n(P,0))}$ and
\begin{equation*}
r_{\langle\rangle}(\mathcal C_tP)=(\cut,\varphi^{\bar\vartheta(s)}),\qquad n(\mathcal C_tP,a)=\begin{cases}
\mathcal B_{\neg\varphi,\bar\vartheta(s)}\mathcal C_t n(P,0) & \text{if $a=0$},\\
\mathcal C_s\mathcal B_{\varphi,\bar\vartheta(s)} n(P,a) & \text{otherwise}.
\end{cases}
\end{equation*}
Local correctness is verified as before.

\emph{Case $r_{\langle\rangle}(P)=(\refl,\exists_w\forall_{x\in b}\exists_{y\in w}\theta)$:} Using condition (L) for $P$ we see that the sequent $l_{\langle\rangle}(n(P,0))\subseteq l_{\langle\rangle}(P)\cup\{\forall_{x\in b}\exists_y\theta\}$ consists of $\Sigma(\alpha)$-formulas. As before we can deduce that $n(P,0)$ is $t$-collapsing. Let us set
\begin{equation*}
\gamma=\bar\vartheta(t+\omega^{o_{\langle\rangle}(n(P,0))})=o_{\langle\rangle}(\mathcal C_t n(P,0)).
\end{equation*}
Note that we cannot reapply the reflection rule to $\mathcal C_t n(P,0)$, since it requires ordinal height at least $\Omega$. Instead we invoke Lemma~\ref{lem:boundedness} to obtain
\begin{multline*}
l_{\langle\rangle}(\mathcal B_{\forall_{x\in b}\exists_y\theta,\gamma}\mathcal C_t n(P,0))=(l_{\langle\rangle}(n(P,0))\backslash\{\forall_{x\in b}\exists_y\theta\})\cup\{\forall_{x\in b}\exists_{y\in L^u_\gamma}\theta\}\subseteq\\
\subseteq l_{\langle\rangle}(P)\cup\{\forall_{x\in b}\exists_{y\in L^u_\gamma}\theta\}=l_{\langle\rangle}(\mathcal C_tP)\cup\{\forall_{x\in b}\exists_{y\in L^u_\gamma}\theta\}.
\end{multline*}
Using the existential witness $L^u_\gamma$ we can reintroduce the conclusion $\exists_w\forall_{x\in b}\exists_{y\in w}\theta$ of the reflection rule (recall that $\theta$ may contain $x$ and $y$ but not $w$). Officially, this idea is implemented by the clauses
\begin{equation*}
r_{\langle\rangle}(\mathcal C_tP)=(\bigvee,\exists_w\forall_{x\in b}\exists_{y\in w}\theta,L^u_\gamma),\qquad n(\mathcal C_tP,a)=\mathcal B_{\forall_{x\in b}\exists_y\theta,\gamma}\mathcal C_t n(P,a).
\end{equation*}
As in the case of a cut rule we get
\begin{equation*}
\gamma\in\mathcal H_{t+\omega^{o_{\langle\rangle}(n(P,0))}}(\emptyset)\cap\Omega.
\end{equation*}
In view of $t+\omega^{o_{\langle\rangle}(n(P,0))}<t+\omega^{o_{\langle\rangle}(P)}$ we can use Proposition~\ref{prop:operators-collapse}(c) to conclude
\begin{equation*}
\suppl_\alpha(L^u_\gamma)=\{\gamma\}\lef\bar\vartheta(t+\omega^{o_{\langle\rangle}(P)})=o_{\langle\rangle}(\mathcal C_tP),
\end{equation*}
as required by condition (L) for $\mathcal C_tP$. We can also infer
\begin{equation*}
\suppl_\alpha(L^u_\gamma)\subseteq\mathcal H_{t+\omega^{o_{\langle\rangle}(P)}}(\emptyset)=\mathcal H_{\mathcal C_tP}(\emptyset),
\end{equation*}
as condition (H2) for $\mathcal C_tP$ demands. The remaining verifications are straightforward. 

\emph{Case $r_{\langle\rangle}(P)=(\rep,b)$:} We set
\begin{equation*}
r_{\langle\rangle}(\mathcal C_tP)=(\rep,b),\qquad n(\mathcal C_tP,a)=\mathcal C_t n(P,a).
\end{equation*}
Crucially, condition (H2) for $P$ ensures $\suppl_\alpha(b)\subseteq\mathcal H_P(\emptyset)$. Using condition~(H3) for $P$ and the fact that $P$ is $t$-collapsing we can deduce
\begin{equation*}
h_1(n(P,b))\subseteq\mathcal H_P(\suppl_\alpha(b))\subseteq\mathcal H_P(\emptyset)\subseteq\mathcal H_t(\emptyset).
\end{equation*}
Based on this observation it is straightforward to show that $n(P,b)$ is $t$-collapsing. The local correctness conditions can now be verified as in the previous cases.
\end{proof}

Our ordinal analysis culminates in the following soundness result:

\begin{corollary}
Assume that the $(u,\alpha)$-code $P$ is $t$-controlled, for some $t\in\varepsilon(S)^u_\alpha$ with $t\in\mathcal H_t(\emptyset)$. Then we have $\mathbb L^u_\alpha\vDash\varphi$ for some formula $\varphi\in l_{\langle\rangle}(P)$.
\end{corollary}
\begin{proof}
Form the $(u,\alpha)$-code $\mathcal E^{d(P)}P=\mathcal E\cdots\mathcal EP$ with $d(P)$ occurrences of the function symbol $\mathcal E$. By Proposition~\ref{prop:cut-elimination} we have $d(\mathcal E^{d(P)}P)\leq 2$. Since $P$ is $t$-controlled we can infer that $\mathcal E^{d(P)}P$ is $t$-collapsing. Thus the previous theorem yields
\begin{equation*}
o_{\langle\rangle}(\mathcal C_t\mathcal E^{d(P)}P)=\bar\vartheta(t+\omega^{o_{\langle\rangle}(\mathcal E^{d(P)}P)})<\Omega.
\end{equation*}
According to Proposition~\ref{prop:codes-to-proofs} and Remark~\ref{rmk:codes-methodology} the $(u,\alpha)$-code $\mathcal C_t\mathcal E^{d(P)}P$ is interpreted as a $(u,\alpha)$-proof $[\mathcal C_t\mathcal E^{d(P)}P]$. In view of Definition~\ref{def:code-reconstruct-proof} it has ordinal height
\begin{equation*}
o_{\mathcal C_t\mathcal E^{d(P)}P}(\langle\rangle)=o_{\langle\rangle}(\bar n(\mathcal C_t\mathcal E^{d(P)}P,\langle\rangle))=o_{\langle\rangle}(\mathcal C_t\mathcal E^{d(P)}P)<\Omega.
\end{equation*}
Now Proposition~\ref{prop:soundness-below-Omega} yields $\mathbb L^u_\alpha\vDash\varphi$ for some formula
\begin{equation*}
\varphi\in l_{\mathcal C_t\mathcal E^{d(P)}P}(\langle\rangle)=l_{\langle\rangle}(\mathcal C_t\mathcal E^{d(P)}P)=l_{\langle\rangle}(P),
\end{equation*}
as desired.
\end{proof}

Putting things together we can prove the main result of our paper:

\begin{theorem}\label{thm:main-abstract}
 The following are equivalent over $\atrs$:
 \begin{enumerate}[label=(\roman*)]
  \item The principle of $\Pi^1_1$-comprehension.
  \item The statement that every set is an element of an admissible set.
  \item The abstract Bachmann-Howard principle, which states that every dilator has a well-founded Bachmann-Howard fixed point.
 \end{enumerate}
\end{theorem}
\begin{proof}
As pointed out in the introduction, the equivalence between (i) and~(ii) is known (see~\cite[Section~7]{jaeger-admissibles} and the additional verification in~\cite[Proposition~1.4.12]{freund-thesis}). In Theorem~\ref{thm:admissible-to-bhp} above we have shown that~(ii) implies~(iii). It remains to prove that (iii) implies~(ii): Given an arbitrary set $x$, we consider the transitive closure
\begin{equation*}
v=\operatorname{TC}(\{x,\omega\}).
\end{equation*}
In order to turn the height $o(v)=v\cap\ordi$ into a successor ordinal we set
\begin{equation*}
u=v\cup\{o(v)\}.
\end{equation*}
As $\atrs$ contains the axiom of countability we can enumerate $u=\{u_i\,|\,i\in\omega\}$. Thus $u$ satisfies Assumption~\ref{ass:properties-u}, upon which our construction of search trees was founded. According to Theorem~\ref{thm:admissible-set-dilator} it suffices to consider the following two cases: First assume that there is an admissible set $\mathbb A\supseteq u$. Then we have $x\in\mathbb A$, as required for claim~(ii) of the present theorem. Now assume that the search trees form a dilator $(S^u,\supps)$. In the rest of this proof we show that this contradicts claim~(iii). Theorem~\ref{thm:eps-dilator} tells us that $(\varepsilon(S)^u,\suppe)$ is a dilator as well. By~(iii) there is a well-order $X$ which is a Bachmann-Howard fixed point of $\varepsilon(S)^u$. Since $\atrs$ contains axiom beta we obtain an ordinal $\alpha\cong X$. Using the functoriality of $\varepsilon(S)^u$ and the naturality of $\suppe$ it is straightforward to check that $\alpha$ is a Bachmann-Howard fixed point of $\varepsilon(S)^u$ as well. Thus $\alpha$ satisfies Assumption~\ref{ass:bachmann-howard-collapse}, upon which we have based our ordinal analysis. By Proposition~\ref{prop:proof-from-search-tree} and Definition~\ref{def:proof-codes} the search tree $S^u_\alpha$ can be extended to a $(u,\alpha)$-proof $P^u_\alpha=(P^u_\alpha,l^u_\alpha,r^u_\alpha,o^u_\alpha)$, which is represented by the $(u,\alpha)$-code $P^u_\alpha\langle\rangle$. Together with Definition~\ref{def:operators-for-proofs} we have
\begin{equation*}
l_{\langle\rangle}(P^u_\alpha\langle\rangle)=l^u_\alpha(\langle\rangle)=\langle\rangle,\qquad h_0(P^u_\alpha\langle\rangle)=0,\qquad h_1(P^u_\alpha\langle\rangle)=\supps_\alpha(\langle\rangle)=\emptyset.
\end{equation*}
We can conclude that the $(u,\alpha)$-code $P^u_\alpha\langle\rangle$ is $0$-controlled. By the previous corollary there is a formula $\varphi\in l_{\langle\rangle}(P^u_\alpha\langle\rangle)$ with $\mathbb L^u_\alpha\vDash\varphi$. This, however, contradicts the fact that the sequent $l_{\langle\rangle}(P^u_\alpha\langle\rangle)=\langle\rangle$ is empty.
\end{proof}

As mentioned in the introduction, the equivalence between~(i) and an appropriate formalization of~(iii) also holds over the much weaker base theory~$\rca$. In order to show that this is the case, it suffices to establish that~(iii) implies arithmetical transfinite recursion. This is done in~\cite{freund-computable}. The same paper shows that any prae-dilator $T$ has a minimal Bachmann-Howard fixed point $\vartheta(T)$, which is computable relative to a representation of~$T$. The statement that $\vartheta(T)$ is well-founded for any dilator $T$ will be called the computable Bachmann-Howard principle. It is equivalent to its abstract counterpart, due to the minimality of $\vartheta(T)$. These improvements to Theorem~\ref{thm:main-abstract} are significant, because they show that $\Pi^1_1$-comprehension can be characterized by a computable transformation and a statement about the preservation of well-foundedness, over a base theory that does not itself introduce any non-computable sets.

\bibliographystyle{amsplain}
\bibliography{Bibliography_Freund}

\end{document}